\documentclass[a4paper, 12pt]{amsart}
\usepackage{amssymb, amsmath, amsthm, color, mathrsfs, mathtools}
\usepackage[colorlinks=true, breaklinks=true, linkcolor=black, citecolor=blue, urlcolor=red]{hyperref} 
\usepackage[english]{babel}
\usepackage{tikz-cd}
\usepackage{tikz}
\usetikzlibrary{shapes}

\usepackage{enumitem}

\numberwithin{equation}{section}
\newtheorem{theorem}{Theorem}[section]

\newtheorem{lemma}[theorem]{Lemma}

\newtheorem{proposition}[theorem]{Proposition}

\newtheorem{observation}[theorem]{Observation}

\newtheorem{example}[theorem]{Example}

\newtheorem{mainthm}{Theorem}

\newtheorem{maincor}[mainthm]{Corollary}

\theoremstyle{definition} 
\newtheorem{definition}[theorem]{Definition}
\newtheorem{remark}[theorem]{Remark}

\newcommand{\act}{\curvearrowright}

\DeclareMathOperator{\ab}{ab}

\newcommand{\al}{\alpha}
\newcommand{\cA}{\mathcal{A}}

\newcommand{\cC}{\mathcal C}
\newcommand{\fC}{\mathfrak C}

\DeclareMathOperator{\Cat}{Cat}

\DeclareMathOperator{\Cone}{Cone}

\DeclareMathOperator{\End}{End}
\newcommand{\cF}{\mathcal F}

\DeclareMathOperator{\Frac}{Frac}

\DeclareMathOperator{\FM}{FM}
\DeclareMathOperator{\FS}{FS}
\DeclareMathOperator{\Forest}{Forest}
\DeclareMathOperator{\ForestOre}{Forest_{Ore}}

\DeclareMathOperator{\Gr}{Gr}

\DeclareMathOperator{\Groupoid}{Groupoid}

\newcommand{\Ga}{\Gamma}

\DeclareMathOperator{\Germ}{Germ}
\newcommand{\cH}{\mathcal H}

\DeclareMathOperator{\Homeo}{Homeo}

\DeclareMathOperator{\id}{id}
\newcommand{\into}{\hookrightarrow}
\newcommand{\la}{\langle}
\newcommand{\cL}{\mathcal{L}}
\DeclareMathOperator{\Leaf}{Leaf}

\newcommand{\cM}{\mathcal{M}}

\DeclareMathOperator{\Mon}{Mon}

\newcommand{\N}{\mathbf{N}}
\newcommand{\cNUT}{\mathcal{NUT}}

\DeclareMathOperator{\ob}{ob}
\newcommand{\onto}{\twoheadrightarrow}
\newcommand{\ot}{\otimes}
\newcommand{\ov}{\overline}

\DeclareMathOperator{\prune}{prune}
\newcommand{\Q}{\mathbf{Q}}
\newcommand{\cQ}{\mathcal Q}
\newcommand{\R}{\mathbf{R}}

\newcommand{\ra}{\rangle}

\DeclareMathOperator{\Root}{Root}

\DeclareMathOperator{\Set}{Set}

\DeclareMathOperator{\supp}{supp}

\newcommand{\cT}{\mathcal T}
\newcommand{\ti}{\tilde}

\newcommand{\cU}{\mathcal U}
\newcommand{\cUF}{\mathcal{UF}}
\newcommand{\cUT}{\mathcal{UT}}
\DeclareMathOperator{\Ver}{Ver}

\newcommand{\Z}{\mathbf{Z}}

\newcommand{\ClearySkeinA}{
	\begin{tikzpicture}
		\draw[thick] (-1.25,0) -- (-2,0.75);
		\draw[thick] (-1.25,0) -- (0.25,1.5);
		\draw[thick] (-0.5,0.75) -- (-1.25,1.5);
		\draw[black,fill=white,thick] (-1.25,0) circle [radius=0.25] node {$a$};
		\draw[black,fill=white,thick] (-0.5,0.75) circle [radius=0.25] node {$a$};
		\draw[fill=black] (-2,0.75) circle [radius=0.06];
		\draw[fill=black] (-1.25,1.5) circle [radius=0.06];
		\draw[fill=black] (0.25,1.5) circle [radius=0.06];
	\end{tikzpicture}
}
\newcommand{\ClearySkeinB}{
	\begin{tikzpicture}
		\draw[thick] (-1.25,0) -- (-2,0.75);
		\draw[thick] (-1.25,0) -- (-0.5,0.75);
		\draw[thick] (-2,0.75) -- (-2.75,1.5);
		\draw[thick] (-2,0.75) -- (-1.25,1.5);
		\draw[black,fill=white,thick] (-1.25,0) circle [radius=0.25] node {$b$};
		\draw[black,fill=white,thick] (-2,0.75) circle [radius=0.25] node {$b$};
		\draw[fill=black] (-2.75,1.5) circle [radius=0.06];
		\draw[fill=black] (-1.25,1.5) circle [radius=0.06];
		\draw[fill=black] (-0.5,0.75) circle [radius=0.06];
	\end{tikzpicture}
}
\newcommand{\GroupElementA}{
	\begin{tikzpicture}
		\draw[thick] (0,0) -- (0.7,0.75);
		\draw[thick] (0,0) -- (-0.7,0.75);
		\draw[black,fill=white,thick] (0,0) circle [radius=0.25] node {$a$};
		\draw[fill=black] (-0.7,0.75) circle [radius=0.06];
		\draw[fill=black] (0.7,0.75) circle [radius=0.06];
		\draw[thick] (-0.7,0.75) -- (0.7,1.75);
		\draw[thick] (0.7,0.75) -- (-0.7,1.75);
		\draw[thick] (0.7,1.75) -- (0,2.5);
		\draw[thick] (-0.7,1.75) -- (0,2.5);
		\draw[black,fill=white,thick] (0,2.5) circle [radius=0.25] node {$b$};
		\draw[fill=black] (0.7,1.75) circle [radius=0.06];
		\draw[fill=black] (-0.7,1.75) circle [radius=0.06];
	\end{tikzpicture}
}
\newcommand{\GroupElementB}{
	\begin{tikzpicture}
		\draw[thick] (0,0) -- (0.6,0.65);
		\draw[thick] (0,0) -- (-1.2,1.25);
		\draw[thick] (0.6,0.65) -- (1.2,1.25);
		\draw[thick] (0.6,0.65) -- (0,1.25);
		\draw[black,fill=white,thick] (0,0) circle [radius=0.25] node {$a$};
		\draw[black,fill=white,thick] (0.6,0.65) circle [radius=0.25] node {$a$};
		\draw[fill=black] (1.2,1.25) circle [radius=0.06];
		\draw[fill=black] (0,1.25) circle [radius=0.06];
		\draw[fill=black] (-1.2,1.25) circle [radius=0.06];
		\draw[thick] (-1.2,1.25) -- (1.2,2.25);
		\draw[thick] (0,1.25) -- (-1.2,2.25);
		\draw[thick] (1.2,1.25) -- (0,2.25);
		\draw[thick] (0,3.5) -- (-0.6,2.85);
		\draw[thick] (0,3.5) -- (1.2,2.25);
		\draw[thick] (-0.6,2.85) -- (0,2.25);
		\draw[thick] (-0.6,2.85) -- (-1.2,2.25);
		\draw[black,fill=white,thick] (-0.6,2.9) circle [radius=0.25] node {$a$};
		\draw[black,fill=white,thick] (0,3.5) circle [radius=0.25] node {$b$};
		\draw[fill=black] (1.2,2.25) circle [radius=0.06];
		\draw[fill=black] (0,2.25) circle [radius=0.06];
		\draw[fill=black] (-1.2,2.25) circle [radius=0.06];
	\end{tikzpicture}
}
\newcommand{\GroupElementC}{
	\begin{tikzpicture}
		\draw[thick] (0,0) -- (1.2,1.25);
		\draw[thick] (0,0) -- (-0.6,0.65);
		\draw[thick] (-0.6,0.65) -- (-1.2,1.25);
		\draw[thick] (-0.6,0.65) -- (0,1.25);
		\draw[black,fill=white,thick] (0,0) circle [radius=0.25] node {$b$};
		\draw[black,fill=white,thick] (-0.6,0.65) circle [radius=0.25] node {$b$};
		\draw[fill=black] (1.2,1.25) circle [radius=0.06];
		\draw[fill=black] (0,1.25) circle [radius=0.06];
		\draw[fill=black] (-1.2,1.25) circle [radius=0.06];
		\draw[thick] (-1.2,1.25) -- (1.2,2.25);
		\draw[thick] (0,1.25) -- (-1.2,2.25);
		\draw[thick] (1.2,1.25) -- (0,2.25);
		\draw[thick] (0,3.5) -- (-0.6,2.85);
		\draw[thick] (0,3.5) -- (1.2,2.25);
		\draw[thick] (-0.6,2.85) -- (0,2.25);
		\draw[thick] (-0.6,2.85) -- (-1.2,2.25);
		\draw[black,fill=white,thick] (-0.6,2.85) circle [radius=0.25] node {$a$};
		\draw[black,fill=white,thick] (0,3.5) circle [radius=0.25] node {$b$};
		\draw[fill=black] (1.2,2.25) circle [radius=0.06];
		\draw[fill=black] (0,2.25) circle [radius=0.06];
		\draw[fill=black] (-1.2,2.25) circle [radius=0.06];
	\end{tikzpicture}
}
\newcommand{\InfiniteBinaryTree}{
	\begin{tikzpicture}
		\draw[fill=black] (0,0) circle [radius=0.06];
		\node[below] at (0,0) {$\varepsilon$};
		\draw[thick] (0,0)--(-2,2);
		\draw[thick] (0,0)--(2,2);
		\draw[fill=black] (-2,2) circle [radius=0.06];
		\node[left] at (-2,2) {$l$};
		\draw[thick] (-2,2)--(-3,3);
		\draw[thick] (-2,2)--(-1,3);
		\draw[fill=black] (2,2) circle [radius=0.06];
		\node[right] at (2,2) {$r$};
		\draw[thick] (2,2)--(3,3);
		\draw[thick] (2,2)--(1,3);
		\draw[fill=black] (-3,3) circle [radius=0.06];
		\node[left] at (-3,3) {$ll$};
		\node[above] at (-3,3) {$\vdots$};
		\draw[fill=black] (-1,3) circle [radius=0.06];
		\node[right] at (-1,3) {$lr$};
		\node[above] at (-1,3) {$\vdots$};
		\draw[fill=black] (1,3) circle [radius=0.06];
		\node[left] at (1,3) {$rl$};
		\node[above] at (1,3) {$\vdots$};
		\draw[fill=black] (3,3) circle [radius=0.06];
		\node[right] at (3,3) {$rr$};
		\node[above] at (3,3) {$\vdots$};
	\end{tikzpicture}
}
\newcommand{\ColouredTreeExample}{
	\begin{tikzpicture}
		\draw[thick] (-1.25,0) -- (-2,0.75);
		\draw[thick] (-1.25,0) -- (0.25,1.5);
		\draw[thick] (-2,0.75) -- (-2.75,1.5);
		\draw[thick] (-2,0.75) -- (-1.25,1.5);
		\draw[black,fill=white,thick] (-1.25,0) circle [radius=0.25] node {$a$};
		\draw[black,fill=white,thick] (-2,0.75) circle [radius=0.25] node {$b$};
		\draw[fill=black] (-2.75,1.5) circle [radius=0.06];
		\draw[fill=black] (-1.25,1.5) circle [radius=0.06];
		\draw[fill=black] (0.25,1.5) circle [radius=0.06];
	\end{tikzpicture}   
}
\newcommand{\ColouredForestExample}{
	\begin{tikzpicture}
		\draw[thick] (0,0) -- (0,0.75);
		\draw[fill=black] (0,0) circle [radius=0.06];
		\draw[fill=black] (0,0.75) circle [radius=0.06];
		\draw[thick] (1.5,0) -- (1.5,0.75);
		\draw[fill=black] (1.5,0) circle [radius=0.06];
		\draw[fill=black] (1.5,0.75) circle [radius=0.06];
		\draw[thick] (3,0) -- (2.5,0.75);
		\draw[thick] (3,0) -- (3.5,0.75);
		\draw[black,fill=white,thick] (3,0) circle [radius=0.25] node {$b$};
		\draw[fill=black] (2.5,0.75) circle [radius=0.06];
		\draw[fill=black] (3.5,0.75) circle [radius=0.06];
	\end{tikzpicture}  
}
\newcommand{\CompositionA}{
	\begin{tikzpicture}
		\draw[thick] (-1.25,0) -- (-2,0.75);
		\draw[thick] (-1.25,0) -- (0.25,1.5);
		\draw[thick] (-2,0.75) -- (-2.75,1.5);
		\draw[thick] (-2,0.75) -- (-1.25,1.5);
		\draw[black,fill=white,thick] (-1.25,0) circle [radius=0.25] node {$a$};
		\draw[black,fill=white,thick] (-2,0.75) circle [radius=0.25] node {$b$};
		\draw[fill=black] (-2.75,1.5) circle [radius=0.06];
		\draw[fill=black] (-1.25,1.5) circle [radius=0.06];
		\draw[fill=black] (0.25,1.5) circle [radius=0.06];
		\draw[thick,dotted] (-2.75,2.25) -- (-2.75,1.5);
		\draw[thick,dotted] (-1.25,2.25) -- (-1.25,1.5);
		\draw[thick,dotted] (0.25,2.25) -- (0.25,1.5);
		\draw[fill=black] (-2.75,2.25) circle [radius=0.06];
		\draw[fill=black] (-1.25,2.25) circle [radius=0.06];
		\draw[thick] (0.25,2.25) -- (-0.25,3);
		\draw[thick] (0.25,2.25) -- (0.75,3);
		\draw[thick] (-2.75,2.25) -- (-2.75,3);
		\draw[thick] (-1.25,2.25) -- (-1.25,3);
		\draw[black,fill=white,thick] (0.25,2.25) circle [radius=0.25] node {$b$};
		\draw[fill=black] (-0.25,3) circle [radius=0.06];
		\draw[fill=black] (0.75,3) circle [radius=0.06];
		\draw[fill=black] (-2.75,3) circle [radius=0.06];
		\draw[fill=black] (-1.25,3) circle [radius=0.06];
	\end{tikzpicture}  
}
\newcommand{\CompositionB}{
	\begin{tikzpicture}
		\draw[thick] (-1.25,0) -- (-2,0.75);
		\draw[thick] (-1.25,0) -- (0.25,1.5);
		\draw[thick] (-2,0.75) -- (-2.75,1.5);
		\draw[thick] (-2.25,1) -- (-1.75,1.5);
		\draw[thick] (-0.25,1) -- (-0.75,1.5);
		\draw[black,fill=white,thick] (-1.25,0) circle [radius=0.25] node {$a$};
		\draw[black,fill=white,thick] (-2.25,1) circle [radius=0.25] node {$b$};
		\draw[black,fill=white,thick] (-0.25,1) circle [radius=0.25] node {$b$};
		\draw[fill=black] (-2.75,1.5) circle [radius=0.06];
		\draw[fill=black] (-1.75,1.5) circle [radius=0.06];
		\draw[fill=black] (-0.75,1.5) circle [radius=0.06];
		\draw[fill=black] (0.25,1.5) circle [radius=0.06];
	\end{tikzpicture}  
}
\newcommand{\MonochromaticCaret}{
	\begin{tikzpicture}
		\draw[thick] (0,0) -- (-0.5,0.75);
		\draw[thick] (0,0) -- (0.5,0.75);
		\draw[black,fill=black,thick] (0,0) circle [radius=0.06];
		\draw[fill=black] (-0.5,0.75) circle [radius=0.06];
		\draw[fill=black] (0.5,0.75) circle [radius=0.06];
	\end{tikzpicture}  
}
\newcommand{\MonochromaticVineA}{
	\begin{tikzpicture}
		\draw[thick] (0,0) -- (-0.5,0.75);
		\draw[thick] (0,0) -- (0.5,0.75);
		\draw[thick] (-0.5,0.75) -- (-1,1.5);
		\draw[thick] (-0.5,0.75) -- (0,1.5);
		\draw[black,fill=black,thick] (0,0) circle [radius=0.06];
		\draw[fill=black] (-0.5,0.75) circle [radius=0.06];
		\draw[fill=black] (-1,1.5) circle [radius=0.06];
		\draw[fill=black] (0,1.5) circle [radius=0.06];
		\draw[fill=black] (0.5,0.75) circle [radius=0.06];
	\end{tikzpicture}  
}
\newcommand{\MonochromaticVineB}{
	\begin{tikzpicture}
		\draw[thick] (0,0) -- (-0.5,0.75);
		\draw[thick] (0,0) -- (0.5,0.75);
		\draw[thick] (-0.5,0.75) -- (-1,1.5);
		\draw[thick] (-0.5,0.75) -- (0,1.5);
		\draw[thick] (-1,1.5) -- (-1.5,2.25);
		\draw[thick] (-1,1.5) -- (-0.5,2.25);
		\draw[black,fill=black,thick] (0,0) circle [radius=0.06];
		\draw[fill=black] (-0.5,0.75) circle [radius=0.06];
		\draw[fill=black] (-1,1.5) circle [radius=0.06];
		\draw[fill=black] (-1.5,2.25) circle [radius=0.06];
		\draw[fill=black] (-0.5,2.25) circle [radius=0.06];
		\draw[fill=black] (0,1.5) circle [radius=0.06];
		\draw[fill=black] (0.5,0.75) circle [radius=0.06];
	\end{tikzpicture}  
}
\newcommand{\MonochromaticCompleteA}{
	\begin{tikzpicture}
		\draw[thick] (0,0) -- (-0.75,0.75);
		\draw[thick] (0,0) -- (0.75,0.75);
		\draw[thick] (-0.75,0.75) -- (-1.25,1.5);
		\draw[thick] (-0.75,0.75) -- (-0.25,1.5);
		\draw[thick] (0.75,0.75) -- (1.25,1.5);
		\draw[thick] (0.75,0.75) -- (0.25,1.5);
		\draw[black,fill=black,thick] (0,0) circle [radius=0.06];
		\draw[fill=black] (0.75,0.75) circle [radius=0.06];
		\draw[fill=black] (-0.75,0.75) circle [radius=0.06];
		\draw[fill=black] (-1.25,1.5) circle [radius=0.06];
		\draw[fill=black] (-0.25,1.5) circle [radius=0.06];
		\draw[fill=black] (0.25,1.5) circle [radius=0.06];
		\draw[fill=black] (1.25,1.5) circle [radius=0.06];
	\end{tikzpicture}  
}
\newcommand{\MonochromaticCompleteB}{
	\begin{tikzpicture}
		\draw[thick] (0,0) -- (-1,0.75);
		\draw[thick] (0,0) -- (1,0.75);
		\draw[thick] (-1,0.75) -- (-1.5,1.5);
		\draw[thick] (-1,0.75) -- (-0.5,1.5);
		\draw[thick] (1,0.75) -- (1.5,1.5);
		\draw[thick] (1,0.75) -- (0.5,1.5);
		\draw[thick] (-1.5,1.5) -- (-1.75,2.25);
		\draw[thick] (-1.5,1.5) -- (-1.25,2.25);
		\draw[thick] (-0.5,1.5) -- (-0.75,2.25);
		\draw[thick] (-0.5,1.5) -- (-0.25,2.25);
		\draw[thick] (1.5,1.5) -- (1.75,2.25);
		\draw[thick] (1.5,1.5) -- (1.25,2.25);
		\draw[thick] (0.5,1.5) -- (0.75,2.25);
		\draw[thick] (0.5,1.5) -- (0.25,2.25);
		\draw[black,fill=black,thick] (0,0) circle [radius=0.06];
		\draw[fill=black] (1,0.75) circle [radius=0.06];
		\draw[fill=black] (-1,0.75) circle [radius=0.06];
		\draw[fill=black] (-1.5,1.5) circle [radius=0.06];
		\draw[fill=black] (-0.5,1.5) circle [radius=0.06];
		\draw[fill=black] (0.5,1.5) circle [radius=0.06];
		\draw[fill=black] (1.5,1.5) circle [radius=0.06];
		\draw[fill=black] (-1.75,2.25) circle [radius=0.06];
		\draw[fill=black] (-1.25,2.25) circle [radius=0.06];
		\draw[fill=black] (-0.75,2.25) circle [radius=0.06];
		\draw[fill=black] (-0.25,2.25) circle [radius=0.06];
		\draw[fill=black] (1.75,2.25) circle [radius=0.06];
		\draw[fill=black] (1.25,2.25) circle [radius=0.06];
		\draw[fill=black] (0.75,2.25) circle [radius=0.06];
		\draw[fill=black] (0.25,2.25) circle [radius=0.06];
	\end{tikzpicture}  
}
\newcommand{\QuasiVine}{
	\begin{tikzpicture}
		\draw[thick] (0,0) -- (-1,1.5);
		\draw[thick] (0,0) -- (1,1.5);
		\draw[thick] (0.5,0.75) -- (0,1.5);
		\draw[thick] (-1,1.5) -- (0,3);
		\draw[thick] (-1,1.5) -- (-2,3);
		\draw[thick] (-0.5,2.25) -- (-1,3);
		\draw[black,fill=black,thick] (0,0) circle [radius=0.06];
		\draw[fill=black] (0.5,0.75) circle [radius=0.06];
		\draw[fill=black] (-0.5,2.25) circle [radius=0.06];
		\draw[fill=black] (-1,1.5) circle [radius=0.06];
		\draw[fill=black] (0,1.5) circle [radius=0.06];
		\draw[fill=black] (1,1.5) circle [radius=0.06];
		\draw[fill=black] (-2,3) circle [radius=0.06];
		\draw[fill=black] (-1,3) circle [radius=0.06];
		\draw[fill=black] (0,3) circle [radius=0.06];
	\end{tikzpicture}  
}
\newcommand{\Cell}{
	\begin{tikzpicture}
		\draw[thick] (0,0) -- (-1,1.5);
		\draw[thick] (0,0) -- (1,1.5);
		\draw[thick] (0.5,0.75) -- (0,1.5);
		\draw[black,fill=black,thick] (0,0) circle [radius=0.06];
		\draw[fill=black] (0.5,0.75) circle [radius=0.06];
		\draw[fill=black] (-1,1.5) circle [radius=0.06];
		\draw[fill=black] (0,1.5) circle [radius=0.06];
		\draw[fill=black] (1,1.5) circle [radius=0.06];
	\end{tikzpicture}
}
\newcommand{\Skeleton}{
	\begin{tikzpicture}
		\draw[thick] (0,0) -- (-1,1.5);
		\draw[thick] (0,0) -- (0,1.5);
		\draw[thick] (0,0) -- (1,1.5);
		\draw[thick] (-1,1.5) -- (0,3);
		\draw[thick] (-1,1.5) -- (-1,3);
		\draw[thick] (-1,1.5) -- (-2,3);
		\draw[black,fill=black,thick] (0,0) circle [radius=0.06];
		\draw[black,fill=black,thick] (1,1.5) circle [radius=0.06];
		\draw[fill=black] (-1,1.5) circle [radius=0.06];
		\draw[fill=black] (0,1.5) circle [radius=0.06];
		\draw[fill=black] (0,3) circle [radius=0.06];
		\draw[fill=black] (-1,3) circle [radius=0.06];
		\draw[fill=black] (-2,3) circle [radius=0.06];
	\end{tikzpicture}  
}
\newcommand{\TernaryA}{
	\begin{tikzpicture}
		\draw[thick] (-1.25,0) -- (-2,0.75);
		\draw[thick] (-1.25,0) -- (-0.5,0.75);
		\draw[thick] (-2,0.75) -- (-2.75,1.5);
		\draw[thick] (-2,0.75) -- (-1.25,1.5);
		\draw[black,fill=white,thick] (-1.25,0) circle [radius=0.25] node {$b$};
		\draw[black,fill=white,thick] (-2,0.75) circle [radius=0.25] node {$a$};
		\draw[fill=black] (-2.75,1.5) circle [radius=0.06];
		\draw[fill=black] (-1.25,1.5) circle [radius=0.06];
		\draw[fill=black] (-0.5,0.75) circle [radius=0.06];
	\end{tikzpicture}
}
\newcommand{\TernaryB}{
	\begin{tikzpicture}
		\draw[thick] (-1.25,0) -- (-2,0.75);
		\draw[thick] (-1.25,0) -- (0.25,1.5);
		\draw[thick] (-0.5,0.75) -- (-1.25,1.5);
		\draw[black,fill=white,thick] (-1.25,0) circle [radius=0.25] node {$a$};
		\draw[black,fill=white,thick] (-0.5,0.75) circle [radius=0.25] node {$b$};
		\draw[fill=black] (-2,0.75) circle [radius=0.06];
		\draw[fill=black] (-1.25,1.5) circle [radius=0.06];
		\draw[fill=black] (0.25,1.5) circle [radius=0.06];
	\end{tikzpicture}
}
\newcommand{\PermutationA}{
	\begin{tikzpicture}
		\draw[thick] (1,1) -- (0,2);
		\draw[thick] (0,1) -- (2,2);
		\draw[thick] (2,1) -- (1,2);
		\draw[thick] (0,0) -- (1,1);
		\draw[thick] (1,0) -- (0,1);
		\draw[thick] (2,0) -- (2,1);
		\draw[fill=black] (0,0) circle [radius=0.06];
		\draw[fill=black] (1,0) circle [radius=0.06];
		\draw[fill=black] (2,0) circle [radius=0.06];
		\draw[fill=black] (0,1) circle [radius=0.06];
		\draw[fill=black] (1,1) circle [radius=0.06];
		\draw[fill=black] (2,1) circle [radius=0.06];
		\draw[fill=black] (0,2) circle [radius=0.06];
		\draw[fill=black] (1,2) circle [radius=0.06];
		\draw[fill=black] (2,2) circle [radius=0.06];
	\end{tikzpicture}
}
\newcommand{\PermutationB}{
	\begin{tikzpicture}
		\draw[thick] (0,0) -- (0,1);
		\draw[thick] (1,0) -- (2,1);
		\draw[thick] (2,0) -- (1,1);
		\draw[fill=black] (0,0) circle [radius=0.06];
		\draw[fill=black] (1,0) circle [radius=0.06];
		\draw[fill=black] (2,0) circle [radius=0.06];
		\draw[fill=black] (0,1) circle [radius=0.06];
		\draw[fill=black] (1,1) circle [radius=0.06];
		\draw[fill=black] (2,1) circle [radius=0.06];
	\end{tikzpicture}
}
\newcommand{\SymmetricColouredForest}{
	\begin{tikzpicture}
		\node at (-3,0) {$((I,I,Y_b),(23))$};
		\node at (-1,0) {$=$};
		\draw[thick] (0,0) -- (0,0.75);
		\draw[fill=black] (0,0) circle [radius=0.06];
		\draw[fill=black] (0,0.75) circle [radius=0.06];
		\draw[thick] (1.5,0) -- (1.5,0.75);
		\draw[fill=black] (1.5,0) circle [radius=0.06];
		\draw[fill=black] (1.5,0.75) circle [radius=0.06];
		\draw[thick] (3,0) -- (2.5,0.75);
		\draw[thick] (3,0) -- (3.5,0.75);
		\draw[black,fill=white,thick] (3,0) circle [radius=0.25] node {$b$};
		\draw[fill=black] (2.5,0.75) circle [radius=0.06];
		\draw[fill=black] (3.5,0.75) circle [radius=0.06];
		\draw[thick] (0,0.75) -- (0,1.75);
		\draw[thick] (2.5,0.75) -- (1.5,1.75);
		\draw[thick] (1.5,0.75) -- (2.5,1.75);
		\draw[thick] (3.5,0.75) -- (3.5,1.75);
		\draw[fill=black] (0,1.75) circle [radius=0.06];
		\draw[fill=black] (2.5,1.75) circle [radius=0.06];
		\draw[fill=black] (1.5,1.75) circle [radius=0.06];
		\draw[fill=black] (3.5,1.75) circle [radius=0.06];
	\end{tikzpicture}  
}
\newcommand{\ZappaSzep}{
	\begin{tikzpicture}
		\node at (-3,0) {$(123)\circ(Y_a,I,I)$};
		\node at (-1,0) {$=$};
		\node at (2.75,0) {$=$};
		\node at (7,0) {$=$};
		\node at (9,0) {$(I,Y_a,I)\circ(1234)$.};
		\draw[thick] (1,0) -- (0,1);
		\draw[thick] (0,0) -- (2,1);
		\draw[thick] (2,0) -- (1,1);
		\draw[fill=black] (0,0) circle [radius=0.06];
		\draw[fill=black] (1,0) circle [radius=0.06];
		\draw[fill=black] (2,0) circle [radius=0.06];
		\draw[thick] (0,1) -- (-0.5,1.75);
		\draw[thick] (0,1) -- (0.5,1.75);
		\draw[thick] (1,1) -- (1,1.75);
		\draw[thick] (2,1) -- (2,1.75);
		\draw[fill=black] (-0.5,1.75) circle [radius=0.06];
		\draw[fill=black] (0.5,1.75) circle [radius=0.06];
		\draw[fill=black] (1,1.75) circle [radius=0.06];
		\draw[fill=black] (2,1.75) circle [radius=0.06];
		\draw[black,fill=white,thick] (0,1) circle [radius=0.25] node {$a$};
		\draw[fill=black] (1,1) circle [radius=0.06];
		\draw[fill=black] (2,1) circle [radius=0.06];
		\draw[thick] (4.5,0.75) -- (3.5,1.75);
		\draw[thick] (5.5,0.75) -- (4.5,1.75);
		\draw[thick] (6.5,0.75) -- (5.5,1.75);
		\draw[thick] (3.5,0.75) -- (6.5,1.75);
		\draw[fill=black] (3.5,1.75) circle [radius=0.06];
		\draw[fill=black] (4.5,1.75) circle [radius=0.06];
		\draw[fill=black] (5.5,1.75) circle [radius=0.06];
		\draw[fill=black] (6.5,1.75) circle [radius=0.06];
		\draw[thick] (3.5,0) -- (3.5,0.75);
		\draw[fill=black] (3.5,0) circle [radius=0.06];
		\draw[fill=black] (3.5,0.75) circle [radius=0.06];
		\draw[thick] (5,0) -- (4.5,0.75);
		\draw[thick] (5,0) -- (5.5,0.75);
		\draw[fill=black] (4.5,0.75) circle [radius=0.06];
		\draw[fill=black] (5.5,0.75) circle [radius=0.06];
		\draw[black,fill=white,thick] (5,0) circle [radius=0.25] node {$a$};
		\draw[thick] (6.5,0) -- (6.5,0.75);
		\draw[fill=black] (6.5,0) circle [radius=0.06];
		\draw[fill=black] (6.5,0.75) circle [radius=0.06];
	\end{tikzpicture}
}
\newcommand{\PointedCaretA}{
	\begin{tikzpicture}
		\draw[thick] (0,0) -- (-0.5,0.75);
		\draw[thick] (0,0) -- (0.5,0.75);
		\draw[black,fill=white,thick] (0,0) circle [radius=0.25] node {$a$};
		\draw[fill=black] (0.5,0.75) circle [radius=0.06];
		\draw[thick,fill=white] (-0.5,0.75) circle [radius=0.12];
	\end{tikzpicture}  
}
\newcommand{\PointedCaretB}{
	\begin{tikzpicture}
		\draw[thick] (0,0) -- (-0.5,0.75);
		\draw[thick] (0,0) -- (0.5,0.75);
		\draw[black,fill=white,thick] (0,0) circle [radius=0.25] node {$b$};
		\draw[fill=black] (-0.5,0.75) circle [radius=0.06];
		\draw[thick,fill=white] (0.5,0.75) circle [radius=0.12];
	\end{tikzpicture}  
}
\newcommand{\PointedTreeCompositionA}{
	\begin{tikzpicture}
		\draw[thick,dotted] (-0.5,0.75) -- (-0.5,1.5);
		\draw[thick] (0,0) -- (-0.5,0.75);
		\draw[thick] (0,0) -- (0.5,0.75);
		\draw[thick] (-0.5,1.5) -- (-1,2.25);
		\draw[thick] (-0.5,1.5) -- (0,2.25);
		\draw[black,fill=white,thick] (0,0) circle [radius=0.25] node {$a$};
		\draw[black,fill=white,thick] (-0.5,1.5) circle [radius=0.25] node {$b$};
		\draw[fill=black] (-1,2.25) circle [radius=0.06];
		\draw[fill=black] (0.5,0.75) circle [radius=0.06];
		\draw[thick,fill=white] (0,2.25) circle [radius=0.12];
		\draw[thick,fill=white] (-0.5,0.75) circle [radius=0.12];
	\end{tikzpicture}  
}
\newcommand{\PointedTreeCompositionB}{
	\begin{tikzpicture}
		\draw[thick] (0,0) -- (-0.5,0.75);
		\draw[thick] (0,0) -- (0.5,0.75);
		\draw[thick] (-0.5,0.75) -- (-1,1.5);
		\draw[thick] (-0.5,0.75) -- (0,1.5);
		\draw[black,fill=white,thick] (0,0) circle [radius=0.25] node {$a$};
		\draw[black,fill=white,thick] (-0.5,0.75) circle [radius=0.25] node {$b$};
		\draw[fill=black] (-1,1.5) circle [radius=0.06];
		\draw[fill=black] (0.5,0.75) circle [radius=0.06];
		\draw[thick,fill=white] (0,1.5) circle [radius=0.12];
	\end{tikzpicture}  
}
\newcommand{\PointedTreeActionA}{
	\begin{tikzpicture}
		\draw[thick,dotted] (-0.5,0.75) -- (-0.5,1.5);
		\draw[thick] (0,0) -- (-0.5,0.75);
		\draw[thick] (0,0) -- (0.5,0.75);
		\draw[thick] (-0.5,1.5) -- (-1,2.25);
		\draw[thick] (-0.5,1.5) -- (0,2.25);
		\draw[black,fill=white,thick] (0,0) circle [radius=0.25] node {$a$};
		\draw[black,fill=white,thick] (-0.5,1.5) circle [radius=0.25] node {$b$};
		\draw[fill=black] (-1,2.25) circle [radius=0.06];
		\draw[fill=black] (0.5,0.75) circle [radius=0.06];
		\draw[fill=black] (0,2.25) circle [radius=0.06];
		\draw[thick,fill=white] (-0.5,0.75) circle [radius=0.12];
	\end{tikzpicture}  
}
\newcommand{\PointedTreeActionB}{
	\begin{tikzpicture}
		\draw[thick] (0,0) -- (-0.5,0.75);
		\draw[thick] (0,0) -- (0.5,0.75);
		\draw[thick] (-0.5,0.75) -- (-1,1.5);
		\draw[thick] (-0.5,0.75) -- (0,1.5);
		\draw[black,fill=white,thick] (0,0) circle [radius=0.25] node {$a$};
		\draw[black,fill=white,thick] (-0.5,0.75) circle [radius=0.25] node {$b$};
		\draw[fill=black] (-1,1.5) circle [radius=0.06];
		\draw[fill=black] (0.5,0.75) circle [radius=0.06];
		\draw[fill=black] (0,1.5) circle [radius=0.06];
	\end{tikzpicture}  
}
\newcommand{\NarrowA}{
	\begin{tikzpicture}
		\draw[thick] (0,0) -- (-0.5,0.75);
		\draw[thick] (0,0) -- (0.5,0.75);
		\draw[thick] (-0.5,0.75) -- (-1,1.5);
		\draw[thick] (-0.5,0.75) -- (0,1.5);
		\draw[thick] (-1,1.5) -- (-1.5,2.25);
		\draw[thick] (-1,1.5) -- (-0.5,2.25);
		\draw[black,fill=white,thick] (0,0) circle [radius=0.25] node {$b$};
		\draw[black,fill=white,thick] (-0.5,0.75) circle [radius=0.25] node {$b$};
		\draw[black,fill=white,thick] (-1,1.5) circle [radius=0.25] node {$a$};
		\draw[fill=black] (-0.5,2.25) circle [radius=0.06];
		\draw[fill=black] (0.5,0.75) circle [radius=0.06];
		\draw[fill=black] (0,1.5) circle [radius=0.06];
		\draw[thick,fill=white] (-1.5,2.25) circle [radius=0.12];
	\end{tikzpicture}  
}
\newcommand{\NarrowB}{
	\begin{tikzpicture}
		\draw[thick] (0,0) -- (-0.75,0.75);
		\draw[thick] (0,0) -- (0.75,0.75);
		\draw[thick] (-0.75,0.75) -- (-1.25,1.5);
		\draw[thick] (-0.75,0.75) -- (-0.25,1.5);
		\draw[thick] (0.75,0.75) -- (1.25,1.5);
		\draw[thick] (0.75,0.75) -- (0.25,1.5);
		\draw[black,fill=white,thick] (0,0) circle [radius=0.25] node {$a$};
		\draw[black,fill=white,thick] (-0.75,0.75) circle [radius=0.25] node {$a$};
		\draw[black,fill=white,thick] (0.75,0.75) circle [radius=0.25] node {$a$};
		\draw[fill=black] (1.25,1.5) circle [radius=0.06];
		\draw[fill=black] (0.25,1.5) circle [radius=0.06];
		\draw[fill=black] (-0.25,1.5) circle [radius=0.06];
		\draw[thick,fill=white] (-1.25,1.5) circle [radius=0.12];
	\end{tikzpicture}  
}
\newcommand{\FSimpleA}{
	\begin{tikzpicture}
		\draw[thick] (0,0) -- (-1,1);
		\draw[thick] (0,0) -- (0.5,0.5);
		\draw[thick] (-0.5,0.5) -- (0,1);
		\draw[black,fill=black,thick] (0,0) circle [radius=0.06];
		\draw[black,fill=black,thick] (-0.5,0.5) circle [radius=0.06];
		\draw[black,fill=black,thick] (-1,1) circle [radius=0.06];
		\draw[black,fill=black,thick] (0,1) circle [radius=0.06];
		\draw[black,fill=black,thick] (0.5,0.5) circle [radius=0.06];
	\end{tikzpicture}
}
\newcommand{\FSimpleB}{
	\begin{tikzpicture}
		\draw[thick] (0,0) -- (-1,1);
		\draw[thick] (0,0) -- (0.5,0.5);
		\draw[thick] (-0.5,0.5) -- (0,1);
		\draw[thick] (0,1) -- (-1,2);
		\draw[thick] (0,1) -- (0.5,1.5);
		\draw[thick] (-0.5,1.5) -- (0,2);
		\draw[black,fill=black,thick] (0,0) circle [radius=0.06];
		\draw[black,fill=black,thick] (-0.5,0.5) circle [radius=0.06];
		\draw[black,fill=black,thick] (0,1) circle [radius=0.06];
		\draw[black,fill=black,thick] (-0.5,1.5) circle [radius=0.06];
		\draw[black,fill=black,thick] (-1,1) circle [radius=0.06];
		\draw[black,fill=black,thick] (-1,2) circle [radius=0.06];
		\draw[black,fill=black,thick] (0.5,1.5) circle [radius=0.06];
		\draw[black,fill=black,thick] (0,2) circle [radius=0.06];
		\draw[black,fill=black,thick] (0.5,0.5) circle [radius=0.06];
	\end{tikzpicture}
}
\newcommand{\FSimpleC}{
	\begin{tikzpicture}
		\draw[thick] (0,0) -- (-1,1);
		\draw[thick] (0,0) -- (1,1);
		\draw[thick] (-0.25,0.25) -- (0.5,1);
		\draw[thick] (0,0.5) -- (-0.5,1);
		\draw[thick] (-0.25,0.75) -- (0,1);
		\draw[black,fill=black,thick] (0,0) circle [radius=0.03];
		\draw[black,fill=black,thick] (-0.25,0.25) circle [radius=0.03];
		\draw[black,fill=black,thick] (0,0.5) circle [radius=0.03];
		\draw[black,fill=black,thick] (-0.25,0.75) circle [radius=0.03];
		\draw[black,fill=black,thick] (-1,1) circle [radius=0.03];
		\draw[black,fill=black,thick] (-0.5,1) circle [radius=0.03];
		\draw[black,fill=black,thick] (0,1) circle [radius=0.03];
		\draw[black,fill=black,thick] (0.5,1) circle [radius=0.03];
		\draw[black,fill=black,thick] (1,1) circle [radius=0.03];
		\draw[thick] (-1,1) -- (0,2);
		\draw[thick] (1,1) -- (0,2);
		\draw[thick] (-0.5,1.5) -- (0,1);
		\draw[thick] (-0.75,1.25) -- (-0.5,1);
		\draw[thick] (0.75,1.25) -- (0.5,1);
		\draw[black,fill=black,thick] (0,2) circle [radius=0.03];
		\draw[black,fill=black,thick] (-0.5,1.5) circle [radius=0.03];
		\draw[black,fill=black,thick] (-0.75,1.25) circle [radius=0.03];
		\draw[black,fill=black,thick] (0.75,1.25) circle [radius=0.03];
		\draw[thick] (0,2) -- (-0.5,2.5);
		\draw[thick] (0,2) -- (0.5,2.5);
		\draw[thick] (-0.25,2.25) -- (0,2.5);
		\draw[black,fill=black,thick] (-0.25,2.25) circle [radius=0.03];
		\draw[black,fill=black,thick] (-0.5,2.5) circle [radius=0.03];
		\draw[black,fill=black,thick] (0.5,2.5) circle [radius=0.03];
		\draw[black,fill=black,thick] (0,2.5) circle [radius=0.03];
		\draw[thick] (-0.5,2.5) -- (-0.5,3.5);
		\draw[thick] (0,2.5) -- (0,3.5);
		\draw[thick] (0.5,2.5) -- (0.5,3.5);
		\node[diamond,draw,black,fill=white,thick,scale=1.25] at (0,3) {};
		\node at (0,3) {$l$};
		\draw[thick] (-0.5,3.5) -- (0,4);
		\draw[thick] (0.5,3.5) -- (0,4);
		\draw[thick] (0,3.5) -- (-0.25,3.75);
		\draw[black,fill=black,thick] (-0.5,3.5) circle [radius=0.03];
		\draw[black,fill=black,thick] (0.5,3.5) circle [radius=0.03];
		\draw[black,fill=black,thick] (0,3.5) circle [radius=0.03];
		\draw[black,fill=black,thick] (-0.25,3.75) circle [radius=0.03];
		\draw[black,fill=black,thick] (0,4) circle [radius=0.03];
		\draw[thick] (0,4) -- (-1,5);
		\draw[thick] (0,4) -- (1,5);
		\draw[thick] (-0.5,4.5) -- (0,5);
		\draw[thick] (0.75,4.75) -- (0.5,5);
		\draw[thick] (-0.75,4.75) -- (-0.5,5);
		\draw[black,fill=black,thick] (-0.5,4.5) circle [radius=0.03];
		\draw[black,fill=black,thick] (0.75,4.75) circle [radius=0.03];
		\draw[black,fill=black,thick] (-0.75,4.75) circle [radius=0.03];
		\draw[black,fill=black,thick] (-1,5) circle [radius=0.03];
		\draw[black,fill=black,thick] (-0.5,5) circle [radius=0.03];
		\draw[black,fill=black,thick] (0,5) circle [radius=0.03];
		\draw[black,fill=black,thick] (0.5,5) circle [radius=0.03];
		\draw[black,fill=black,thick] (1,5) circle [radius=0.03];
		\draw[thick] (-1,5) -- (0,6);
		\draw[thick] (1,5) -- (0,6);
		\draw[thick] (0.5,5) -- (-0.25,5.75);
		\draw[thick] (-0.5,5) -- (0,5.5);
		\draw[thick] (0,5) -- (-0.25,5.25);
		\draw[black,fill=black,thick] (0,6) circle [radius=0.03];
		\draw[black,fill=black,thick] (-0.25,5.75) circle [radius=0.03];
		\draw[black,fill=black,thick] (0,5.5) circle [radius=0.03];
		\draw[black,fill=black,thick] (-0.25,5.25) circle [radius=0.03];
	\end{tikzpicture}
}
\newcommand{\FSimpleD}{
	\begin{tikzpicture}
		\draw[thick] (0,0) -- (-1,1);
		\draw[thick] (0,0) -- (1,1);
		\draw[thick] (-0.25,0.25) -- (0.5,1);
		\draw[thick] (0,0.5) -- (-0.5,1);
		\draw[thick] (-0.25,0.75) -- (0,1);
		\draw[black,fill=black,thick] (0,0) circle [radius=0.03];
		\draw[black,fill=black,thick] (-0.25,0.25) circle [radius=0.03];
		\draw[black,fill=black,thick] (0,0.5) circle [radius=0.03];
		\draw[black,fill=black,thick] (-0.25,0.75) circle [radius=0.03];
		\draw[black,fill=black,thick] (-1,1) circle [radius=0.03];
		\draw[black,fill=black,thick] (-0.5,1) circle [radius=0.03];
		\draw[black,fill=black,thick] (0,1) circle [radius=0.03];
		\draw[black,fill=black,thick] (0.5,1) circle [radius=0.03];
		\draw[black,fill=black,thick] (1,1) circle [radius=0.03];
		\draw[thick] (-1,1) -- (-0.75,1.25);
		\draw[thick] (-0.5,1) -- (-0.75,1.25);
		\draw[thick] (1,1) -- (0.75,1.25);
		\draw[thick] (0.5,1) -- (0.75,1.25);
		\draw[thick] (-1,2) -- (-0.75,1.75);
		\draw[thick] (-0.5,2) -- (-0.75,1.75);
		\draw[thick] (1,2) -- (0.75,1.75);
		\draw[thick] (0.5,2) -- (0.75,1.75);
		\draw[thick] (-0.75,1.25) -- (-0.75,1.75);
		\draw[thick] (0,1) -- (0,2);
		\draw[thick] (0.75,1.25) -- (0.75,1.75);
		\draw[black,fill=black,thick] (-0.75,1.25) circle [radius=0.03];
		\draw[black,fill=black,thick] (0.75,1.25) circle [radius=0.03];
		\draw[black,fill=black,thick] (-0.75,1.75) circle [radius=0.03];
		\draw[black,fill=black,thick] (0.75,1.75) circle [radius=0.03];
		\node[diamond,draw,black,fill=white,thick,scale=1.25] at (0,1.5) {};
		\node at (0,1.5) {$l$};
		\draw[thick] (-1,2) -- (0,3);
		\draw[thick] (1,2) -- (0,3);
		\draw[thick] (0.5,2) -- (-0.25,2.75);
		\draw[thick] (-0.5,2) -- (0,2.5);
		\draw[thick] (0,2) -- (-0.25,2.25);
		\draw[black,fill=black,thick] (0,3) circle [radius=0.03];
		\draw[black,fill=black,thick] (-0.25,2.75) circle [radius=0.03];
		\draw[black,fill=black,thick] (0,2.5) circle [radius=0.03];
		\draw[black,fill=black,thick] (-0.25,2.25) circle [radius=0.03];
		\draw[black,fill=black,thick] (-1,2) circle [radius=0.03];
		\draw[black,fill=black,thick] (-0.5,2) circle [radius=0.03];
		\draw[black,fill=black,thick] (0,2) circle [radius=0.03];
		\draw[black,fill=black,thick] (0.5,2) circle [radius=0.03];
		\draw[black,fill=black,thick] (1,2) circle [radius=0.03];
	\end{tikzpicture}
}
\newcommand{\FSimpleE}{
	\begin{tikzpicture}
		\draw[thick] (0,0) -- (-1,1);
		\draw[thick] (0,0) -- (1,1);
		\draw[thick] (-0.25,0.25) -- (0.5,1);
		\draw[thick] (0,0.5) -- (-0.5,1);
		\draw[thick] (-0.25,0.75) -- (0,1);
		\draw[black,fill=black,thick] (0,0) circle [radius=0.03];
		\draw[black,fill=black,thick] (-0.25,0.25) circle [radius=0.03];
		\draw[black,fill=black,thick] (0,0.5) circle [radius=0.03];
		\draw[black,fill=black,thick] (-0.25,0.75) circle [radius=0.03];
		\draw[black,fill=black,thick] (-1,1) circle [radius=0.03];
		\draw[black,fill=black,thick] (-0.5,1) circle [radius=0.03];
		\draw[black,fill=black,thick] (0.5,1) circle [radius=0.03];
		\draw[black,fill=black,thick] (1,1) circle [radius=0.03];
		\draw[black,fill=black,thick] (0,1) circle [radius=0.03];
		\draw[thick] (-1,1) -- (-1,2);
		\draw[thick] (-0.5,1) -- (-0.5,2);
		\draw[thick] (0,1) -- (0,2);
		\draw[thick] (0.5,1) -- (0.5,2);
		\draw[thick] (1,1) -- (1,2);
		\node[diamond,draw,black,fill=white,thick,scale=1.25] at (0,1.5) {};
		\node at (0,1.5) {$l$};
		\draw[thick] (0,3) -- (-1,2);
		\draw[thick] (0,3) -- (1,2);
		\draw[thick] (-0.25,2.75) -- (0.5,2);
		\draw[thick] (0,2.5) -- (-0.5,2);
		\draw[thick] (-0.25,2.25) -- (0,2);
		\draw[black,fill=black,thick] (0,3) circle [radius=0.03];
		\draw[black,fill=black,thick] (-0.25,2.75) circle [radius=0.03];
		\draw[black,fill=black,thick] (0,2.5) circle [radius=0.03];
		\draw[black,fill=black,thick] (-0.25,2.25) circle [radius=0.03];
		\draw[black,fill=black,thick] (-1,2) circle [radius=0.03];
		\draw[black,fill=black,thick] (-0.5,2) circle [radius=0.03];
		\draw[black,fill=black,thick] (0,2) circle [radius=0.03];
		\draw[black,fill=black,thick] (0.5,2) circle [radius=0.03];
		\draw[black,fill=black,thick] (1,2) circle [radius=0.03];
	\end{tikzpicture}
}
\newcommand{\VineDecompositionA}{
	\begin{tikzpicture}
		\draw[thick] (0,0) -- (-0.75,0.75);
		\draw[thick] (0,0) -- (0.75,0.75);
		\draw[thick] (-0.75,1.5) -- (-1.25,2.25);
		\draw[thick] (-0.75,1.5) -- (-0.25,2.25);
		\draw[thick] (0.75,0.75) -- (1.25,1.5);
		\draw[thick] (0.75,0.75) -- (0.25,1.5);
		\draw[thick] (-0.75,0.75) -- (-0.75,1.5);
		\draw[thick] (0.25,1.5) -- (0.25,2.5);
		\draw[thick] (0.25,2.5) -- (-0.25,3.25);
		\draw[thick] (0.25,2.5) -- (0.75,3.25);
		\draw[black,fill=black,thick] (0,0) circle [radius=0.06];
		\draw[fill=black] (0.75,0.75) circle [radius=0.06];
		\draw[fill=black] (-0.75,0.75) circle [radius=0.06];
		\draw[fill=black] (-1.25,2.25) circle [radius=0.06];
		\draw[fill=black] (-0.25,2.25) circle [radius=0.06];
		\draw[fill=black] (0.25,1.5) circle [radius=0.06];
		\draw[fill=black] (1.25,1.5) circle [radius=0.06];
		\draw[fill=black] (-0.25,3.25) circle [radius=0.06];
		\draw[fill=black] (0.75,3.25) circle [radius=0.06];
		\draw[black,fill=white,thick] (0,0) circle [radius=0.25] node {$a$};
		\draw[black,fill=white,thick] (-0.75,1.5) circle [radius=0.25] node {$c$};
		\draw[black,fill=white,thick] (0.75,0.75) circle [radius=0.25] node {$b$};
		\draw[black,fill=white,thick] (0.25,2.5) circle [radius=0.25] node {$a$};
	\end{tikzpicture}  
}
\newcommand{\VineDecompositionB}{
	\begin{tikzpicture}
		\draw[thick] (0,0) -- (-0.75,0.75);
		\draw[thick] (0,0) -- (0.75,0.75);
		\draw[thick] (-0.75,0.75) -- (-1.25,1.5);
		\draw[thick] (-0.75,0.75) -- (-0.25,1.5);
		\draw[thick] (0.75,1.5) -- (1.25,2.25);
		\draw[thick] (0.75,1.5) -- (0.25,2.25);
		\draw[thick] (0.75,0.75) -- (0.75,1.5);
		\draw[thick] (0.25,2.25) -- (-0.25,3);
		\draw[thick] (0.25,2.25) -- (0.75,3);
		\draw[black,fill=black,thick] (0,0) circle [radius=0.06];
		\draw[fill=black] (0.75,0.75) circle [radius=0.06];
		\draw[fill=black] (-0.75,0.75) circle [radius=0.06];
		\draw[fill=black] (-1.25,1.5) circle [radius=0.06];
		\draw[fill=black] (-0.25,1.5) circle [radius=0.06];
		\draw[fill=black] (1.25,2.25) circle [radius=0.06];
		\draw[fill=black] (-0.25,3) circle [radius=0.06];
		\draw[fill=black] (0.75,3) circle [radius=0.06];
		\draw[black,fill=white,thick] (0,0) circle [radius=0.25] node {$a$};
		\draw[black,fill=white,thick] (-0.75,0.75) circle [radius=0.25] node {$c$};
		\draw[black,fill=white,thick] (0.75,1.5) circle [radius=0.25] node {$b$};
		\draw[black,fill=white,thick] (0.25,2.25) circle [radius=0.25] node {$a$};
	\end{tikzpicture}  
}
\newcommand{\FaithfulSkeinA}{
	\begin{tikzpicture}
		\draw[thick] (-1.25,0) -- (-2,0.75);
		\draw[thick] (-1.25,0) -- (-0.5,0.75);
		\draw[thick] (-0.5,0.75) -- (-1.25,1.5);
		\draw[thick] (-0.5,0.75) -- (0.25,1.5);
		\draw[thick] (-1.25,1.5) -- (-2,2.25);
		\draw[thick] (-1.25,1.5) -- (-0.5,2.25);
		\draw[black,fill=white,thick] (-1.25,0) circle [radius=0.25] node {$a$};
		\draw[black,fill=white,thick] (-0.5,0.75) circle [radius=0.25] node {$a$};
		\draw[black,fill=white,thick] (-1.25,1.5) circle [radius=0.25] node {$a$};
		\draw[fill=black] (-2,0.75) circle [radius=0.06];
		\draw[fill=black] (0.25,1.5) circle [radius=0.06];
		\draw[fill=black] (-2,2.25) circle [radius=0.06];
		\draw[fill=black] (-0.5,2.25) circle [radius=0.06];
	\end{tikzpicture}
}
\newcommand{\FaithfulSkeinB}{
	\begin{tikzpicture}
		\draw[thick] (-1.25,0) -- (-2,0.75);
		\draw[thick] (-1.25,0) -- (-0.5,0.75);
		\draw[thick] (-2,0.75) -- (-2.75,1.5);
		\draw[thick] (-2,0.75) -- (-1.25,1.5);
		\draw[thick] (-2.75,1.5) -- (-3.5,2.25);
		\draw[thick] (-2.75,1.5) -- (-2,2.25);
		\draw[black,fill=white,thick] (-1.25,0) circle [radius=0.25] node {$b$};
		\draw[black,fill=white,thick] (-2,0.75) circle [radius=0.25] node {$b$};
		\draw[black,fill=white,thick] (-2.75,1.5) circle [radius=0.25] node {$b$};
		\draw[fill=black] (-2,2.25) circle [radius=0.06];
		\draw[fill=black] (-3.5,2.25) circle [radius=0.06];
		\draw[fill=black] (-1.25,1.5) circle [radius=0.06];
		\draw[fill=black] (-0.5,0.75) circle [radius=0.06];
	\end{tikzpicture}
}
\newcommand{\SimpleFaithfulA}{
	\begin{tikzpicture}
		\draw[thick] (0,0) -- (-1,1);
		\draw[thick] (0,0) -- (1,1);
		\draw[thick] (0.5,0.5) -- (0,1);
		\draw[thick] (-0.75,0.75) -- (-0.5,1);
		\draw[thick] (0.75,0.75) -- (0.5,1);
		\draw[black,fill=black,thick] (0,0) circle [radius=0.03];
		\draw[black,fill=black,thick] (0.5,0.5) circle [radius=0.03];
		\draw[black,fill=black,thick] (-0.75,0.75) circle [radius=0.03];
		\draw[black,fill=black,thick] (0.75,0.75) circle [radius=0.03];
		\draw[black,fill=black,thick] (-1,1) circle [radius=0.03];
		\draw[black,fill=black,thick] (-0.5,1) circle [radius=0.03];
		\draw[black,fill=black,thick] (0,1) circle [radius=0.03];
		\draw[black,fill=black,thick] (0.5,1) circle [radius=0.03];
		\draw[black,fill=black,thick] (1,1) circle [radius=0.03];
		\draw[thick] (-1,1) -- (0,2);
		\draw[thick] (1,1) -- (0,2);
		\draw[thick] (-0.5,1.5) -- (0,1);
		\draw[thick] (-0.75,1.25) -- (-0.5,1);
		\draw[thick] (0.75,1.25) -- (0.5,1);
		\draw[black,fill=black,thick] (0,2) circle [radius=0.03];
		\draw[black,fill=black,thick] (-0.5,1.5) circle [radius=0.03];
		\draw[black,fill=black,thick] (-0.75,1.25) circle [radius=0.03];
		\draw[black,fill=black,thick] (0.75,1.25) circle [radius=0.03];
		\draw[thick] (0,2) -- (-1,3);
		\draw[thick] (0,2) -- (1,3);
		\draw[thick] (-0.75,2.75) -- (-0.5,3);
		\draw[thick] (0.75,2.75) -- (0.5,3);
		\draw[black,fill=black,thick] (-0.75,2.75) circle [radius=0.03];
		\draw[black,fill=black,thick] (0.75,2.75) circle [radius=0.03];
		\draw[black,fill=black,thick] (-1,3) circle [radius=0.03];
		\draw[black,fill=black,thick] (-0.5,3) circle [radius=0.03];
		\draw[black,fill=black,thick] (0.5,3) circle [radius=0.03];
		\draw[black,fill=black,thick] (1,3) circle [radius=0.03];
		\draw[thick] (-1,3) -- (-1,4);
		\draw[thick] (-0.5,3) -- (-0.5,4);
		\draw[thick] (0.5,3) -- (0.5,4);
		\draw[thick] (1,3) -- (1,4);
		\node[diamond,draw,black,fill=white,thick,scale=1.25] at (-0.5,3.5) {};
		\node at (-0.5,3.5) {$g$};
		\draw[thick] (-1,4) -- (0,5);
		\draw[thick] (1,4) -- (0,5);
		\draw[thick] (-0.5,4) -- (-0.75,4.25);
		\draw[thick] (0.5,4) -- (0.75,4.25);
		\draw[black,fill=black,thick] (-1,4) circle [radius=0.03];
		\draw[black,fill=black,thick] (-0.5,4) circle [radius=0.03];
		\draw[black,fill=black,thick] (0.5,4) circle [radius=0.03];
		\draw[black,fill=black,thick] (1,4) circle [radius=0.03];
		\draw[black,fill=black,thick] (-0.75,4.25) circle [radius=0.03];
		\draw[black,fill=black,thick] (0.75,4.25) circle [radius=0.03];
		\draw[black,fill=black,thick] (0,5) circle [radius=0.03];
		\draw[thick] (0,5) -- (-1,6);
		\draw[thick] (0,5) -- (1,6);
		\draw[thick] (-0.5,5.5) -- (0,6);
		\draw[thick] (0.75,5.75) -- (0.5,6);
		\draw[thick] (-0.75,5.75) -- (-0.5,6);
		\draw[black,fill=black,thick] (-0.5,5.5) circle [radius=0.03];
		\draw[black,fill=black,thick] (0.75,5.75) circle [radius=0.03];
		\draw[black,fill=black,thick] (-0.75,5.75) circle [radius=0.03];
		\draw[black,fill=black,thick] (-1,6) circle [radius=0.03];
		\draw[black,fill=black,thick] (-0.5,6) circle [radius=0.03];
		\draw[black,fill=black,thick] (0,6) circle [radius=0.03];
		\draw[black,fill=black,thick] (0.5,6) circle [radius=0.03];
		\draw[black,fill=black,thick] (1,6) circle [radius=0.03];
		\draw[thick] (-1,6) -- (0,7);
		\draw[thick] (1,6) -- (0,7);
		\draw[thick] (0,6) -- (0.5,6.5);
		\draw[thick] (0.5,6) -- (0.75,6.25);
		\draw[thick] (-0.5,6) -- (-0.75,6.25);
		\draw[black,fill=black,thick] (0.5,6.5) circle [radius=0.03];
		\draw[black,fill=black,thick] (0,7) circle [radius=0.03];
		\draw[black,fill=black,thick] (0.75,6.25) circle [radius=0.03];
		\draw[black,fill=black,thick] (-0.75,6.25) circle [radius=0.03];
	\end{tikzpicture}
}
\newcommand{\SimpleFaithfulB}{
	\begin{tikzpicture}
		\draw[thick] (0,0) -- (-1,1);
		\draw[thick] (0,0) -- (1,1);
		\draw[thick] (-0.75,0.75) -- (-0.5,1);
		\draw[thick] (0.75,0.75) -- (0.5,1);
		\draw[thick] (0.5,0.5) -- (0,1);
		\draw[black,fill=black,thick] (0,0) circle [radius=0.03];
		\draw[black,fill=black,thick] (-0.75,0.75) circle [radius=0.03];
		\draw[black,fill=black,thick] (0.75,0.75) circle [radius=0.03];
		\draw[black,fill=black,thick] (0.5,0.5) circle [radius=0.03];
		\draw[black,fill=black,thick] (-1,1) circle [radius=0.03];
		\draw[black,fill=black,thick] (-0.5,1) circle [radius=0.03];
		\draw[black,fill=black,thick] (0,1) circle [radius=0.03];
		\draw[black,fill=black,thick] (0.5,1) circle [radius=0.03];
		\draw[black,fill=black,thick] (1,1) circle [radius=0.03];
		\draw[thick] (-1,1) -- (-0.75,1.25);
		\draw[thick] (-0.5,1) -- (-0.75,1.25);
		\draw[thick] (-0.75,1.25) -- (-0.75,1.75);
		\draw[thick] (-0.75,1.75) -- (-1,2);
		\draw[thick] (-0.75,1.75) -- (-0.5,2);
		\draw[thick] (0,1) -- (0,2);
		\draw[thick] (0.5,1) -- (0.5,2);
		\draw[thick] (1,1) -- (1,2);
		\node[diamond,draw,black,fill=white,thick,scale=1.25] at (0,1.5) {};
		\node at (0,1.5) {$g$};
		\draw[black,fill=black,thick] (-0.75,1.25) circle [radius=0.03];
		\draw[black,fill=black,thick] (-0.75,1.75) circle [radius=0.03];
		\draw[thick] (0,3) -- (-1,2);
		\draw[thick] (0,3) -- (1,2);
		\draw[thick] (-0.5,2) -- (-0.75,2.25);
		\draw[thick] (0,2) -- (0.5,2.5);
		\draw[thick] (0.5,2) -- (0.75,2.25);
		\draw[black,fill=black,thick] (-0.75,2.25) circle [radius=0.03];
		\draw[black,fill=black,thick] (0.75,2.25) circle [radius=0.03];
		\draw[black,fill=black,thick] (0.5,2.5) circle [radius=0.03];
		\draw[black,fill=black,thick] (0,3) circle [radius=0.03];
		\draw[black,fill=black,thick] (-1,2) circle [radius=0.03];
		\draw[black,fill=black,thick] (-0.5,2) circle [radius=0.03];
		\draw[black,fill=black,thick] (0,2) circle [radius=0.03];
		\draw[black,fill=black,thick] (0.5,2) circle [radius=0.03];
		\draw[black,fill=black,thick] (1,2) circle [radius=0.03];
	\end{tikzpicture}
}
\newcommand{\SimpleFaithfulC}{
	\begin{tikzpicture}
		\draw[thick] (0,0) -- (-1,1);
		\draw[thick] (0,0) -- (1,1);
		\draw[thick] (-0.75,0.75) -- (-0.5,1);
		\draw[thick] (0.75,0.75) -- (0.5,1);
		\draw[black,fill=black,thick] (0,0) circle [radius=0.03];
		\draw[black,fill=black,thick] (-0.75,0.75) circle [radius=0.03];
		\draw[black,fill=black,thick] (0.75,0.75) circle [radius=0.03];
		\draw[black,fill=black,thick] (-1,1) circle [radius=0.03];
		\draw[black,fill=black,thick] (-0.5,1) circle [radius=0.03];
		\draw[black,fill=black,thick] (0.5,1) circle [radius=0.03];
		\draw[black,fill=black,thick] (1,1) circle [radius=0.03];
		\draw[thick] (-1,1) -- (-1,2);
		\draw[thick] (-0.5,1) -- (-0.5,2);
		\draw[thick] (0.5,1) -- (0.5,2);
		\draw[thick] (1,1) -- (1,2);
		\node[diamond,draw,black,fill=white,thick,scale=1.25] at (0.5,1.5) {};
		\node at (0.5,1.5) {$g$};
		\draw[thick] (0,3) -- (-1,2);
		\draw[thick] (0,3) -- (1,2);
		\draw[thick] (-0.5,2) -- (-0.75,2.25);
		\draw[thick] (0.5,2) -- (0.75,2.25);
		\draw[black,fill=black,thick] (0,3) circle [radius=0.03];
		\draw[black,fill=black,thick] (-0.75,2.25) circle [radius=0.03];
		\draw[black,fill=black,thick] (0.75,2.25) circle [radius=0.03];
		\draw[black,fill=black,thick] (-1,2) circle [radius=0.03];
		\draw[black,fill=black,thick] (-0.5,2) circle [radius=0.03];
		\draw[black,fill=black,thick] (0.5,2) circle [radius=0.03];
		\draw[black,fill=black,thick] (1,2) circle [radius=0.03];
	\end{tikzpicture}
}

\providecommand{\keywords}[1]{\tbf{\textit{Index terms---}} #1}

\begin{document}
	
	\title[Forest-skein groups III]{Forest-skein groups III: simplicity}
	
	\thanks{
		AB is supported by the Australian Research Council Grant DP200100067.\\
		RS is supported by an Australian Government Research Training Program (RTP) Scholarship.}
	\author{Arnaud Brothier and Ryan Seelig}
	\address{Arnaud Brothier, Ryan Seelig\\ School of Mathematics and Statistics, University of New South Wales, Sydney NSW 2052, Australia}
	\email{arnaud.brothier@gmail.com\endgraf
		\url{https://sites.google.com/site/arnaudbrothier/}}
	
	\begin{abstract}
		An Ore forest-skein category provides three forest-skein groups equipped with a powerful diagrammatic calculus analogous to Richard Thompson's groups $F\subset T\subset V$.
		We investigate when forest-skein groups have simple derived subgroups and establish two characterisations: a dynamical one and a categorical one. We then construct two classes of examples. The first associates two finitely presented simple groups to every finite binary tree and the second associates two simple groups to every $n$-ary Higman--Thompson group.
	\end{abstract}
	\maketitle
	
	
	\section*{Introduction}
	
	As part of the program to reconstruct conformal field theories from subfactors, Vaughan Jones discovered deep connections between Richard Thompson's groups $F\subset T\subset V$, knot theory, and physics \cite{Jones17}, see also the survey \cite{Brothier20}.
	To strengthen these connections, the first author has introduced {\it forest-skein} (FS) groups \cite{Brothier22,Brothier23c}, which are loosely a mixture of Thompson's groups and Jones' planar algebra \cite{Jones21}.
	
	\medskip{\bf Forest-skein categories and groups.} 
	Thompson's groups consist of piecewise linear bijections of the unit interval $[0,1]$ with dyadic rational breakpoints and power of $2$ slopes. Subsequently, each element of Thompson's groups can be represented by a pair of binary trees having the same number of leaves and a permutation of those leaves \cite{Brown87,Cannon-Floyd-Parry96}. In other words, $F$, $T$, and $V$ are isotropy groups of fraction groupoids of certain categories of binary forests with permutations. Other groups admit a similar tree-pair description.
	Burillo, Nucinkis, and Reeves showed that Cleary's irrational slope Thompson group $F_\tau$ arises from binary forests made from {\it two} types of caret (binary tree with two leaves) $\{Y_a,Y_b\}$ modulo the {\it skein relation}
	\begin{align*}
		\ClearySkeinA
		\quad\sim\quad
		\ClearySkeinB,
	\end{align*}
	and moreover they defined $T$ and $V$-versions of $F_\tau$ \cite{Cleary00,Burillo-Nucinkis-Reeves21,Burillo-Nucinkis-Reeves22}. A {\it forest-skein (FS) category} is a category of forests made from a non-empty set $S$ of binary carets modulo any set $R$ of skein relations (pairs of coloured binary trees with the same number of leaves). 
	The data $\la S|R\ra$ is called a {\it skein presentation} and we interpret $S$ as a set of {\it colours} for the monochromatic binary caret. 
	When an FS category $\cF$ is an Ore category (i.e.~is left-cancellative and has common right-multiples) one can construct three {\it types} of {\it forest-skein (FS) group}: $G_\cF\subset G_\cF^T\subset G_\cF^V$ 
	(there is also a braided type that we don't consider in this article).
	An element of an FS group is a equivalence class of planar diagrams: $t\circ\pi\circ s^{-1}$, where $t$ and $s$ are coloured trees with the same number of leaves and $\pi$ is a permutation of their leaves (taken up certain growing relations and skein relations). 
	Here is an example for the above skein relation:
	\begin{align*}
		Y_a\circ(12)\circ Y_b^{-1}\quad=\quad
		\GroupElementA
		\quad\sim_{grow}\quad
		\GroupElementB
		\quad\sim_{skein}\quad
		\GroupElementC.
	\end{align*}
	The class of FS groups is vast and contains many new groups as well as old ones in different forms. 
For instance, any homogeneously presented Ore monoid produces FS groups that decompose as wreath products \cite{Brothier23c}. Moreover, any family of monochromatic trees with same number of leaves gives an FS group \cite{Brothier22}. 	
	
\medskip{\bf Properties of forest-skein groups.} 
Every FS group $G^X$, with $X$ being $F,T,$ or $V$, contains a copy of $X$, hence $G^X$ has exponential growth and infinite geometric dimension. Moreover, every $T$ and $V$-type FS group has torsion and contains a copy of the non-abelian free group $\Z*\Z.$
	There are explicit examples of $F$-type FS groups with and without the following properties: being orderable, bi-orderable, torsion-free, having torsion-free abelianisation, having trivial centre, containing a copy of $\Z*\Z$. Hence many FS groups are not Guba--Sapir diagram groups, despite both families being defined by similar looking planar diagrams (diagram groups are all bi-orderable with torsion-free abelianisation, see \cite{Wiest03,Guba-Sapir06} and \cite[Corollary 9.9]{Guba-Sapir97}). Moreover, any (countable) FS group has vanishing first $\ell^2$-Betti number.
	
An Ore FS category with a finite skein presentation produces finitely presented FS groups and a large class of them enjoy topological finiteness property F$_\infty$. 
Though, we know of no examples that are finitely presented but not of type F$_\infty$.
	
	There are infinitely many FS groups satisfying the Haagerup property (using \cite{Brothier23a}), though we don't know if there exists an FS group without it, or even one that cannot act properly on a CAT(0) cube complex. 
	
	\medskip{\bf Infinite simple groups.}
	In comparison to {\it finite} simple groups, little is known about the class of {\it infinite} (discrete) simple groups. Thompson's groups $T$ and $V$ (defined in the $1960$'s) are famous for being the first examples of finitely presented infinite simple groups. 
	Today there are a large number of constructions of finitely presented simple groups that are essentially all obtained using dynamics 	\cite{Higman74,Brown87,Stein92,Burger-Mozes97,Rover99, Cleary00,Brin04, Caprace-Remy09, Brin10, Juschenko-Monod13, Matui15, Lodha-Moore16, Skipper-Witzel-Zaremsky19,Hyde-Lodha19,Lawson-Vdovina20, Burillo-Nucinkis-Reeves22,Belk-Zaremsky22, Belk-Bleak-Matucci-Zaremsky23,Tarocchi24}. We insist this list is far from being exhaustive. 
	Many basic questions remain open and new examples of infinite simple groups are needed \cite{Khukhro-Mazurov22}.
	
	\medskip{\bf Motivations and aims.} In this article we initiate a study of simple groups arising within the FS formalism, our primary motivation being to witness many new examples of finitely presented simple groups satisfying exotic properties. In this article we characterise when an FS category produces FS groups having simple derived subgroups and give two concrete classes of examples. In a future article \cite{Brothier-Seelig24} we will address the novelty of these groups via a deeper dynamical analysis.
	
	\medskip{\bf Content of the article and main results.}
	In Section \ref{sec:preliminaries} we recall the basics of the FS formalism and note some key results.
	
	In Section \ref{sec:actions} we further investigate the canonical group action $G\act Q$ of an FS group on its ``$Q$-space'' which was introduced in \cite{Brothier22}. Recall, $G\act Q$ is conjugate to the homogeneous action $G\act G^T/G$ and extends canonically to an action $G^V\act Q$. 
	We introduce a new point of view with {\it pointed trees} (i.e.~a tree with a distinguished leaf) forming a monoid $\cT_p$ via gluing. The $Q$-space is then a quotient of $\cT_p$ and we get a ``canonical monoid action" $\beta:\cT_p\act Q$ which extends (using matrices) to an action $\beta$ of the whole FS category $\cF.$ The action $\beta:\cT_p\act Q$ captures the local information of $\al:G\act Q$ which morally consists of ``prefix" replacement (see Remark \ref{rem:prefix}).
	
	Next, we extend these actions to a topological completion $\fC$ of $Q$.
	This will be mostly used in future articles, however, it already simplifies the exposition in this article.
	We construct $\fC$ in two different ways: first using Jones' technology and second using uniform structures. These actions on $Q$ and $\fC$ are analogous to the actions of Thompson's groups on the dyadic rationals and Cantor space, though
	a crucial difference in the FS world is:
	
	\begin{center}
		{\bf The canonical group action is not necessarily faithful.}    
	\end{center}
	
	This is because FS groups are defined {\it diagramatically} and not {\it dynamically} (see Example \ref{ex:non-faithful-action}).
	After that, we discuss topological and order properties of the completed action $\al:G^V\act\fC$. In particular, $\fC$ is always a totally ordered profinite space with smallest and largest elements denoted $o$ and $\omega$, respectively. 
	We define $K$ (resp. $K_o$, $K_\omega$) to be the normal subgroup of $G$ consisting of elements acting trivially on a neighbourhood of $\{o,\omega\}$ (resp. $o$, $\omega$) via $\al:G\act\fC$. 
	With this, we obtain a dynamical characterisation of simplicity in Section \ref{sec:dynamical-characterisation}.
	
	\begin{mainthm}[Theorems \ref{theo:simple-F}, \ref{theo:simple-T-V}, \ref{theo:simple-implies-faithful}]\label{theo:main-simple-dynamical-equivalence}
		Let $\cF$ be an Ore FS category with FS groups $G$, $G^T$, $G^V$, and normal subgroups $K_o,K_\omega,K\lhd G$. The following are equivalent:
		\begin{enumerate}
			\item The homogeneous action $G\act G^T/G$ is faithful.
			\item One (resp.~each) of the derived subgroups $[K,K]$, $[G^T,G^T]$, $[G^V,G^V]$, is simple.
		\end{enumerate}
		Moreover, the following are equivalent:
		\begin{enumerate}
			\item[(3)] The homogeneous action $G\act G^T/G$ is faithful and the two germ groups $G/K_o$ and $G/K_\omega$ are abelian.
			\item[(4)] The derived subgroup $[G,G]$ is simple.
		\end{enumerate}
	\end{mainthm}
	
	If $G\act G^T/G$ is faithful, we quickly deduce that $\al:G^V\act \fC$ is faithful too.
	Then a thin analysis (see Theorems \ref{theo:simple-F} and \ref{theo:simple-T-V}) of the action $\al:G^V\act\fC$ permits us to apply a Higman--Epstein-type argument (see \cite{Higman54,Epstein70}) to deduce $(1)\Rightarrow(2)$. This does {\it not} give simplicity of $[G,G]$ since $[G,G]/[K,K]$ is a proper quotient of $[G,G]$. However, when the homogeneous action is faithful, $[G,G]/[K,K]$ is the {\it only} proper quotient of $[G,G]$, and it is trivial if and only if $G/K_o$ and $G/K_\omega$ are abelian (see Theorem \ref{theo:simple-F}). 
	
There is a (strong) converse: the existence of {\it any} simple group $\Gamma$ satisfying $[K,K]\subset\Gamma\subset G^V$ implies the homogeneous action must be faithful. The proof is elementary using diagrams inspired by relations appearing in the Temperley--Lieb--Jones planar algebra (see Theorem \ref{theo:simple-implies-faithful}).
	
	In Section \ref{sec:categorical-characterisation} we witness a striking rigidity phenomena: 
	a large class of subgroups of $G$ normalised by $G^T$ come from quotients of the FS category $\cF$. 
	In particular, the subgroup $\ker(G\act G^T/G)$ arises as a quotient of $\cF$. 
	From there we deduce a categorical characterisation of faithfulness of the canonical group action.
	
	\begin{mainthm}[Theorem \ref{theo:Ore-simple}]\label{theo:main-category}
		Let $\cF$ be an Ore FS category with FS groups $G$ and $G^T$. The following are equivalent:
		\begin{enumerate}
			\item The homogeneous action $G\act G^T/G$ is faithful.
			\item The FS category $\cF$ admits no proper quotients in the category of Ore FS categories.
			\item The FS category $\cF$ admits no proper quotients that are left-cancellative.
		\end{enumerate}
	\end{mainthm}
	Item (3) means that adding any ``new" skein relation to a given Ore FS category $\cF$ yields a category that is no longer left-cancellative.
	Combined with algorithmic tools of Dehornoy to detect left-cancellativity (such as the strong cube condition see \cite{Dehornoy03} and \cite[Section 2]{Brothier22}) this may give a useful way to deduce faithfulness of $G\act G^T/G$, however, we have been more successful working directly with the dynamics.
	
	Recall that an Ore category is left-cancellative and has common right-multiples. Since the latter property is preserved under taking quotients we immediately deduce $(2) \Leftrightarrow(3)$.
	Now, if there exists a proper quotient $\cF\onto \ti\cF$ of Ore FS categories, with respective FS groups $G,G^T$ and $\ti{G},\ti{G}^T$,
	then $G\act G^T/G$ factors through $\ti G\act \ti{G}^T/\ti G$ giving that $G\act G^T/G$ is not faithful.
	The surprising direction is the implication $(2)\Rightarrow (1)$ which comes from the rigidity phenomena described above. 
	Better yet, we obtain that $G/\ker(\alpha)$ is always an FS group whose underlying category is $\cF/\ker(\beta)$ where $\beta$ is the canonical category action of $\cF.$ 
We deduce the following (see Proposition \ref{prop:F-modulo-kerbeta-simple}).
		
\begin{maincor}\label{cor:FS-quotient}
Let $\cF$ be an Ore FS category with FS group $G$ and canonical actions $\alpha,\beta$.
There exists an FS group $\ti G$ which is a quotient of $G$ and for which the canonical action is faithful.
Namely, $G\onto \ti G$ is the quotient map $G\onto G/\ker(\alpha)$ and $G/\ker(\alpha)$ is the FS group of $\cF/\ker(\beta)$.
\end{maincor}
		
	Our last two theorems immediately give the following simplicity characterisations.
	
	\begin{maincor}\label{cor:main-equivalences}
		Let $\cF$ be an Ore FS category with FS groups $G,G^T,G^V$ and normal subgroup $K\lhd G$.
		The following are equivalent:
		\begin{enumerate}
			\item One (resp.~each) of the groups $[K,K]$, $[G^T,G^T]$, $[G^V,G^V]$ is simple;
			\item The homogeneous action $G\act G^T/G$ (resp.~the canonical group action $G^V\act Q$ or its completion $G^V\act \fC)$ is faithful; and
			\item There are no surjective morphisms $\cF\onto\ti\cF$ with $\ti\cF$ Ore.
		\end{enumerate}
	\end{maincor}
	
	The $F$-type simple groups we produce (i.e.~$[K,K]$ that is sometimes equal to $[G,G]$) are (left-)orderable, never finitely-generated, and contain a copy of $F$. Some of them contain non-abelian free subgroups (see Example \ref{ex:germ-examples}).

	The $T$ and $V$-type simple groups we produce (i.e. $[G^T,G^T]$ and $[G^V,G^V]$) always contain Thompson's group $T$, thus have torsion and contain $\Z*\Z$.
	Often they have finite index in their overgroups, so inherit most properties from $G^T$ and $G^V$. 
	This leads us to study the abelianisations $G^T_{\ab}$ and $G^V_{\ab}$ coming from {\it any} Ore FS category. 
	We do this in Section \ref{sec:abelianisation} and obtain a friendly description for $G^T_{\ab}$ and $G^V_{\ab}$ that can be read directly off a skein presentation for associated FS category.
	The $F$-case is more subtle and we do not treat it in this article. 
	
	\begin{mainthm}[Theorem \ref{theo:abelianisation}]\label{theo:main-abelianisation}
		Let $\cF=\FS\la S|R\ra$ be a presented Ore FS category with FS groups $G^T,G^V$ and let $a\in S$ be a fixed colour.
		For each coloured tree $t$ define $\chi(t)\in \Z^S$ where $\chi(t)(b)$ equals to the number of $b$-coloured vertices in $t.$
		The map $\chi$ provides group isomorphisms
		$$G^T_{\ab}\simeq G^V_{\ab}\simeq\Z^S/\la\chi(u)=\chi(v),a=0:(u,v)\in R\ra.$$
	\end{mainthm}
	Let $Y$ be either $T$ or $V$. It is rather straightforward to check that ``counting colours" provides a morphism inside an abelian group with generator set $S.$ Since the trees $t,s$ in a group element $t\circ\pi\circ s^{-1}\in G^Y$ have same number of vertices, we may disregard one colour, say $a\in S$, and deduce a well-defined surjective group morphism 
	$$G^Y\ni t\circ \pi\circ s^{-1}\mapsto \chi(t)-\chi(s)\in \Z^S/\la \chi(u)=\chi(v),a=0:(u,v)\in R\ra.$$ 
	Conversely, using relations of $G^Y$ we are able to conclude that all morphisms valued in an abelian group factor through the one above.

Proving that a given skein presentation produces an Ore FS category with no proper Ore quotient (or has faithful canonical group action) is a difficult task and in most situations we don't know if this is true.
Sections \ref{sec:examples-via-dynamics} and \ref{sec:examples-from-simple-groups} are dedicated to find such examples.
	
	\medskip{\bf Generalising Cleary's irrational slope Thompson group.} 
	Our first class of examples is inspired by the Burillo--Nucinkis--Reeves presentation for Cleary's irrational slope Thompson group
	$$\FS\la a,b| Y(a)(I\ot Y(a))=Y(b)(Y(b)\ot I)\ra,$$
	whose relation is the equality of two vines of length $2$ (see the diagram on the first page) \cite{Burillo-Nucinkis-Reeves21}. The FS formalism suggests a natural generalisation by considering two vines of length $n$ rather than length $2$ as explained in \cite[Section 3.6.3]{Brothier22}.
	Our technology permits us to treat even more general examples of the form:
	$$\cF_z:=\FS\la a,b| Y(a)(I\ot z(a))= \lambda_n(b)\ra$$
	where $z$ is {\it any} tree with $n$ leaves and $\lambda_n(b)$ is the $b$-coloured left-vine with $n+1$ leaves, see (\ref{ex:skein-relation}) for an example of this skein relation.
	
	\begin{maincor}\label{maincor-Cleary}
		For any monochromatic tree $z$ with $n$ leaves we have that $\cF_z$ is an Ore FS category giving the FS groups $G_z,G_z^T,G_z^V$.
		The action $G_z\act G_z^T/G_z$ is faithful and the germ groups of $G_z$ at $o,\omega$ are abelian. Moreover,
		$(G^T_{z})_{\ab}\simeq(G^V_{z})_{\ab}\simeq\Z/n\Z.$
		Hence:
		\begin{enumerate}
			\item the derived subgroups $[K_z,K_z]$ and $[G_z,G_z]$ are equal, simple; and
			\item the derived subgroups $[G^T_{z},G^T_{z}]$ and $[G^V_{z},G^V_{z}]$ are simple and of type \textnormal{F}$_\infty$.
		\end{enumerate}
	In particular $[G^T_{z},G^T_{z}]$ and $[G^V_{z},G^V_{z}]$ are finitely presented simple groups.
	\end{maincor}
	Most of the proof resides in proving faithfulness of $G_z\act G_z^T/G_z$. This is done by constructing a semi-normal form for elements of $G_z$ using a similar methods to those appearing in \cite{Cannon-Floyd-Parry96} and nicely extended in \cite[Section 7]{Burillo-Nucinkis-Reeves21} (see Theorem \ref{theo:canonical-action-faithful}). 
	For more general skein presentations this strategy needs to be substantially refined as we will explain in \cite{Brothier-Seelig24}.

\medskip{\bf Higman--Thompson forest-skein categories.} 
Our second class of examples is inspired by the ternary Higman--Thompson group $F_3$ obtained in \cite{Brothier22} using the FS category 
$$\FS\la a,b| Y(a)(I\ot Y(b))=Y(b)(Y(a)\ot I)\ra.$$ 
For all $n\geqslant2$ define the FS category
$$\cH_n:=\FS\la S_n |Y(r)(I\otimes Y(s)) = Y(s)(Y(r)\otimes I) : r<s\ra$$
where $S_n$ is a totally ordered set with $n-1$ elements.
\begin{maincor}\label{maincor-Higman}
	For all $n\geqslant 2$ we have that $\cH_n$ is Ore with FS groups denoted $H_n,H_n^T,H_n^V.$
	The homogeneous action $H_n\act H_n^T/H_n$ is faithful.
	The group $H_n$ is isomorphic to the $n$-ary Higman--Thompson group (usually denoted $F_n$).
	We have the abelianisation
	\begin{align*}
		(H^T_{n})_{\ab}\simeq(H^V_{n})_{\ab}\simeq\Z^{n-2}.
	\end{align*}
	The derived subgroups $[H^T_{n},H^T_{n}]$ and $[H^V_{n},H^V_{n}]$ are simple.
\end{maincor}
Denote by $F_n\subset T_n\subset V_n$  the $n$-ary Higman--Thompson groups. The main work is to show $H_n\simeq F_n$. From there, we prove faithfulness of the canonical group action for $\cH_n$ indirectly. Indeed, Brown proved that $[F_n,F_n]$ is simple, which implies via Theorem \ref{theo:main-simple-dynamical-equivalence} that $H_n\act H_n^T/H_n$ is faithful \cite{Brown87}. Using again Theorem \ref{theo:main-simple-dynamical-equivalence} (but in the other direction) we deduce that the derived subgroups $[H_n^T,H_n^T]$ and $[H_n^V,H_n^V]$ are simple, however we can say nothing immediately about their finiteness properties.

Even though $H_n\simeq F_n$ we have $H_n^T\not\simeq T_n$ and $H_n^V\not\simeq V_n$ (and similarly for the derived subgroups which we will prove in \cite{Brothier-Seelig24}). 
Hence, we obtain ``new" simple groups using ``old ones".
There are obvious inclusions between colour sets which provide inclusions of the FS groups, and in particular, between the $F_n$. These embeddings seem new to us. Finally, we may construct many more categories and groups by replacing $S_n$ by a general poset, though most of these will have non-trivial quotients (adding skein relations until one gets $\cH_n$ for instance). 

\medskip{\bf Classification and future work.} It is natural to ask if any of the simple groups we produce are isomorphic to any known simple groups. 
In \cite{Brothier-Seelig24} we will show all FS groups with faithful homogeneous action admit a {\it rigid action} that can be completely reconstructed from the group (as in \cite{McCleary78, Rubin89}).
This provides refined dynamical invariants for such FS groups.
Many of the germ groups we encounter for $F$-type FS groups will contain non-abelian free groups and hence will serve as an obstruction to admitting a piecewise linear action on the real line (see Example \ref{ex:germ-examples}). We are also going to develop methods for proving the canonical group action is faithful for more FS categories.

\medskip{\bf Acknowledgements.} We warmly thank Christian de Nicola Larsen for many helpful conversations and his feedback regarding this work. 
We are grateful to Matt Brin, Yash Lodha, Brita Nucinkis, Dilshan Wijesena, and Matt Zaremsky for their enthusiastic encouragements and comments.

\section{Forest-skein categories, monoids, and groups}\label{sec:preliminaries}

In this section we recall the basics of the forest-skein (FS) formalism. We refer the reader to the first article on the subject \cite{Brothier22} for a comprehensive introduction. In this article we will {\it almost} exclusively consider {\it binary} forests and trees. The only digression to $n$-ary forests will appear in Section \ref{sec:examples-from-simple-groups}.

\medskip{\bf Conventions.}
We write natural numbers $\N=\{1,2,3,\dots\}$, $|A|$ for the cardinality of a set $A$, $\Mon\la S|R\ra$ (resp. $\Gr\la S|R\ra$) for the monoid (resp. group) with presentation $\la S|R\ra$, and $\gamma:\Gamma\act X$ for the action of a group or monoid $\Gamma$ on a set or space $X$. For a group $\Gamma$ its {\it derived subgroup} $[\Gamma,\Gamma]$ is the subgroup of $\Gamma$ generated $ghg^{-1}h^{-1}$ for all $g,h\in\Gamma$. It is standard that $[\Gamma,\Gamma]$ is normal and $\Gamma_{\ab}=\Gamma/[\Gamma,\Gamma]$ is its {\it abelianisation}. 

\subsection{Forest-skein categories and monoids}\label{sec:FS-cat-and-mon}

\subsubsection{Trees and forests}\label{sec:trees-and-forests} A {\it tree} $t$ is a finite connected subgraph of the infinite binary tree (see Figure \ref{fig:infinite-binary-tree}) containing the root and satisfying that every vertex of $t$ has either zero or two children.
\begin{figure}[h]
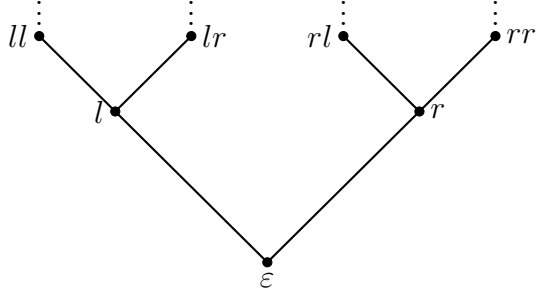
\label{fig:infinite-binary-tree}
	\centering
	\InfiniteBinaryTree
	\caption{The infinite binary tree. Its vertices are labelled by words in the alphabet $\{l,r\}$ and its root is the empty word $\varepsilon$.}
\end{figure}
The vertices of a tree $t$ with no children are called the {\it leaves} of $t$ and form a set $\Leaf(t)$. Leaves are totally ordered from left to right. Vertices of $t$ with two children are called {\it interior} and form a set $\Ver(t)$. The unique tree $I$ with one vertex is called the {\it trivial} tree, though for graphical reasons we draw $I$ as a vertical line segment in the plane as we will eventually identify it with the trivial permutation (see Section \ref{sec:permutations} and \ref{sec:FS-T-and-V}). The unique tree $Y$ with two leaves is called the {\it caret}. 

\medskip{\bf Coloured trees and forests.} Let $S$ be a {\bf non-empty} set of {\it colours}. A {\it coloured tree} (over $S$), or {\it tree} for short, is a pair $(t,c)$ where $t$ is a tree and $c:\Ver(t)\to S$ is a {\it colour function}, which labels each interior vertex of $t$. For example
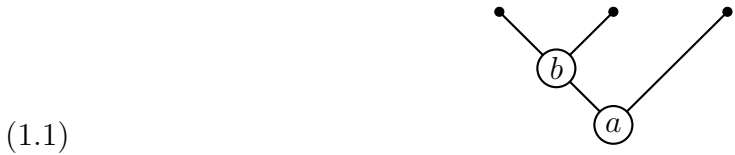
\begin{align}\label{prelim:tree}
	\ColouredTreeExample
\end{align}
is a coloured tree over $S=\{a,b\}$. Given a tree $(t,c)$ we call $\nu\in\Ver(t)$ an {\it $a$-vertex} if $c(\nu)=a$ and a tree with only $a$-vertices is called an {\it $a$-tree}. Often we just write $t$ for a coloured tree having the colour function be implicit. A {\it coloured forest} (over $S$), or {\it forest} for short, is a list $f=(f_1,f_2,\dots,f_n)$ of coloured trees. Graphically, $f$ is the disjoint union of its trees, drawn with $f_i$ to the left of $f_{j}$ if $i<j$. 
For example
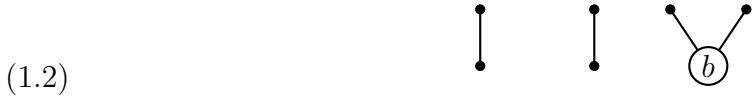
\begin{align}\label{prelim:forest}
	\ColouredForestExample
\end{align}
is the coloured forest $(I,I,Y_b)$ where $Y_b$ denotes the caret whose root is a $b$-vertex. The {\it roots} of $f=(f_1,f_2,\dots,f_n)$ are elements of $\Root(f)=\{1,2,\dots,n\}$ and the {\it leaves} of $f$ are elements of $\Leaf(f)=\sqcup_{i=1}^n\Leaf(f_i)$. The leaves of $f$ inherit a total order from $\Root(f)$ and each of the leaf sets $\Leaf(f_i)$. 
We may identify $\Leaf(f)$ with $\{1,2,\dots,m\}$, where $m=|\Leaf(f)|$ is the number of leaves of $f$. An {\it infinite forest} is an infinite list of trees that is the trivial tree for all but finitely many entries. 
We denote the infinite forest all of whose trees are trivial by $I^{\ot \infty}$. 

\medskip{\bf Composing forests.} We may {\it compose} a forest $f$ with a forest $g$ if the number of leaves of $f$ equals the number of roots of $g$. When this holds, we form a new forest $f\circ g$ by gluing the $i$th tree of $g$ to the $i$th leaf of $f$ for every $i\in\Root(g)$. For example let $t$ be the tree in (\ref{prelim:tree}) and $f$ be the forest in (\ref{prelim:forest}). Since $t$ has $3$ leaves and $f$ has $3$ roots we may form $t\circ f$, which is the tree:
\begin{align}\label{prelim:composition}
	\CompositionA
	\quad
	=
	\quad
	\CompositionB
\end{align}
Similarly one may compose {\it any} two infinite forests by (vertical) stacking. Often we refer to going from $t$ to $t\circ f$ as {\it growing $t$ by $f$} and we will mostly write $fg$ instead of $f\circ g$. 

\medskip{\bf Tensor product of forests.} Let $f=(f_1,f_2,\dots,f_n)$ and $g=(g_1,g_2,\dots,g_m)$ be forests. We may form their {\it tensor product} by concatenating their lists, that is
\begin{align*}
	f\otimes g:=(f_1,f_2,\dots,f_n,g_1,g_2,\dots,g_m).
\end{align*}
Graphically $f\ot g$ is obtained by drawing $f$ to the left of $g$. 

\medskip{\bf Particular types of trees and forests.} Monochromatic {\it left-vines} are defined inductively as follows: the first left-vine is the caret $Y$ and the $n$th left-vine is the $(n-1)$th left-vine with a caret glued to its first leaf. For example, the first three monochromatic left-vines are
\begin{align}\label{eqn:vines}
	\MonochromaticCaret,\quad\quad
	\MonochromaticVineA,\quad\quad \textnormal{and}\quad\quad
	\MonochromaticVineB.
\end{align}
The tree in (\ref{prelim:tree}) is a coloured left-vine. Right-vines are defined similarly. The $n$th {\it complete tree} defined inductively as follows: the first complete tree is the caret $Y$ and the $n$th complete tree is formed from the $(n-1)$th complete tree by attaching a caret to each of its leaves. For example, the first three monochromatic complete trees are
\begin{align*}
	\MonochromaticCaret,\quad\quad
	\MonochromaticCompleteA,\quad\quad \textnormal{and}\quad\quad
	\MonochromaticCompleteB.
\end{align*}
The tree in (\ref{prelim:composition}) is a coloured complete tree.

Given a tree $t$ and $1\leqslant j\leqslant n$ we often write $t_{j,n}$ for the $n$-rooted forest whose only non-trivial tree is $t$ rooted at $j$. When $t$ is a caret, we call $t_{j,n}$ an {\it elementary forest} and write $a_{j,n}$ instead of $(Y_a)_{j,n}$. Observe that any forest can be written as a composition of elementary forests. 

\medskip{\bf Subtrees.} Let $t$ be a tree and $\nu\in\Ver(t)$ be an interior vertex (i.e. not a leaf). We say $s$ is a {\it subtree of $t$ rooted at $\nu$} if $s$ is a tree (in our sense) that embeds in $t$ whose root is $\nu$. Moreover, $s$ is a {\it rooted subtree of $t$} if $s$ is a subtree rooted at $\varepsilon$. This is equivalent to say there exists a forest $f$ for which $t=s\circ f$. The set of subtrees of $t$ rooted at $\nu$ are partially ordered by $s\leqslant t$ if $s$ is a rooted subtree of $t$. 

\medskip{\bf Quasi-trees.} Generally, one may consider $k$-ary forests whose non-leaf vertices have $k$ outgoing edges, rather than just $2$. Let $t$ be a $k$-ary tree and $s$ be a binary tree with $k$ leaves. A {\it quasi-tree} with {\it skeleton} $t$ and {\it cell $s$} is the binary tree formed by replacing every non-leaf vertex in $t$ by a copy of $s$. A {\it quasi-left-vine} is a quasi-tree whose skeleton is a left-vine. For example the quasi-left-vine
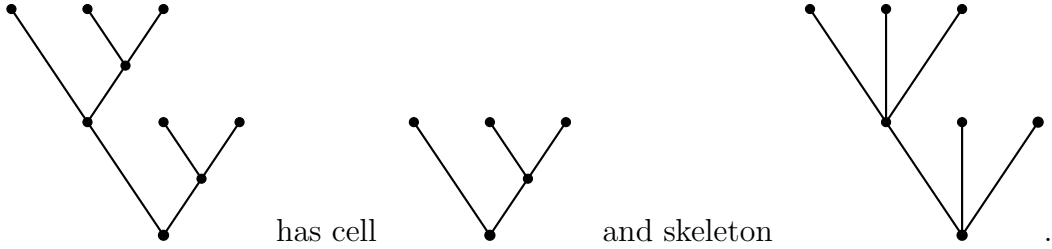
\begin{align*}
	\QuasiVine\quad\textnormal{has cell}\quad\Cell\quad\textnormal{and skeleton}\quad\Skeleton.
\end{align*}

\subsubsection{Forest-skein categories and monoids} Let $S$ be a set of colours that we will always assume to be {\bf non-empty}. We write $\FS\la S\ra$ (resp. $\FS\la S\ra_\infty$) for the set of all forests (resp. infinite forests) coloured by $S$. For a forest $f$ with $r$ roots and $l$ leaves we write $f:l\to r$. Composition and tensor products of forests endows $\FS\la S\ra$ the structure of a monoidal category whose objects are the natural numbers $\N$. We call $\FS\la S\ra$ the {\it free FS category over $S$}. Similarly $(\FS\la S\ra_\infty,\circ)$ is a monoid called the {\it free FS monoid over $S$}.

\medskip{\bf Skein relations.} A {\it skein relation} over $S$ is a pair of trees $(u,v)\in\FS\la S\ra^2$ having the same number of leaves. We often write $u\sim v$ if $(u,v)$ is a skein relation, for instance
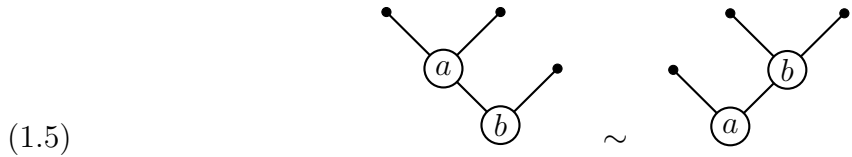
\begin{align}\label{ex:colour-preserving-skein}
	\TernaryA
	\quad
	\sim
	\quad
	\TernaryB
\end{align}
is a skein relation (in fact it gives the ternary Thompson group $F_3$, see Section \ref{sec:examples-from-simple-groups}). Given a set of skein relations $R\subset\FS\la S\ra^2$ we let $\overline{R}$ be the congruence relation on $(\FS\la S\ra,\circ,\otimes)$ generated by $R$ (i.e.~the smallest equivalence relation containing $R$ and closed under $\circ$ and $\ot$). 
Since performing a skein relation doesn't change number of leaves, the monoidal category structure of $\FS\la S\ra$ passes to the quotient $\FS\la S\ra/\overline{R}$. A similar construction applies for the free FS monoid. Write $f\ot I^{\ot\infty}$ for the infinite forest beginning with the forest $f$ and let 
\begin{align*}
	R_\infty=\{(I^{\ot n}\otimes u\otimes I^{\ot \infty},I^{\ot n}\otimes v\otimes I^{\ot \infty}):(u,v)\in R,n\geqslant0\}    
\end{align*}
and write $\ov{R}_\infty$ for the congruence on $(\FS\la S\ra_\infty,\circ)$ generated by $R_\infty$. The monoid structure of $\FS\la S\ra_\infty$ passes to the quotient $\FS\la S\ra_\infty/\overline{R}_\infty$.

\begin{definition}
	\begin{enumerate}
		\item A {\it forest-skein (FS) presentation} or {\it skein presentation} is a pair $\la S|R\ra$ where $S$ is a set of colours and $R$ is a set of skein relations over $S$.
		\item A {\it presented forest-skein (FS) category} is the category
		\begin{align*}
			\FS\la S|R\ra:=\FS\la S\ra/\overline{R}.
		\end{align*}
		An {\it abstract forest-skein (FS) category} is a category $\cF$ isomorphic to a presented FS category. A morphism in $\cF$ from $2$ to $1$ is called a {\it caret} or a {\it colour} of $\cF$.
		\item A {\it presented forest-skein (FS) monoid} is the monoid
		\begin{align*}
			\FM\la S|R\ra:=\FS\la S\ra_\infty/\overline{R}_\infty.
		\end{align*}
		An {\it abstract forest-skein (FS) monoid} is a monoid isomorphic to a presented FS monoid.
	\end{enumerate}
\end{definition}

\begin{remark}\label{rem:abstract-FS}
	An abstract FS category is a small monoidal category $(\cF,\circ,\ot)$ together with a favorite object $1\in\ob(\cF)$. We demand that a morphism between FS categories is monoidal and sends $1$ to $1.$ This assures that the number of roots and leaves are preserved.
	Moreover, note that a skein presentation for $\cF$ is not a category presentation in the usual sense (see \cite{Dehornoy-Digne-Godelle-Krammer-Michel15}), but is rather a presentation for $\cF$ as a {\it monoidal} category. 
\end{remark}

\subsubsection{Notation}\label{sec:notation} We typically denote a generic FS category by $\cF$, its set of trees by $\cT$, and its associated FS monoid by $\cF_\infty$. We often write $f$ for the class of a forest $f$ inside a presented FS category $\FS\la S|R\ra$. Trees are typically notated by $r,s,t,x,y,z$, while forests by $f,g,h,k,p,q$. If $(t,c)$ is a coloured tree with $c(\nu)=a$ for all $\nu\in\Ver(t)$ (i.e. $(t,c)$ is an $a$-tree), then we write $(t,c)=t(a)$. In the case when $t$ is a caret we often write $Y(a)=Y_a$. It will often be convenient to adopt the following compact notation for trees: writing a tree $t$ as a composition of elementary forests we suppress the root index, for example, we write $a_1b_1c_3$ instead of $a_{1,1}b_{1,2}c_{3,3}$ or $Y_a(Y_b\ot Y_c)$. This notation is frequently used for writing skein presentations. We will also sometimes write $a_1^n$ for $a_{1,1}a_{1,2}\cdots a_{1,n}$.

\subsubsection{Permutations}\label{sec:permutations}

For each integer $n\geqslant1$ let $C_n$ be the  $n$th cyclic group, $\Sigma_n$ the $n$th symmetric group, and let us identify $C_n$ as a subgroup of $\Sigma_n$.
We will view these as groups of planar diagrams. Indeed, for $i,j\in\N$ write $\ell(i,j)$ for the line segment between the points $(i,0)$ and $(j,1)$ in $\R\times[0,1]$. For each $n\geqslant1$ we associate to $\sigma\in\Sigma_n$ the union of line segments $\cup_{j=1}^n\ell(\sigma(j),j)$ inside $\R\times[0,1]$. 
The product $\sigma\circ\tau$ of two permutations $\sigma$ and $\tau$ corresponds to stacking $\tau$ on top of $\sigma$ and reducing. For instance, in $\Sigma_3$ we have:
\begin{align*}
	(12)\circ(123)\quad=\quad\PermutationA\quad=\quad\PermutationB\quad=\quad(23).
\end{align*}
Similar can be done for the group $\Sigma_\infty$ of all finitely supported permutations of $\N$ by drawing infinitely many vertical lines to the right of a diagram. For any two permutations $\sigma\in\Sigma_n$, $\tau\in\Sigma_m$ we can form their horizontal concatenation $\sigma\otimes\tau\in\Sigma_{n+m}$. Writing $I$ for the vertical line $\ell(1,1)$, the neutral element in $\Sigma_n$ is the diagram $I^{\otimes n}$. 

\subsubsection{Forest-skein categories and monoids: $T$ and $V$-type.}\label{sec:FS-T-and-V} 
The formal way to construct the analogues of Thompson's groups $T$ and $V$ (and other interesting generalisations \cite{Witzel-Zaremsky18}) associated to an FS category $\cF$ is to form Brin--Zappa--Sz\'{e}p products of $\cF$ with certain groupoids, as Brin did to define the braided Thompson group $BV$ \cite{Brin05,Brin07}. This is rather technical in full generality, so we present a concrete construction adapted to our purposes.

\medskip{\bf Symmetric forests.} Let $\cF$ be an FS category. A {\it symmetric forest} over $\cF$ is a pair $(f,\pi)$ where $f\in\cF$ is a forest with $l$ leaves and $\pi\in\Sigma_l$. The notion of roots and leaves naturally extends to symmetric forests. Pictorially, a symmetric forest is a forest with a permutation stacked on top, for instance:
\begin{align*}
	\SymmetricColouredForest.
\end{align*}
We write $\cF^V$ for the set of symmetric forests over $\cF$. We now explain how the category structure of $\cF$ extends to $\cF^V$. 

\medskip{\bf Mutual actions of forests and permutations.} Let $f:l\to r$ be a forest in $\cF$ and let $\pi\in\Sigma_r$. Write $l_j$ for the number of leaves in the $j$th tree of $f$. We set $f^\pi:l\to r$ to be the forest in $\cF$ whose $\pi(j)$th tree is the $j$th tree of $f$ for all $1\leqslant j\leqslant n$ and set $\pi^f\in\Sigma_l$ to be the permutation formed from $\pi$ by ``splitting'' the $j$th line segment $\ell(\pi(j),j)$ into $l_j$ parallel line segments for all $1\leqslant j\leqslant n$. Composition and tensor product in $\cF^V$ are then defined as
\begin{align*}
	(f,\sigma)\circ(g,\tau):=(f\circ g^\sigma,\sigma^g\circ\tau)\quad\textnormal{and}\quad(f,\sigma)\otimes(g,\tau):=(f\otimes g,\sigma\otimes\tau)
\end{align*}
where $\sigma\otimes\tau$ is horizontal concatenation of permutations. It is easy to check $(\cF^V,\circ,\otimes)$ is a monoidal category containing $(\cF,\circ,\otimes)$ as monoidal subcategory. We identify a forest $f\in\cF$ having $l$ leaves with $(f,I^{\otimes l})\in\cF^V$ and a permutation $\pi\in\Sigma_l$ with $(I^{\otimes l},\pi)\in\cF^V$, hence $(f,\pi)=f\pi$. We have so-called ``Zappa--Sz\'{e}p'' relations $\sigma\circ h=h^\sigma\circ\sigma^h$ when the number of strands of $\sigma$ equals the number of roots of $h$. These are best understood pictorially, for example
\begin{align*}
	\ZappaSzep
\end{align*}
We see that dragging the caret down creates two parallel strands as per the definition of $\sigma^h$. A symmetric forest $(f,\pi)\in\cF^V$ is called {\it cyclic} if $\pi$ is a cyclic permutation. The set $\cF^T$ of all cyclic forests is a subcategory of $\cF^V$, however it is not closed under $\otimes$, so is not a \emph{monoidal} subcategory.
It will sometimes be convenient to notate $\cF^F=\cF$.
The above construction can be applied to $\cF_\infty$ and $\Sigma_\infty$ to produce a $V$-type FS monoid $\cF_\infty^V$, however there is no $T$-type monoid as the family of cyclic groups admits no directed structures.
Once again we may notate $\cF_\infty^F=\cF_\infty$.
\begin{observation}
	Let $\cF$ be an FS category and $\cF_\infty$ its associated FS monoid.
	\begin{enumerate}
		\item $(\cF^V,\circ,\otimes)$ is a monoidal category called the \emph{$V$-type FS category associated to $\cF$}.
		\item $(\cF^T,\circ)$ is a category called the \emph{$T$-type category associated to $\cF$}.
		\item $(\cF^V_\infty,\circ)$ is a monoid called the \emph{$V$-type monoid associated to $\cF$}.
	\end{enumerate}
\end{observation}

\subsection{Forest-skein groups}

\subsubsection{Calculus of fractions}
We are interested in categories $\cC$ that admit a {\it right-calculus of fractions}.
This is assured by the two conditions:
\begin{enumerate}
	\item $\cC$ is left-cancellative, i.e.~$f\circ g=f\circ h$ implies $g=h$;
	\item $\cC$ has common right-multiples (often called {\it right-Ore property}), i.e.~if $f,g\in\cC$ have the same source, then there exist $h,k$ satisfying $f\circ h=g\circ k$.
\end{enumerate}
If $\cC$ is both left-cancellative and has common right-multiples we say $\cC$ is an {\it Ore category} and admits a {\it (right-)calculus of fractions}. If a monoid $M$ is an Ore category, then we say it is an {\it Ore monoid}.

\begin{observation}\label{obs:Ore-category}
Let $\cF$ be an FS category with associated FS monoid $\cF_\infty$. The following are equivalent.
	\begin{enumerate}
	\item One of $\cF$, $\cF^T$, $\cF^V$, $\cF_\infty$, $\cF_\infty^V$ is Ore.
	\item Each of $\cF$, $\cF^T$, $\cF^V$, $\cF_\infty$, $\cF_\infty^V$ is Ore.
	\end{enumerate}
\end{observation}

\subsubsection{Forest-skein groupoids and groups}\label{sec:FS-groups} A category $\cC$ with a (right-)calculus of fractions admits a {\it fraction groupoid} $\Frac(\cC)$. This is the set of all pairs $(f,g)$ with $f,g\in\cC$ having the same target, up to $(f,g)\sim(f\circ p,g\circ p)$ where $p$ is composable with $f$ and $g$. The calculus of fractions for $\cC$ provides a partial binary operation on $\Frac(\cC)$. Indeed, when the targets of $g$ and $h$ are equal, we set
\begin{align*}
	[f,g]\circ[h,k]&=[f\circ p,g\circ p]\circ[h\circ q,k\circ q]=[f\circ p,k\circ q],
\end{align*}
where $g\circ p=h\circ q$ is given by existence of common right-multiples. Left-cancellation in $\cC$ implies $\circ$ is well-defined. Now $(\Frac(\cC),\circ)$ is a groupoid with $[f,g]^{-1}=[g,f]$. The {\it isotropy group} at an object $c$ of $\cC$ is
\begin{align*}
\Frac(\cC,c):=\{[f,g]\in\Frac(\cC):\textnormal{$f$ and $g$ have source $c$}\}.
\end{align*} 
It is standard that $(\Frac(\cC,c),\circ)$ is a group. If $M$ is an Ore monoid (thought of as a one object category) then $\Frac(M)$ is already a group.

\begin{definition}\label{def:FS-group}
	Let $\cF$ be an Ore FS category, $\cF_\infty$ the corresponding Ore FS monoid, $X\in\{F,T,V\}$, and $Y\in\{F,V\}$. 
	\begin{enumerate}
		\item The {\it forest-skein (FS) groupoid} of $\cF^X$ is the fraction groupoid $\Frac(\cF^X)$.
		\item The {\it forest-skein (FS) group} of $\cF^X$ is the isotropy group $\Frac(\cF^X,1)$. 
		\item The {\it forest-skein (FS) group} of $\cF_\infty^Y$ is the fraction group $\Frac(\cF_\infty^Y)$. 
	\end{enumerate}
\end{definition}

\begin{remark}
	Let $\cF$ be an Ore FS category.
	\begin{enumerate}
		\item We have $\Frac(\cF,1)\subset\Frac(\cF^T,1)\subset\Frac(\cF^V,1)$.
		\item It is often convenient to identify $[f,g]\in\Frac(\cF)$ with $f\circ g^{-1}$ and this can be viewed diagrammatically as the forest $f$ with the horizontal reflection of $g$ stacked on top (see Section \ref{sec:dynamical-characterisation}). For $[f,g]\mapsto f\circ g^{-1}$ to be one-to-one $\cF$ also needs to be {\it right}-cancellative (as well as being left-cancellative). The groupoid operation is then diagrammatic: if $g$ and $h$ have the same number of roots, then $h\circ k^{-1}$ can be stacked on top of $f\circ g^{-1}$, which can be reduced to a fraction by isotoping carets, applying skein relations in $\cF$, and applying {\it universal skein relations} $Y_a\circ Y_a^{-1}=I$ and $Y_a^{-1}\circ Y_a=I\ot I$ for every colour $a$ of $\cF$.
		\item Any $g\in\Frac(\cF^V,1)$ is of the form $g=[t\pi,s]$ for some trees $t,s\in\cF$ and some permutation $\pi$ of their leaves. Indeed, by definition $g=[t\sigma,s\tau]$ for some permutations $\sigma,\tau$, and so by setting $\pi=\sigma\tau^{-1}$ we have $[t\sigma,s\tau]=[t\sigma\tau,s\tau\tau^{-1}]=[t\pi,s]$.
		\item For having a right-calculus of fractions we could relax the first axiom by only requiring that if $f\circ g=f\circ h$, then there exists $f'$ satisfying $g\circ f'=h\circ f'$, see \cite{Gabriel-Zisman67}.
		However, we demand the stronger axiom of left-cancellativity for reasons inherent to the FS category formalism.
	\end{enumerate}
\end{remark}

\subsubsection{Relationship between FS category group and FS monoid group}\label{sec:CGP} Let $\cF$ be an Ore FS category with associated Ore FS monoid $\cF_\infty$. Write $G=\Frac(\cF,1)$ and $H=\Frac(\cF_\infty)$. 
It was shown in \cite[Proposition 3.14]{Brothier22} that $G$ and $H$ always embed in each other. Moreover, sufficient conditions for $G$ and $H$ to be isomorphic were outlined. The relevant property is the {\it (right) colour generating property} (CGP) at a colour $a\in\cF$. An FS group $G$ has the CGP at $a$ if any element $g\in G$ can be represented as a fraction of trees whose right-side is $a$-coloured. When $G$ has the CGP at $a$ there is an isomorphism $G\simeq H$ (depending on $a$). We use the CGP in Section \ref{sec:examples-from-simple-groups}.

\subsection{Previous results}

We recall some results on FS groups that we will need in this article. They can be found in the first article on the subject \cite{Brothier22}.

\begin{theorem}\cite[Corollary 3.9, 
	Theorem 8.3]{Brothier22} \label{theo:previous-FS-groups}
	Consider an Ore FS category $\cF$ with associated FS groups $G$, $G^T$, and $G^V$. The following assertions are true.
	\begin{enumerate}
		\item By colouring monochromatic forests with a fixed colour we deduce an embedding $V\into G^V$ sending $F,T$ inside $G,G^T$, respectively.
		Up to identifying $T$ and its copy in $G^T$ we obtain that $G^T\act G^T/G$ restricts into $T\act T/F.$
		\item Let $S$ be a colour set (countable or not) and let $\tau=(t_s:s\in S)$ be a family of monochromatic trees having the same number of leaves. Consider the FS category
		\begin{align*}
			\cF_\tau=\FS\la S|t_r(r)=t_s(s):r\not=s\ra    
		\end{align*}
		and fix a colour $a\in S$. Then
		\begin{itemize}
			\item $\cF_\tau$ is left-cancellative; and
			\item the sequence $(c_n(a))_{n\geqslant1}$ of complete $a$-trees is cofinal in $\cF_\tau$ (i.e. $c_n(a)\leqslant c_m(a)$ for $n\leqslant m$ and for all $t\in\cT$ there exists $k\geqslant1$ for which $t\leqslant c_k(a)$).
		\end{itemize}
		Thus $\cF_\tau$ is an Ore FS category. 
		When $S$ is finite, the corresponding FS groups $G_\tau$, $G_\tau^T$, and $G_\tau^V$ all have topological finiteness property F$_\infty$ (i.e.~for each $d\geqslant 1$ there exists a classifying space with finite $d$-skeleton). In particular, these groups are finitely presented.
	\end{enumerate}
\end{theorem}

\section{Pointed trees and actions}\label{sec:actions}

Let $\cF$ be an FS category and $\cT$ its set of trees. In this section we construct several actions associated to $\cF$ that will be use throughout the article. Fix skein presentation $\la S|R\ra$ for $\cF$. For convenience, we will notate the free FS category by $\cUF:=\FS\la S\ra$ and refer to elements in $\cUF$ as ``free forests''. We write $P_\cF:\cUF\onto\cF$ for the canonical quotient functor and $\cUT$ for the set of trees in $\cUF$. Similarly, elements of $\cUT$ will be referred to as ``free trees''.

\subsection{Monoids of pointed trees}\label{sec:pointed-tree}

Initially we do not assume $\cF$ is left-cancellative nor has common right-multiples.

\medskip{\bf The monoid of pointed trees of $\cF$.}
Let $\cT_p$ the set of pairs $(t,\nu)$ where $t\in\cT$ is a tree and $\nu\in\Leaf(t)$ is a leaf of $t$.
We say $(t,\nu)$ is a {\it pointed tree} of $\cF$ and that $\nu$ is the {\it distinguished leaf} of $t$.
We also write $(t,j)$ for the pointed tree $(t,\nu)$ where $\nu$ is the $j$th leaf of $t$. A pointed tree can be thought diagrammatically as a tree with its distinguished leaf drawn slightly larger than the others and coloured white. For example 
\begin{align*}
	(Y_a,1)=\PointedCaretA\quad\quad\textnormal{and}\quad\quad
	(Y_b,2)=
	\PointedCaretB.
\end{align*}
We can {\it compose} any two pointed trees. Indeed, for $(t,j),(s,k)\in\cT_p$ we write $(t,j)\circ (s,k)$ for the pointed tree obtained by gluing the root of $s$ to the $j$th leaf of $t$ and keeping the distinguished leaf of $s$. For example
\begin{align*}
	(Y_a,1)\circ(Y_b,2)\quad=
	\PointedTreeCompositionA
	=
	\PointedTreeCompositionB=\quad(Y_a(Y_b\ot I),2).
\end{align*}
This defines an associative binary operation on $\cT_p$ with neutral element $(I,1)$ which consists in the trivial tree $I$ with its unique leaf distinguished. Hence, the algebraic structure $(\cT_p,\circ)$ is a monoid that we call the {\it monoid of pointed trees of $\cF$}. 

\medskip{\bf Action of pointed trees on trees.} There is a natural action of pointed trees on trees. For a pointed tree $(t,j)\in\cT_p$ and a tree $s\in\cT$ we write $(t,j)\bullet s$ for the tree formed by gluing the root of $t$ to the $j$th leaf of $s$. For example
\begin{align*}
	(Y_a,1)\bullet Y_b=
	\PointedTreeActionA
	=
	\PointedTreeActionB
	=Y_a(Y_b\otimes I).
\end{align*}
It is not hard to see that this furnishes a monoid action $\cT_p\act \cT$, i.e. $\bullet$ is compatible with $\circ$ and $(I,1)$ acts trivially.

\medskip{\bf The monoid of free pointed trees.} Similarly, we define $(\cUT_p,\circ)$ to be the monoid of pointed trees of the free FS category $\cUF$.
An element of $\cUT_p$ is called a {\it free pointed tree} and $\cUT_p$ the {\it monoid of free pointed trees} (over the colour set $S$). 

It is not hard to see that the assignment $\cF\mapsto \cT_p$ defines a functor from the category of FS categories (see Remark \ref{rem:abstract-FS}) to the category of monoids.
In particular, the quotient functor $P_\cF:\cUF\onto\cF$ provides a surjective monoid morphism 
$$P_{\cT_p}:\cUT_p\onto \cT_p, \ (t,j)\mapsto (P_\cF(t),j).$$
Moreover, note that $\cT_p$ is never finitely generated. We will now define a submonoid of $\cUT_p$ that is often finitely generated.

\medskip{\bf The monoid of narrow free pointed trees.}
Let $\cNUT_p$ be the submonoid of $\cUT_p$ generated by the free pointed carets $(Y_a,j)$ with $a\in S, j\in\{1,2\}.$
An element of $\cNUT_p$ is called a {\it narrow free pointed tree} and $\cNUT_p$ the {\it monoid of narrow free pointed trees}.
Observe that $\cNUT_p$ is isomorphic to the free monoid over the set $S\sqcup S$, so is finitely generated when $S$ is finite.
Moreover, a free pointed tree $(t,\nu)\in \cUT_p$ is narrow when all the interior vertices of $t$ are on the path from the root of $t$ to the leaf $\nu$.

\medskip{\bf Prune morphism.}
Consider a free pointed tree $(t,\nu)\in \cUT_p$. Let $t'\leqslant t$ be the rooted subtree of $t$ whose interior vertices are those vertices on the path from the root to $\nu$ in $t$. 
We set $\prune(t,\nu):=(t',\nu)$ and note that this free pointed tree is narrow.
This operation defines a surjective monoid morphism
$$\prune:\cUT_p\onto \cNUT_p$$
which restricts to the identity on $\cNUT_p$.

\begin{remark}
	Consider the FS category $\cF=\FS\la a,b|a_1a_2=b_1b_1\ra$ with pointed tree monoid $\cT_p$. Consider the free pointed trees
	\begin{align*}
		(t_0,1)=\NarrowA\quad\textnormal{and}\quad (t_1,1)=\NarrowB.
	\end{align*} 
	We have that $t_0$ and $t_1$ are equivalent in $\cF$.
	However, $(t_0,1)$ is narrow, while $(t_1,1)$ is not.
	Moreover, the image inside $\cT_p$ of $\prune(t_0,1)$ and $\prune(t_1,1)$ are not equal. 
	Hence, a notion of narrow pointed tree (dropping freeness) and prune morphism is not well-defined for most FS categories. 
	However, these notions may become useful for {\it shape-preserving} FS categories (see Example \ref{ex:non-faithful-action} and \cite[Section 2]{Brothier23c}). 
\end{remark}

\medskip{\bf Pruned relations and pruned equivalence.}
Recall that $\la S|R\ra$ is a skein presentation for the FS category $\cF$. 
Let $(u,v)\in R$ be a skein relation over $S$ and let $n$ be the common number of leaves for $u,v\in\cUT$.
We define a so-called {\it prune relation}
$$(\prune(u,j),\prune(v,j)) \text{ for each leaf } 1\leqslant j\leqslant n.$$
Denote by $\sim_{\prune}$ the congruence relation on the monoid $(\cNUT_p,\circ)$ generated by all the prune relations. If $(t,j)\sim_{\prune}(s,k)$, then we say that $(t,j)$ is {\it pruned equivalent} to $(s,k)$.
The quotient $\cNUT_p/\sim_{\prune}$ is again a monoid with generators the pointed carets and relations all the prune relations.

\subsection{Canonical group action}\label{sec:canonical-action}
In \cite[Section 6]{Brothier22} we defined for each Ore FS category $\cF$ a set $Q$ and the so-called canonical group action $\alpha:\Frac(\cF^V,1)\act Q$.
We now extend the construction of $Q$ to {\it any} FS category $\cF$ (relaxing the existence of a calculus of fractions). 
With this we construct a monoid action $\beta:\cT_p\act Q$ and a category action of $\cF$ for {\it any} FS category $\cF$.
We will then show compatibility of $\alpha$ and $\beta$ when $\cF$ is an {\it Ore} category.

\subsubsection{Construction of the space $Q$}\label{sec:Q-space}
Recall that $\cF$ is an FS category. 
Consider the monoid of pointed trees $\cT_p$ as a set. 
Let $(t,j)\in\cT_p$ be a pointed tree, whose tree $t$ has $n$ leaves, and let $f\in\cF$ be a forest with $n$ roots. We form a new pointed tree $(tf, j^f)\in\cT_p$ where $j^f$ is the first leaf of the $j$th tree of $f$.
In other words, $(tf,j^f)$ is the pointed tree obtained by stacking $f$ on top of $t$ and letting the distinguished leaf of $t$ go up to the left until it reaches a leaf of $f$.
Define $\sim_Q$ to be the equivalence relation on the set $\cT_p$ (that is not required to be compatible with $\circ$) generated by 
$$(t,j)\sim_Q (tf, j^f)$$
for all pointed trees $(t,j)\in\cT_p$ and forests $f$ composable with $t$. We write $[t,j]$ for the class of $(t,j)$ and write $Q$ or $Q_\cF$ for the quotient set $\cT_p/\sim_Q$.
We call $Q$ the {\it $Q$-space} of $\cF$ and denote by $\End(Q)$ the monoid of all maps from $Q$ to $Q$.

\begin{observation}\label{obs:prune}
	For any pointed tree $(t,j)\in\cT_p$ and any of its free representatives $(t_0,j)\in\cUT_p$, we have that 
	$$P_{\cT_p}\circ\prune(t_0,j)\sim_Q (t,j).$$
	Hence, we may always assume that a class $[t,j]\in Q$ is represented by a narrow free pointed tree.
\end{observation}

\subsubsection{Canonical monoid action}
We now define how the monoid of pointed trees $\cT_p$ acts on the $Q$-space of $\cF$.

\begin{proposition}\label{prop:pointed-tree-action}
	The composition of pointed trees factorises into a monoid action 
	$$\beta:\cT_p\to\End(Q),\ \beta(t,j)([s,k]):= [(t,j)\circ (s,k)].$$
	We call $\beta$ the \emph{canonical (monoid) action} of the pointed trees of $\cF$.
	The action $\beta$ factors through the quotient monoid $\cNUT_p/\sim_{\prune}$ and 
	$$\beta(t,j) = \beta(P_{\cT_p}\circ\prune(t_0,j))$$
	where $(t,j)\in\cT_p$ and $(t_0,j)\in\cUT_p$ is one of its free representatives.
	Hence, we have the following commutative diagram of monoid morphisms:
	$$\begin{tikzcd}
	\cUT_p \arrow[r, "\prune", shift left=.2cm] \arrow[d, two heads,"P_{\cT_p}"] & \cNUT_p \arrow[l , hook]  \arrow[d, two heads] \\ 
	\cT_p \arrow[d, "\beta"]& \cNUT_p/\sim_{\prune} \arrow[dl, bend left] \\
	\End(Q) &  
	\end{tikzcd}$$
	In particular, the monoid $\beta(\cT_p)$ is a quotient of $\cNUT_p/\sim_{\prune}$ and thus generated by pointed carets.
\end{proposition}

\begin{proof}
	Consider a skein relation $(u,v)\in R$ and $q=[s,k]\in Q.$
		Since the monoid structure of $\cT_p$ is well-defined we have that $(u,j)\circ (s,k)$ and $(v,j)\circ (s,k)$ are equal in $\cT_p$ and thus their class in $Q$ are equal.
		This implies that $\beta:\cT_p\to\End(Q)$ is a well-defined monoid action.
		The rest of the proposition follows easily from the fact that $\beta$ is well-defined, $\prune:\cUT_p\to\cNUT_p$ is a monoid morphism, and Observation \ref{obs:prune}.
\end{proof}

\subsubsection{Canonical category action}
Using the monoid action $\beta:\cT_p\act Q$ constructed above we will build an action of the whole category $\cF$ in a rather obvious way using matrices.
Given a tree $t$ with leaf set $\Leaf(t)$ we define 
$$\beta(t):Q\to Q^{\Leaf(t)},\ q\mapsto (\ell\mapsto \beta(t,\ell)(q)).$$
Hence, if $t\in\cT$ has $n$ leaves, then $\beta(t)$ defines a map from $Q$ to $Q^n$ with $j$th component being $\beta(t,j).$ This map extends to all forests 
	such as
	$$\beta(f):Q^{\Root(f)}\to Q^{\Leaf(f)}, \ \beta(f)(q_r)(\ell):=\beta(f_r,\ell)q_r$$
	where $f_r$ is the $r$th tree of $f$, $\ell$ a leaf of $f_r$, and $q=(q_r)_r\in Q^{\Root(f)}.$

It is rather obvious from the definition that $\beta$ preserves composition: $\beta(p\circ q)=\beta(p)\circ \beta(q)$ for $p,q\in\cF$ composable forests.
We deduce that $\beta$ defines a functor $\beta:\cF\to \cQ$ where $\cQ$ is the category of functions between powers of $Q$.

\subsubsection{Canonical group action}\label{sec:group-action} 
We now define the action of an FS group on the $Q$-space of its category as described above. We also relate this to previous constructions of the canonical group action of an Ore FS category from \cite[Section 6]{Brothier22}. 
\begin{center}
	{\bf From now on assume that $\cF$ is an Ore FS category.}    
\end{center}
Let $Q$ be the $Q$-space of $\cF$ as constructed in Section \ref{sec:Q-space}. Since $\cF$ is an Ore category, we may define its FS groups $G$, $G^T$, and $G^V$. Given $g:=[t\pi,s]\in G^V$ and $q:=[s,j]\in Q$, we set
$$\al(g)(q)=g\cdot q = [t\pi,s]\cdot [s,j]:=[t,\pi(j)].$$
Having common right-multiples permits to extend this formula and left-cancellation gives well-definition. 
This gives the {\it canonical group action} $\al:G^V\act Q$ of $\cF$.

\medskip{\bf Jones' technology.} The above action is conjugate to the canonical group action appearing in \cite[Section 6]{Brothier22} constructed from Jones' technology via the (covariant) monoidal functor
$\Phi:(\cF,\otimes)\to(\Set,\sqcup)$ (where $(\Set,\sqcup)$ is the monoidal ``category'' of sets, functions, and disjoint union)
where $\Phi(f):\Root(f)\to\Leaf(f)$ sends the $j$th root of $f$ to the first leaf of its $j$th tree. 
Identifying $\Root(f)$ with $\Leaf(t)$ when $t$ and $f$ are composable, we may interpret
$Q$ as the direct limit $\varinjlim_{t\in\cT}\Leaf(t)$ 
with maps $\Phi(f):\Leaf(t)\to \Leaf(tf).$ 

\medskip{\bf Homogeneous action.} Note that the canonical group action $\al:G^V\act Q$ restricted to $G^T$ is conjugate to the homogeneous action $G^T\act G^T/G$ 
via the map $[t,j]\mapsto [t\pi,t]G$ where $\pi(1)=j$.

\medskip{\bf Compatibility with the canonical monoid action.}
We now relate the canonical group and monoid actions for $\cF$.
Consider an element $g=[t,s]\in G$. 
By having common right-multiples for $\cF$, any $q\in Q$ can be written in the form $[(s,j)\circ  (x,k)]$ for some leaf $j$ of $s$ and some pointed tree $(x,k)\in\cT_p$.
Now, by definition of $\al$, we have $$g\cdot[(s,j)\circ  (x,k)] = [(t,j)\circ (x,k)].$$
This operation was done in two steps:
$$[(s,j)\circ (x,k)] \mapsto [x,k] \mapsto [(t,j)\circ (x,k)].$$
The first step is the inverse of $\beta(s,j)$ (hence a map from the range of $\beta(s,j)$ to $Q$) and the second the map $\beta(t,j)$.
Hence, $\alpha(g)$ restricted to $\beta(s,j)(Q)$ is the map 
$$\beta(t,j)\circ \beta(s,j)^{-1}.$$
More generally, in the $V$-case, we may have a non-trivial permutation $\pi$ of the leaves and replace in the formula of above $\beta(t,j)$ by $\beta(t,\pi(j)).$

\begin{remark}\label{rem:prefix}
	The canonical group action $\alpha:G^V\act Q$ acts locally by changing prefixes, where here ``prefix" means ``pointed tree".
	Hence, for $t$ and $s$ trees with $n$ leaves
	$$\alpha([t,s])=\alpha(t\circ s^{-1}):\sqcup_{j=1}^n \beta(s,j)(Q)\to \sqcup_{j=1}^n\beta(t,j)(Q)$$ 
	is similar to the map 
	$$\beta(t)\circ \beta(s)^{-1} :\prod_{j=1}^n \beta(s,j)(Q)\to \prod_{j=1}^n\beta(t,j)(Q).$$ However, these two maps are rather different. The first realises a bijection from $Q$ to $Q$ while the second realise a bijections between two subsets of $Q^n$.
\end{remark}

\subsection{Completion of the canonical group action and some properties}
As above we keep the assumption that $\cF$ is an {\it Ore} FS category.

\subsubsection{Construction of a profinite set}
The canonical group action $G^V\act Q$ is analogous to the action of Thompson's group $V$ on {\it finitely supported} binary sequences.
We will now construct a completion $\fC$ of $Q$ and an action $G^V\act\fC$ analogous to $V$ acting on {\it all infinite} binary sequences.
We construct a topological completion of $Q$ by reversing the arrows in the functor $\Phi:\cF\to\Set$ from Section \ref{sec:group-action}.
Define the functor $$\Psi:\cF\to\Set$$ such that $\Psi(f):\Leaf(f)\to\Root(f)$ sends a leaf $\ell$ of $f$ to the root $r$ of the tree containing $\ell$. This functor is monoidal, hence
Jones' technology provides an action 
$G^V\act \fC$
on the inverse limit $\fC:=\varprojlim_{t\in\cT}\Leaf(t)$.
Concretely, $\fC$ is the profinite set consisting of all tuples
$$x=(x(t))_{t\in\cT}\in\prod_{t\in\cT}\Leaf(t)$$ satisfying that if $t\leqslant s$, then $x(s)$ is a descendent of $x(t)$.
If $g=[t\pi,s]\in G^V,$ then 
\begin{align*}
	(g\cdot x)(t)=\pi(x(s))
\end{align*}
where now $\Leaf(t)$ and $\Leaf(s)$ are identified with $\{1,\dots,n\}$ with $n$ being the common number of leaves of $t$ and $s$.
We will see below that $G^V\act \fC$ is an extension by continuity of the canonical group action $\al:G^V\act Q$ described in Section \ref{sec:group-action}. For this reason we will also notate this action $\al:G^V\act\fC$.

\subsubsection{Embedding of $Q$ inside $\fC$.}\label{sec:Q-embedding}
To each $q$ in $Q$ we want to define an element of $\fC$. 
To do this we need to assign to each tree $s\in\cT$ one of its leaves $\ell_s$ depending on $q$ so that the family $(\ell_s)_{s\in\cT}$ defines an element of $\fC$.
We do this in a canonical way using the two functors $\Phi$ and $\Psi$.
Let $q:=[t,j]\in Q$ and fix a tree $s\in\cT$. 
By having common right-multiples there exist some forests $f,h\in\cF$ such that $tf=sh$.
Noting that $\Phi(f)(j)=j^f$ we have $$q=[tf,\Phi(f)(j)]=[sh,\Phi(f)(j)].$$
Consider now the pointed tree $(sh,\Phi(f)(j))$.
Apply the functor $\Psi$ to reduce $sh$ to $s$ and set 
$$\ell_s:= \Psi(h)((\Phi(f)(j))).$$
One may check that $\ell_s$ does not depend on the choice of the representative of $q$ nor on the choice of forests $f,h$.
Moreover, one can check that the family $(\ell_s)_{s\in\cT}$ is indeed in $\fC$, i.e.~it satisfies that $\Psi(k)(\ell_{zk})=\ell_z$ for all trees $z\in\cT$ and forest $k\in\cF$. Informally, we have grown $t$ into $tf$ until it is larger than $s$ and let the distinguished leaf $j$ of $t$ grow up left. We then reduce $tf$ until it equals $s$ and letting the distinguished leaf of $tf$ go down in the unique possible way.

\subsubsection{Topology on $\fC$}
We equip $\fC$ with its topology inherited from the inverse limit $\varprojlim_{t\in\cT}\Leaf(t)$ taking $\Leaf(t)$ to have the discrete topology.
It is a profinite topological space and thus satisfies all the axioms of being a Cantor space except that it may or may not be second countable.
If $t\in\cT$ is a tree and $p_t:\fC\onto \Leaf(t)$ is the canonical projection, then the preimage of a leaf $j$ is open and closed ({\it clopen} for short).
We write $\Cone(t,j)$ for this preimage and call it a {\it cone}.
Hence, $\Cone(t,j)$ is the set of all $x=(x(s))_{s\in\cT}$ such that if $t\leqslant s,$ then $x(s)$ is a descendent of $j.$ We note that $\Cone(t,j)\cap Q$ is the range of $\beta(t,j)$: all classes of pointed trees of the form $[(t,j)\circ (z,k)].$
This justifies the terminology ``cone". 

\begin{example}\label{ex:non-cantor-C-space}
	If $\cF$ has countably many colours, then $\cT_p$ is countable and $\fC$ is a Cantor space since it is profinite and since the cones defines a countable basis of clopen spaces.
	If the colour-set is uncountable, then $\fC$ may be second countable or not as illustrated below.
	
	\begin{enumerate}
		\item Consider any non-empty set $S$, a natural number $n\geqslant 2$, and a family of monochromatic trees $(t_s:s\in S)$ indexed by $S$ so that each $t_s$ has $n$ leaves. Define a skein presentation with colour set $S$ and skein relations $t_r(r)\sim t_s(s)$ for all $r,s\in S$. This provides an Ore FS category $\cF$ by \cite[Theorem 8.3]{Brothier22} (see also item (2) of Theorem \ref{theo:previous-FS-groups}). Fix a colour $a\in S$. We have that any tree in $\cF$ is smaller than a tree with all vertices coloured by $a$. This implies that $Q$ and $\fC$ are in bijection with the usual spaces of finitely supported binary sequences and all infinite binary sequences. Hence, for these class of examples, whatever the cardinality of $S$ is, we obtain a Cantor space for $\fC$.
		\item Consider now $S=(0,1)$ the real open unit interval. For each $s\in S$ define the partition $[0,s)\sqcup [s,1)$ of the half-open real interval $[0,1)$. By identifying this partition with the caret $Y_s$ we deduce the skein relations $$Y_s\circ (Y_r\ot I) = Y_{sr} \circ (I\ot Y_x)$$ where $s=sr+ x(1-sr)$. This defines an Ore FS category with FS group $G$ isomorphic to the group of {\it all} piecewise linear order-preserving homeomorphisms of $[0,1]$. The $Q$-space is now the real open interval $(0,1)$. 
		One can prove that $\fC$ is not second countable similar to a proof that the Sorgenfrey line is not second countable.
	\end{enumerate}
\end{example}

\subsubsection{A second construction using uniform structures}
Consider an Ore FS category $\cF$ and the associated space $Q$.
We provide another equivalent definition for $\fC$ using Bourbaki's uniform structures rather than an inverse limit. 

For any tree $t\in\cT$ we consider the set 
$$U_t:=\{(q,q')\in Q\times Q:\ q,q'\in \Cone(t,\ell) \text{ for a leaf } \ell\}.$$
Observe that $U_t$ is symmetric and $U_t=U_t\circ U_t$ 
in the sense that if $(q,q'')\in Q\times Q$ satisfies that there exists $q'\in Q$ so that $(q,q')\in U_t$ and $(q',q'')\in U_t$, then $(q,q'')\in U_t.$ 
Indeed, this follows from the fact: $\Cone(t,j)\cap \Cone(t,k)=\varnothing$ when $j\neq k.$
Moreover, having common right-multiples for $\cF$ implies that given $t,s\in \cT$ there exists $z\in \cT$ so that $U_z\subset U_t\circ U_s.$
From there we easily deduce that the collection $\mathcal B$ of all the $U_t$ with $t\in\cT$ forms a fundamental system of entourage of the set $Q$ in the sense of \cite[Chapter 2]{BourbakiTG}.

Therefore, $\mathcal B$ defines a unique uniform structure $\cU$ on $Q$. Its associated topology is by definition the one generated by all cones $\Cone(t,j)$ and thus coincides with the topology that we previously defined.
Hence, we now have a uniform structure $\cU$ on $Q.$
The completion of $(Q,\cU)$ obtained by considering minimal Cauchy filters is isomorphic to $(\fC,\ti \cU)$ where $\ti \cU$ is similarly defined than $\cU$ using trees and cones.
This justifies to say that ``$\fC$ is the completion of $Q$".

When $\cF$ admits finitely (or even countably) many colours, then $Q$ admits countable basis of neighbourhoods and thus we can complete $Q$ using Cauchy sequences rather than Cauchy filters. 
We immediately obtained the following fact.

\begin{observation}
	The action $G^V\act \fC$ is the unique extension by continuity of $G^V\act Q.$ Similarly, we may extend the monoid action of $\cT_p$ and the category action of $\cF$ to $\fC$.
\end{observation}

\subsubsection{Order on the canonical space}
\medskip{\bf Total order.}
Having common right-multiples implies that for all $q_1,q_2\in Q$ there exists a tree $t$ and leaves $j_1,j_2\in\Leaf(t)$ satisfying $q_1=[t,j_1]$ and $q_2=[t,j_2]$.
Setting $q_1\leqslant q_2$ when $j_1\leqslant j_2$ defines a total order on $Q$.
This order extends to $\fC$.
Given two elements $x,y\in\fC$ we set $x\leqslant y$ when eventually $x(t)\leqslant y(t)$. 
If $\cF$ has finitely many colours, then this is equivalent to say that $x(t)\leqslant y(t)$ for all but finitely many trees $t\in\cT$.
This is again a total order on $\fC$.

\medskip{\bf Order topology and endpoints.}
Observe that $\Cone(t,j)$ is the half-open interval
$$\{q\in \fC:\ [t,j]\leqslant q < [t,j+1]\}$$
for the total order of $\fC.$
This implies that the topology of $\fC$ is the order topology given by $\leqslant$.
Since $\fC$ is compact and its topology is the order topology we deduce that it has two end points or ``extremities'' $o$ and $\omega$ where $$o=\min\fC \text{ and } \omega=\max\fC.$$

\medskip{\bf Two copies of $Q$ inside $\fC$.}
Similarly to $\fC$, a cone $\Cone(t,j)$ has two endpoints
$[t,j]=\min\Cone(t,j)$ and $\max\Cone(t,j).$ 
Hence, as an interval, $\Cone(t,j)$ is both closed and open (we say they are clopen intervals).
Given any $x\in Q\setminus\{o\}$ there exists a pointed tree $(t,j)$ such that $x=[t,j+1]$.
Then we may write $x^+$ for $[t,j+1]$ inside $\fC$ (for the embedding described in Section \ref{sec:Q-embedding}) and $x^-$ for $\max\Cone(t,j).$
We have that $x^-\leqslant x^+$ and in fact $x^+$ is the successor of $x^-$ for the total order of $\fC$ (i.e.~if $x^-\leqslant z\leqslant x^+$, then $z=x^-$ or $z=x^+$).
We obtain two copies of $Q\setminus\{o\}$ inside $\fC$. 
We write $Q^+$ for the collection of all $\min C$ with $C$ any cones and $Q^-$ for the collection of all $\max C.$
Note that $Q^+$ corresponds to all the $x^+$ defined as above plus the first point $o$ of $\fC$ while $Q^-$ is obtained by adding the last point $\omega$ of $\fC.$
We always have that $Q^+$ and $Q^-$ are disjoint.
We will continue to identify $Q$ with its copy $Q^+$ inside $\fC.$

\subsubsection{Summary}\label{sec:summary} 
We collect a number of facts and properties on the canonical actions. 
They all follow easily from \cite[Section 6]{Brothier22} and the discussions of this section.
\begin{proposition}\label{prop:action-facts}
	Let $\cF$ be an Ore FS category with FS groups $G$, $G^T$, $G^V$, canonical group actions $\alpha:G^V\act Q$ and $\alpha:G^V\act \fC$, monoid of pointed trees $\cT_p$, and canonical monoid action $\beta:\cT_p\act \fC.$
	The following assertions are true.
	\begin{enumerate}
		\item The space $(Q,\leqslant)$ is a totally ordered infinite set. If $\cF$ has countably many colours, then $Q$ is countable. 
		\item The action $\al:G^V\act Q$ restricts into the homogeneous action $G^T\act G^T/G.$
		\item The space $\fC$ is profinite, has a total order $\leqslant$, its topology is the order topology. It contains two disjoint copies of $Q$ denoted $Q^\pm$ and each copy is dense inside $\fC.$
		\item 
		The collection of cones $\Cone(t,j)$ forms a clopen basis of the topology. Additionally, a cone is a clopen interval for $\leqslant$ of the form $[x^+,y^-]$ with $x,y\in Q.$
		If $\cF$ has countably many colours, then the space $\fC$ is second countable (and thus is a Cantor space).
		\item The group $G$ acts in an order preserving way on $(\fC,\leqslant).$ Conversely $G$ is the subgroup of $g\in G^V$ that preserves $\leqslant.$
		\item The group $G^V$ acts by homeomorphisms on $\fC$.
		\item The kernel of $\al:G^V\act \fC$ is equal to the kernel of $\al:G\act Q.$ In particular, it is a subgroup of $G$ that is normalised by $G^V.$ Moreover, it is equal to $\cap_{h\in G^T} hGh^{-1}.$
		\item If $(t,j)\in \cT_p$, then $\beta(t,j)$ is an order preserving homeomorphism from $\fC$ to $\Cone(t,j).$ 
		\item If $g=[t\pi, s]$, then for each leaf $j$ of $s$ we have that 
		$$\alpha(g)\restriction_{\Cone(s,j)}=\beta(t,\pi(j)) \circ \beta(s,j)^{-1}$$
		realises an order preserving homeomorphism from $\Cone(s,j)$ to $\Cone(t,\pi(j)).$
	\end{enumerate}
\end{proposition}

\begin{example}
	When $\cF$ is the Ore FS category of monochromatic binary forests, we have 
	$(G,G^T,G^V)=(F,T,V)$, $\fC=\{0,1\}^\N$, $Q=Q^+$ the finitely supported sequences, $Q^-$ the sequences eventually constant equal to $1$, cones are of the form $C_w=w\cdot \fC$ with $w$ a finite word in alphabet $\{0,1\}$, and $\leqslant$ is the lexicographic order. 
	If $g\in V$, then it acts locally on $\fC$ by changing prefixes: $C_w\to C_v, w\cdot z\mapsto v\cdot z.$
\end{example}

The following example shows that when $\cF$ is not a Ore category we may have pathological canonical actions.

\begin{example}
	Consider the FS category $\cF=\FS\la a|a_1a_1=a_1a_2\ra$.
	An easy induction shows that any two free trees $t,s\in\FS\la a\ra$ with same number of leaves are equal in $\cF.$
	The category $\cF$ has common right-multiples (since it is the quotient of a right-Ore category) but is not left-cancellative since $Y(Y\ot I)=Y(I\ot Y)$, but $Y\ot I\not=I\ot Y$.
	Observe that
	\begin{align*}
		(Y,2)\sim_{grow}(Y(I\ot Y),2)
		\sim_{skein}(Y(Y\ot I),2)
	\end{align*}
	and hence
	$$ \Cone(Y,2) \ni [Y,2]=[Y(Y\ot I),2]\in \Cone(Y,1)$$
	and thus $\Cone(Y,1)\cap \Cone(Y,2)\neq\varnothing.$
	In fact, the space $Q$ has only two points that are $[Y,1]$ and $[Y,2]$.
	In particular, $\beta(Y,2):Q\to Q$ has for range the singleton $[Y,2]$ and is thus not injective.
\end{example}

\section{Dynamical characterisation of simplicity}\label{sec:dynamical-characterisation}

Let $\cF$ be an Ore FS category with FS groups $G$, $G^T$, $G^V$, monoid of pointed trees $\cT_p$, canonical group action $\alpha:G^V\act\fC$, and canonical monoid action $\beta:\cT_p\act\fC$ as discussed in the previous section. In this section we characterise simplicity of the derived subgroups $[G,G]$, $[G^T,G^T]$, and $[G^V,G^V]$ in terms of the canonical group action. 
For convenience we write $g\cdot x$ rather than $\alpha(g)(x)$ for $g\in G^V$ and $x\in\fC$ and warn the reader that $g\cdot x=x$ for all $x\in\fC$ implies $g=e$ only if the canonical group action is {\it faithful}. Moreover, we may write $g\restriction_A$ for the restriction of $\al(g)$ (or any map $g:\fC\to\fC$) to a subset $A\subset\fC.$

\subsection{Groups of germs and group $K$}\label{sec:germs}
Recall that $o$ and $\omega$ denote the smallest and largest elements of $\fC$ for its total order $\leqslant.$
If $g\in G$, then $g\cdot o=o$ and $g\cdot \omega=\omega.$
\begin{definition}\label{def:K-groups} Let $\cF$, $G^V$, and $\al:G^V\act\fC$ be as above.
	\begin{enumerate}
		\item For $g\in G^V$ we write $\supp(g)\subset\fC$ for the {\bf closure} (in the profinite topology) of $$\{x\in\fC:\ g\cdot x\not=x\}.$$
		This is called the {\it support} of $g$. We say $g$ has {\it proper support} if $\supp(g)$ is a proper subset of $\fC$. 
		\item If $C=\Cone(z,j)$, let $G_C$ be the set of all $g\in G$ with $\supp(g)\subset C$. The map
		\begin{align*}
		\eta_C:G\to G_C \  [t,s]\mapsto[(z,j)\bullet t,(z,j)\bullet s],
		\end{align*}  
		where $\bullet$ is the action of pointed trees on trees, is an injective homomorphism.
		\item 
		Let $K_o\lhd G$ (resp. $K_\omega\lhd G$) be the normal subgroup of elements acting trivially on a neighbourhood of $o$ (resp.~$\omega$).
		\item We set $K:=K_o\cap K_\omega\lhd G$, so $k\in K$ acts trivially on a neighbourhood of $o$ and $\omega$.
		\item Let $\Germ(G\act\fC,o):=G/K_o$ and $\Germ(G\act\fC,\omega):=G/K_\omega$ be the {\it germ groups} for the canonical group action $\alpha:G\act \fC$ at the points $o$ and $\omega$.
	\end{enumerate}
\end{definition}

\begin{remark}\label{rem:K-remarks}
	\begin{enumerate}
		\item The subgroups $K_o,K_\omega,K,[K,K],[G,G]$ are not finitely generated.
		\item If $\cF=\FS\la a\ra,$ then $G\simeq F$ is Thompson's group and $\fC$ is the usual Cantor space $\{0,1\}^\N$ with $o$ and $\omega$ being the constant binary strings $0^\infty$ and $1^\infty$. 
		In this case, the derived subgroup $[F,F]$ of $F$ is the subgroup of the elements acting trivially around $o$ and $\omega$: this is the group $K$ defined above.
		The groups of germs of $F$ at $o$ and at $\omega$ are isomorphic to $\Z$. If $F$ is interpreted as a group of homeomorphisms of $[0,1]$, then $K$ corresponds to those with slope $1$ at the real points $0$ and $1.$
		\item If $\cF=\FS\la a,b| a_1a_2=b_1b_1\ra,$ then $G\simeq F_\tau$ is the description of Cleary's group given by Burillo, Nucinkis, and Reeves \cite{Cleary00,Burillo-Nucinkis-Reeves21}. 
		When $F_\tau$ is interpreted as a group of homeomorphisms of $[0,1]$, $K$ is the subgroup of elements with trivial slopes at $0$ and $1$. However, $[F_\tau,F_\tau]$ is the index $2$ subgroup of $K$ of elements $g=[t,s]$ that have same number of $a$-carets modulo $2$ in $t$ and $s$.
		The group $K$ is known as the group of elements with ``bounded support'' in \cite[Definition 5.1]{Burillo-Nucinkis-Reeves21} and is denoted $F_\tau^c$.
	\end{enumerate}
\end{remark}

\subsection{Presentations of germ groups}\label{sec:germ-presentation}

We briefly outline how to calculate certain germ groups for a {\it faithful} canonical group action and discuss this in further detail in \cite{Brothier-Seelig24}. Fix an Ore FS category $\cF=\FS\la S|R\ra$ with FS group $G$ and normal subgroups $K_o,K_\omega\lhd G$ as in Definition \ref{def:K-groups}. Sending a free forest $f=(f_1,f_2,\dots,f_n)\in\FS\la S\ra$ to $\prune(f_1,1)$ furnishes a surjective (non-monoidal) functor
\begin{align*}
	\gamma_o:\cF\to\Mon\la S|\prune(u,1)=\prune(v,1):(u,v)\in R\ra
\end{align*}
and hence a surjective group morphism
\begin{align*}
	\gamma_o:G\to\Gr\la S|\prune(u,1)=\prune(v,1):(u,v)\in R\ra=:\Gamma_o.
\end{align*}
If the canonical group action is faithful, then elements of $K_o$ can be written as $[tp,tq]$ where $t\in\cT$ and the first trees of $p,q\in\cF$ are trivial. It follows that $\gamma_o$ factors through $\gamma_o:G/K_o\to\Gamma_o$. 
One can prove that, if $g=[t,s]$ satisfies that $\prune(t,1)=\prune(s,1)$, then $g\in K_o$. 
	Therefore, $G/K_o\to \Gamma_o$ is injective and thus is a group isomorphism.
Similar can be done for pruning at last leaves to define the presented group $\Gamma_\omega$ for the germ group $G/K_\omega$. It follows that when the canonical group action is faithful, we can easily read the germ groups at $o$ and $\omega$ from a skein presentation. As a quick example, consider the Ore FS category $\cF=\FS\la a,b|a_1a_1=b_1b_2\ra$ that gives Cleary's irrational slope Thompson group. In Section \ref{sec:examples-via-dynamics} we will demonstrate the canonical group action $\al:G\act\fC$ for $\cF$ is faithful. Hence $G/K_o\simeq\Gamma_o= \Gr\la a,b|aa=b\ra\simeq\Z$ and $G/K_\omega\simeq\Gamma_\omega=\Gr\la a,b|a=bb\ra\simeq\Z$.

\subsection{Faithfulness implies simplicity}

In this section we show faithfulness of the canonical group action of a given Ore FS category implies ``good simplicity properties'' of its FS groups, with a caveat in the $F$-case. In particular, we deduce when the statement:
\begin{align}\label{eqn:simple}
	\textit{Any non-trivial subgroup $\Lambda\subset \Gamma$ normalised by $[\Gamma,\Gamma]$ contains $[\Gamma,\Gamma]$.}
\end{align}
holds true for $\Gamma$ being $K$, $G$, $G^T$, or $G^V$. If (\ref{eqn:simple}) holds for $\Gamma$, then it is standard that
\begin{itemize}
	\item $[\Gamma,\Gamma]$ is simple; and
	\item all proper quotients of $\Gamma$ are abelian.
\end{itemize}
The proof scheme is due to Higman \cite{Higman54} and Epstein \cite{Epstein70} and has been used to show many Thompson-like groups are simple \cite{Brown87,Stein92,Brin04,Matui15,Kim-Koberda-Lodha19,Burillo-Nucinkis-Reeves21}. 

\medskip{\bf Notation and assumptions.} For the remainder of this section we consider an Ore FS category $\cF$ with FS groups $G$, $G^T$, $G^V$, canonical group action $\al:G^V\act\fC$, and canonical monoid action $\beta:\cT_p\act\fC$. For the remainder of the section:

\begin{center}
{\bf Assume the canonical group action of $\cF$ is faithful.}
\end{center}

Hence we identify $G^V$ with a subgroup of the homeomorphism group $\Homeo(\fC)$. We write $[g,h]=ghg^{-1}h^{-1}$ for the commutator of two group elements $g$ and $h$. So as not to confuse this notation with fractions of trees $[t,s]$, we reserve symbols $g,h,k,l,m,n$ for elements of $G^V$ and $r,s,t,z$ for trees in $\cF$. 
We also write $g^h=hgh^{-1}$ and $g\restriction_A$ for the restriction of $g:\fC\to\fC$ to a subset $A\subset\fC$.

We will use the following variant of Higman's theorem appearing in \cite[Section 2]{Kim-Koberda-Lodha19}.

\begin{theorem}\label{theo:well-transitive}
An action by homeomorphisms $\Gamma\act X$ is called {\it compact-open transitive} (CO-transitive for short) if for each proper compact $A\subset X$ and each non-empty open $B\subset X$ there exists $\gamma\in\Gamma$ with $\gamma(A)\subset B$. 
If $X$ is a non-compact Hausdorff space and the action $\Gamma\act X$ is faithful, CO-transitive, and acts by compactly supported homeomorphisms, then (\ref{eqn:simple}) holds for $\Gamma$.
\end{theorem}

\subsubsection{Simplicity: $F$-case} 
Here are some useful facts.

\begin{proposition}\label{prop:simplicity-facts}
Let $g,h,k\in G^V$,	$A\subset\fC$ non-empty clopen, and $C\subset\fC$ a proper cone.
	\begin{enumerate}
		\item If $h(C)\cap C=\varnothing$ and $k^{-1}(\supp(g)\cup A)\subset C$, then $g\restriction_A=[g,h^k]\restriction_A$.
		\item If $k(\supp(g))\cap\supp(h)=\varnothing$, then $[g,h]=[[g,k],h]$.
	\end{enumerate}
\end{proposition}
\begin{proof}
	(1) It follows that
	\begin{align*}
		(h^k)^{-1}(\supp(g)\cup A)\cap(\supp(g)\cup A)=\varnothing.
	\end{align*}
	In particular $g$ acts trivially on $(h^k)^{-1}(A)$. Thus, for all $a\in A$
	\begin{align*}
		[g,h^k](a)&=gh^kg^{-1}(h^k)^{-1}(a)
		=gh^k(h^k)^{-1}(a)
		=g(a)
	\end{align*}
	which gives the result.
	
	(2) It follows that $\supp(g^k)$ and $\supp(h)$ are disjoint, hence $g^k$ and $h$ commute. We have
	\begin{align*}
		[[g,k],h]&=g(g^{-1})^khg^kg^{-1}h^{-1}
		=gh(g^{-1})^kg^kg^{-1}h^{-1}
		=ghg^{-1}h^{-1}
		=[g,h]
	\end{align*}
	as required.
\end{proof}
\begin{theorem}\label{theo:simple-F} Let $\cF$ be an Ore FS category with FS group $G$, canonical group action $\al:G\act\fC$, and normal subgroups $K_o,K_\omega,K\lhd G$. Suppose $\al:G\act\fC$ is faithful.
	\begin{enumerate}
		\item For any non-trivial subgroup $\Lambda\subset K$ normalised by $[K,K]$ we have $[K,K]\subset\Lambda$. 
		\item If the germ groups $G/K_o$ and $G/K_\omega$ are abelian, then item (1) is true if one replaces $K$ with $G$.
	\end{enumerate}
\end{theorem}
\begin{proof}
	(1) Equip $\fC^*:=\fC\setminus\{o,\omega\}$ with the subspace topology, which we note is Hausdorff and non-compact. Since $K$ fixes a neighbourhood of $o$ and $\omega$ in $\fC$ we get a restricted action $\al:K\act\fC^*$ by compactly supported homeomorphisms. Hence, to prove (1) it is sufficient to show $\al:K\act\fC^*$ is CO-transitive by Theorem \ref{theo:well-transitive}. 
	
	\medskip{\bf Claim 1: The action $\al:K\act\fC^*$ is CO-transitive.}
	
	Fix a compact subset $A\subset\fC^*$ and a non-empty open subset $B\subset\fC^*$. Having common right-multiples implies that  $A\subset\cup_{\ell\in L}\Cone(s,\ell)$ for some tree $s\in\cT$ and subset $L\subset\Leaf(s)$ not containing the first or last leaf of $s$. Let $L':=\{\ell\in\Leaf(s):\min(L)\leqslant\ell\leqslant\max(L)\}$, write $A'$ for $\cup_{\ell\in L'}\Cone(s,\ell)$, and write $d$ for the cardinality of $|L'|$. Since $B$ is open we may pick a cone $B'\subset B$ with $o,\omega\notin B'$. Hence $B'=\Cone(r,j)$ for some tree $r\in\cT$ and leaf $j$ which is not the first or last leaf of $r$. We now let $f\in\cF$ be {\it any} forest composable with $r$ whose $j$th tree has $d$ leaves and set $t:=r\circ f$. This allows us to split $B'$ into $d$ subcones $B'=\cup_{k\in D}\Cone(t,k)$, where $D$ is the set of leaves of $f_j$ in $t$. We are now going to construct an element of $K$ that sends $A'$ into $B'$. We note that $L'$ and $D$ do not contain the first and last leaves of their respective trees. Hence, up to growing at the first and last leaves, having common right-multiples permits us to assume the trees $t$ and $s$ act the same at their first and last leaves (with respect to the canonical monoid action). Moreover, up to growing leaves not in $L'$ or $D$ we may assume that $L'=D$ and that the number of leaves of $t$ and $s$ is the same. With these assumptions, it makes sense to form the fraction $g:=[t,s]\in G$. Since $t$ and $s$ act the same at their first and last leaves we have $g\in K$, and since $L'=D$, we have that $g(A')=B'$, so $g(A)\subset B$.

	(2) We begin by improving Claim 1.
	
	\medskip{\bf Claim 2: The action $\al:[K,K]\act\fC^*$ is CO-transitive.}
	
	Let $A$ be compact and $B$ non-empty open. By Claim 1 pick $g\in K$ with $g(A)\subset B$. Fix $h\in K$ non-trivial. Since $\al$ is faithful, $\fC^*$ is Hausdorff, and $h$ is a homeomorphism, there is a cone $C\subset\fC^*$ with $h(C)\cap C=\varnothing$. Since $\supp(g)\cup A$ is compact, we use Claim 1 to pick $k\in K$ with $k^{-1}(\supp(g)\cup A)\subset C$. By item (1) of Proposition \ref{prop:simplicity-facts} we have $[g,h^k](A)\subset B$ and $[g,h^k]\in[K,K]$, as required.
	
	\medskip{\bf Claim 3: We have $[G,G]\cap K=[K,K]$.}
	
	The inclusion $[K,K]\subset[G,G]\cap K$ is clear. Fix $g\in[G,G]\cap K$ and a caret $Y\in\cF$ which we draw as a monochromatic caret for ease. Consider the trees
	\begin{align*}
		w=Y_{1,1}Y_{1,2}=\FSimpleA\quad\textnormal{and}\quad z=Y_{1,1}Y_{1,2}Y_{2,3}Y_{2,4}=\FSimpleB
	\end{align*}
	and the cones $W=\Cone(w,2)$ and $Z=\Cone(z,3)$. Recall $G_Z\subset G$ and $\eta_Z:G\to G_Z$ from Definition \ref{def:K-groups}. Since $\supp(g)\subset\fC^*$ is compact, by Claim 1 we pick $h\in K$ with $k:=g^h\in[G,G]\cap G_Z$. Since the canonical group action is faithful, we must have $k=\eta_Z(l)$ for some $l\in G$. Set $\gamma:=[Y_{1,1}Y_{1,2}Y_{2,3}Y_{2,4},Y_{1,1}Y_{1,2}Y_{1,3}Y_{4,4}]\in G$. We now conjugate $\eta_W(l)$ by $\gamma$ using the diagrammatic calculus. The group element $l$ is drawn in a diamond shaped cell to denote that it is a fraction of trees. We have
	\begin{align*}
		\gamma\circ\eta_W(l)\circ\gamma^{-1}=\FSimpleC=\FSimpleD=\FSimpleE=\eta_Z(l).
	\end{align*}
	Since $k=\eta_Z(l)$ is in $[G,G]$ and $[G,G]\lhd G$, the above calculation implies $\eta_W(l)\in[G,G]$. Moreover, by the definition of $Z$ and $W$, we have $k=\eta_Z(l)=\eta_W(\eta_W(l))$. It follows that  $k\in\eta_W([G,G])\subset[G_W,G_W]$, which is contained in $[K,K]$ since $o,\omega\notin W$. Finally, as $k=g^h$ where $h\in K$ and $[K,K]\lhd K$, it follows that $g\in[K,K]$, as required.
	
	We can now complete the proof of (2). If $G/K_o$ and $G/K_\omega$ are abelian, then $[G,G]\subset K$. By Claim 3 we have $[G,G]=[K,K]$. It follows by item (1) that $[G,G]$ is simple and by Claim 2 that the action $\al:[G,G]\act\fC^*$ is CO-transitive. Now fix $\Lambda\subset G$ a non-trivial subgroup normalised by $[G,G]$. Let $g,h\in[G,G]$ which we note are compactly supported in $\fC^*$. Let $\lambda\in\Lambda$ be non-trivial and let $C$ be a cone with $\lambda(C)\cap C=\varnothing$. By CO-transitivity we pick $k\in[G,G]$ with $k^{-1}(\supp(g)\cup\supp(h))\subset C$. Hence $\lambda^k(\supp(g))\cap\supp(h)=\varnothing$, so by item (2) of Proposition \ref{prop:simplicity-facts} we have $[g,h]=[[g,\lambda^k],h]$, which is in $\Lambda$ as $g,h,k\in[G,G]$. Since $[G,G]$ is simple this gives $[G,G]=[[G,G],[G,G]]\subset\Lambda$, as required.
\end{proof}

Here we witness an explicit example of an FS group having non-abelian germ groups.

\begin{example}\label{ex:germ-examples}
	In \cite{Brothier-Seelig24} we will show that when the canonical group action $\al:G\act\fC$ of an FS category $\cF$ is faithful, a presentation for the group of germs for the action $\al$ at points in $Q^+\subset\fC$ (resp. $Q^-\subset\fC$) can be read from an FS presentation of $\cF$ by pruning skein relations at the first (resp. last) leaves (see Section \ref{sec:germ-presentation} for a sketch). This has some interesting consequences:
	\begin{enumerate}
		\item {\bf An FS group with a non-abelian germ group.} With the above in mind, consider the FS category $\cF=\FS\la a,b|a_1a_1a_3=b_1b_2b_3\ra$, which we know by item (2) of Theorem \ref{theo:previous-FS-groups} is left-cancellative and has common right-multiples. In \cite{Brothier-Seelig24} we will show the corresponding canonical group action $\al:G\act\fC$ is faithful. Hence, by the above discussion
		\begin{align*}
			G/K_\omega&\simeq\Gr\la a,b|\prune(a_1a_1a_3,4)=\prune(b_1b_2b_3,4)\ra\\
			&=\Gr\la a,b|aa=bbb\ra
		\end{align*}
		and this group admits as a quotient the non-abelian group $\Gr\la a,b|a^2=b^3=e\ra\simeq\Z_2*\Z_3$. Hence $[G,G]$ is not contained in $K$ and so $[G,G]$ is {\it not} simple.
		\item {\bf A simple FS group with no piecewise linear action.} Consider the Ore FS category in the previous example and let $\pi:G\onto\Z_2*\Z_3$ be the quotient described above. We claimed that the canonical group action for $\cF$ is faithful, so by Theorem \ref{theo:simple-F}, the group $[K,K]$ is simple. Note that $\Z*\Z\into\Z_2*\Z_3$ (the derived subgroup for instance). This free group lifts via $\pi$ to a free subgroup of $G$. Since the derived subgroup of $\Z*\Z$ contains a copy of $\Z*\Z$ we have $\Z*\Z\into[G,G]$. Finally, for any cone $\Cone(z,\ell)\subset\fC^*$, conjugation by $\beta(z,\ell):\fC\to\Cone(z,\ell)$ witnesses an embedding $G\into K$. Hence $\Z*\Z\into[K,K]$. By the famous theorem of Brin and Squier \cite{Brin-Squier85}, the simple group $[K,K]$ does not admit any piecewise linear action on the real line.  
		This will be further investigated in \cite{Brothier-Seelig24}.
	\end{enumerate}
\end{example}

\subsubsection{Simplicity: $T$ and $V$-case} We now show that $G^T$ and $G^V$ have good simplicity properties when the canonical group action is faithful. No assumptions on the germ groups are necessary. The proof is similar to that above, although we also need to show that $G^T$ and $G^V$ are generated by elements of proper support \cite{Epstein70,Brin04,Matui15}.

\begin{theorem}\label{theo:simple-T-V} Let $\cF$ be an Ore FS category with FS groups $G^T$, $G^V$, canonical group action $\al:G^V\act\fC$, and canonical monoid action $\beta:\cT_p\act\fC$. If the $\al:G^V\act\fC$ is faithful, then for each $X\in\{T,V\}$, any non-trivial subgroup $\Lambda\subset G^X$ normalised by $[G^X,G^X]$ contains $[G^X,G^X]$.
\end{theorem}

\begin{proof}
Let $\cF$, $G^T$, $G^V$, $\al$, and $\beta$ be as above. Fix $X$ to be $T$ or $V$.
	
\medskip{\bf Claim 1: The action $\al:[G^X,G^X]\act\fC$ is CO-transitive.}

It is sufficient to prove it for $X=T$ since being CO-transitive is closed under taking larger groups.
For any proper compact subset $A\subset\fC$ one can construct $g\in G^T$ with $g^{-1}(A)\subset\fC^*$ by taking a fine enough partition of $\fC$ and cyclically permuting the cones. For any non-empty open $B\subset\fC$ we can find a cone $C\subset B$ with $g^{-1}(C)\subset\fC^*$. By CO-transitivity of $\al:[K,K]\act\fC^*$ (established in Theorem \ref{theo:simple-F}) we pick $k\in[K,K]$ with $k(g^{-1}(A))\subset g^{-1}(C)$. It follows that $k^g\in[G^T,G^T]$ and $k^g(A)\subset B$. It is also important for the coming arguments to note the support of $k^g$ is not too big in the following sense: we have $A\cup\supp(k^g)\not=\fC$ since $g^{-1}(A)\cup\supp(k)\not=\fC^*$.

\medskip{\bf Claim 2: For any $g\in G^X$ there is $g'\in[K,K]$ and $g''\in G^X$ of proper support with $g=g'g''$.}
	
	 Up to growing a tree pair diagram for $g\in G^X$ we may pick cones $A=\Cone(s,k)$ and $B=\Cone(t,j)$ in $\fC^*$ satisfying $g\restriction_A=\beta(t,j)\beta(s,k)^{-1}$. Applying the construction in Claim 1 of Theorem \ref{theo:simple-F} to the cones $A$ and $B$ produces an element of $K$ that acts  like $\beta(t,j)\beta(s,k)^{-1}$ on $A$. Then by Claim 2 of Theorem \ref{theo:simple-F} we may pick $g'\in[K,K]$ with
	\begin{align*}
	g'\restriction_A
	=\beta(t,j)\beta(s,k)^{-1}.
	\end{align*}
	Thus $g$ and $g'$ act the same on $A$ and so $g'':=g'^{-1}g$ acts trivially on $A$. Hence $g''$ has proper support and since $[K,K]\subset G^X$ we have $g''\in G^X$. This gives the desired decomposition.
	
	Now let us fix a non-trivial subgroup $\Lambda\subset G^X$ normalised by $[G^X,G^X]$.
	
	\medskip{\bf Claim 3: If $g,h\in G^X$ have proper support, then $[g,h]\in\Lambda$.}
	
	Let $g,h\in G^X$ have proper support. By Claim 1 we pick $k\in G^X$ with 
	\begin{align*}
	k(\supp(g))\cap\supp(h)=\varnothing\quad\textnormal{and}\quad\supp(g)\cup\supp(k)\not=\fC.
	\end{align*}
	The second condition implies $[g,k]$ has proper support. We may apply Claim 1 again to pick $l\in G^X$ satisfying
	\begin{align*}
	l(\supp(h))\cap(\supp(g)\cup\supp(k))=\varnothing.
	\end{align*}  
	Similarly, $[h,l]$ has proper support. It follows from two applications of item (2) of Proposition \ref{prop:simplicity-facts} that $[g,h]=[[g,k],[h,l]]$. Now let $\lambda\in\Lambda$ be non-trivial and $C$ a cone with $\lambda(C)\cap C=\varnothing$. By Claim 1 we pick $m\in[G^X,G^X]$ satisfying
	\begin{align*}
	m(\supp([g,k]))\cap\supp([h,l])=\varnothing\quad\textnormal{and}\quad\supp(m)\cup\supp([g,k])\not=\fC.
	\end{align*} 
	By Claim 1 again we may pick $n\in[G^X,G^X]$ satisfying
	\begin{align*}
	n^{-1}(\supp(m)\cup\supp([g,k]))\subset C.
	\end{align*}  
	It follows by item (1) of Proposition \ref{prop:simplicity-facts} that $m$ and $\gamma:=[m,\lambda^n]$ act the same on $\supp([g,k])$. Note also that $\gamma\in\Lambda$ as $m,n\in[G^X,G^X]$. Hence
	\begin{align*}
		\gamma(\supp[g,k])\cap\supp[h,l]=\varnothing
	\end{align*}
	and so by item (2) of Proposition \ref{prop:simplicity-facts} we have
	\begin{align*}
		[g,h]=[[g,k],[h,l]]=[[[g,k],\gamma],[h,l]]
	\end{align*}
	which is an element of $\Lambda$ as $[g,k],[h,l]\in[G^X,G^X]$. 
	
	We can now finish the proof. Let $g,h\in G^X$ be arbitrary. By Claim 2 we decompose $g$ as $g'g''$ and $h$ as $h'h''$ where $g',h'\in[K,K]$ and $g'',h''\in G^X$ have proper support. We have 
	\begin{align*}
		[g,h]=[g'g'',h'h'']=[g'g'',h'][g'g'',h'']^{h'}
		=[g'',h']^{g'}[g',h']([g'',h'']^{g'})^{h'}[g',h'']^{h'}.
	\end{align*}
	By Claim 3 and since $\Lambda$ is a subgroup normalised by $[K,K]\subset[G^X,G^X]$ we have $[g,h]\in\Lambda$. This completes the proof.
\end{proof}

\subsection{Simplicity implies faithfulness}\label{sec:simple-faithful}

In this section we give a converse to the results of the previous section --- existence of cetain simple groups implies the canonical group action is faithful. For simplicity for FS groups, this means the canonical group action is the ``right'' one to study. We write $\ker(\al)$ for the kernel of $\al:G^V\act\fC.$ Recall from Section \ref{sec:summary} that it is contained inside $G$ and equal to the kernel of the action $\al:G\act Q$.

\begin{lemma}
	If $L\subset G$ is a subgroup normalised by $G^T$, then $L\subset \ker(\alpha)$.
	In other words, $\ker(\al)$ is the largest subgroup of $G$ that is normalised by $G^T$.
\end{lemma}

\begin{proof}
	Consider $L$ as above and $g\in L$.
	Recall that $\al:G\act Q$ is conjugate to the homogeneous action $G\act G^T/G$.
	Let $h\in G^T$ and consider the coset $hG$.
	Observe that $$g\cdot hG = ghG = h (h^{-1}gh) G = h g' G$$
	where $g'$ is in $L$ and thus in $G$ by assumption.
	Therefore, $g\cdot hG = hg'G=hG$ and $g$ acts trivially on $G^T/G$.
\end{proof}

\begin{lemma}\label{lem:kernel-cell}
	Let $L\subset G$ be a subgroup that is normalised by $G^T$.
	For all $g\in L$, $n\geqslant 1$, and trees $t$ with $n$ leaves, there exist $h_1,\dots,h_n\in G$ satisfying
	$$g=t\circ (h_1\ot\cdots\ot h_n)\circ t^{-1}.$$
	Moreover, $h_j$ is in $\ker(\al)$ for all $1\leqslant j\leqslant n$.
\end{lemma}

\begin{proof}
	Consider $L$ as above and $g\in L$, $n\geqslant 1$, and $t\in\cT$ a tree with $n$ leaves.
	Having common right-multiples implies the existences of forests $p,q$ satisfying $g=[tp,tq]$.
	Consider now $s=t\pi^{-1} t^{-1}=[t,t\pi]\in G^T$ with $\pi$ a cyclic permutation. 
	Observe that \begin{align*}
		k & = t\circ\pi \circ p \circ (tp)^{-1} \circ g \circ (tq)\circ q^{-1}\circ \pi^{-1}\circ t^{-1}\\
		& = t\circ p^\pi \circ \pi^p \circ (tp)^{-1} \circ g \circ (tq)\circ (\pi^q)^{-1}\circ (tq^\pi)^{-1}\\
		& = t\circ p^\pi \circ \pi^p \circ (\pi^q)^{-1}\circ (tq^\pi)^{-1}.
	\end{align*}
	Note that $k:=s^{-1}gs$ is in $L$ by assumption (and thus in $G$).
	This implies that $\pi^p=\pi^q$ (see Section \ref{sec:FS-T-and-V} for the definition of $\pi^p$).
	If $\pi$ is the cyclic permutation sending $1$ to $2$, then $\pi^p$ is the cyclic permutation sending $1$ to $n_1+1$ where $n_1$ is the number of leaves of $p_1.$
	We deduce that $|\Leaf(p_1)|=|\Leaf(q_1)|$ and by an easy induction that 
	$$|\Leaf(p_j)|=|\Leaf(q_j)| \text{ for all } 1\leqslant j\leqslant n.$$
	Hence, it makes sense to consider $h_j:=[p_j,q_j]=p_j\circ q_j^{-1}$ which is an element of $G$ since $p_j,q_j$ are trees with same number of leaves.
	We deduce that 
	$$g=t\circ (h_1\ot\cdots \ot h_n)\circ t^{-1}$$
	with $h_j \in G$.
	Observe now that the restriction of $\alpha(g)$ to $\Cone(t,j)$ is equal to $\alpha(h_j)$ (up to conjugating it by $\beta(t,j):\fC\to \Cone(t,j)$). 
	Hence, $\alpha(h_j)=\id_\fC$ for all $1\leqslant j\leqslant n$.
\end{proof}

We are now able to prove the main theorem of this section showing that a simplicity assumption implies faithfulness of the canonical group action. 
The last part of the proof is diagrammatic in nature and is reminiscent of arguments present in Jones' planar algebra framework, particularly in the context of the Temperley--Lieb--Jones algebra.

\begin{theorem}\label{theo:simple-implies-faithful}
	If there exists an intermediate subgroup $[K,K]\subset \Ga\subset G^V$ that is simple, then $\al:G^V\act \fC$ is faithful.
\end{theorem}

\begin{proof}
	Fix an intermediate subgroup $[K,K]\subset \Ga\subset G^V$ that we assume to be simple.
	Recall that $\ker(\al)$ is a normal subgroup of $G^V$ that is contained in $G$.
	Moreover, note that by definition $\ker(\alpha)$ is a subgroup of $K$.
	
	\medskip{\bf Claim 1: We have that $\ker(\alpha)\cap \Ga=\{e\}.$} 
	
	Indeed, since $\Ga$ is simple either we have $\ker(\alpha)\cap \Ga=\{e\}$ or $\Ga\subset \ker(\alpha)$.
	Observe that $C_a([F,F])\subset [K,K]\subset \Ga$ where $a$ is a colour of $\cF$ and where $C_a:F\into G$ is the associated vertex-colouring embedding.
	We have that $C_a(F)\act \fC$ restricts into $F\act T/F$ (see Theorem \ref{theo:previous-FS-groups}) which is clearly faithful.
	This implies that $\Ga$ is not contained in $\ker(\alpha)$ and thus $\Ga\cap \ker(\alpha)=\{e\}$.
	
	\medskip{\bf Claim 2: We have that $\ker(\alpha)\subset Z(K)$ (where $Z(K)$ denotes the centre of $K$).}
	
	Indeed, fix $g\in \ker(\alpha)$ and $k\in K$.
	The commutator $gkg^{-1}k^{-1}$ is in $\ker(\alpha)\cap [K,K]$ since $\ker(\alpha)$ is a normal subgroup of $K$. 
	By the previous claim we have that $\ker(\alpha)\cap [K,K]=\{e\}$ since $\ker(\alpha)\cap \Ga=\{e\}$.
	We deduce that $g$ and $k$ commutes and thus $g\in Z(K)$.
	
	\medskip{\bf Claim 3: We have that $\ker(\alpha)\subset Z(G).$}
	
	Fix $g\in \ker(\alpha)$ and let $h\in G$ be arbitrary.
	Define $$\widehat g:= t\circ (e\ot g\ot e)\circ t^{-1} \text{ and } \widehat h:= t\circ (e\ot h\ot e)\circ t^{-1}$$ where $t$ is any tree with three leaves.
	Observe that $\widehat g$ is in $\ker(\alpha)$ and moreover $\widehat h$ is in $K$ since it acts trivially on $\Cone(t,1)$ and $\Cone(t,3).$
	By the previous claim $\widehat g$ and $\widehat h$ commute. By definition of $\widehat{g}$ and $\widehat{h}$ it follows that $g$ and $h$ commute. Hence $g\in Z(G)$.
	
	We can now conclude the proof. Consider $g\in \ker(\alpha)$ and define
	$$\widehat g:= s\circ (e\ot g\ot e\ot e)\circ s^{-1}$$ 
	where $s=Y_a(Y_a\ot Y_a)$ for a fixed colour $a$ of $\cF$ (in fact any tree $s$ with $4$ leaves would work). We have that $\widehat g$ is in $\ker(\al)$ and is thus inside $Z(G)$ by Claim 3.
	Define
	$$h:=[x,y] \text{ with } x:=s\circ (I\ot I\ot I\ot Y_a) \text{ and } y:=s\circ (Y_a\ot I\ot I\ot I).$$
	Observe that 
	$h\circ \widehat g \circ h^{-1} = s\circ (e\ot e\ot g\ot e)\circ s^{-1}.$ This is easy to using the diagrammatic calculus. Here the group element $e$ is identified with the trivial tree $I$, which is a vertical straight line, and $g$ is drawn as a cell of the form of a diamond. For ease, we suppress $a$-vertices. Applying universal skein relations, we have
	\begin{align*}
		h\circ\widehat{g}\circ h^{-1}=\SimpleFaithfulA=\SimpleFaithfulB=\SimpleFaithfulC=s\circ (e\ot e\ot g\ot e)\circ s^{-1}.
	\end{align*}
	Since $\widehat{g}$ is central, we deduce that $e\ot g\ot e\ot e=e\ot e\ot g\ot e$ in the fraction groupoid $\Frac(\cF)$. Hence $g=e$ and so $\ker(\al)$ is the trivial group.
\end{proof}

\begin{remark}
	In the last theorem we proved that if $\ker(\alpha)\subset Z(G)$, then necessarily $\ker(\alpha)$ is trivial.
	This does not mean that $\ker(\alpha)\cap Z(G)$ is always trivial.
	Indeed, consider the following example:
	$$\cF:=\FS\la a,b|a_1a_1=b_1b_1,a_1b_1=b_1a_1\ra$$
	which is an Ore FS category that we already encountered in \cite{Brothier22}.
	Its fraction group $G$ is isomorphic to the restricted permutational wreath product $\oplus_{\Q_2}\Z/2\Z\rtimes F$ where $\Q_2:=\Z[1/2]/\Z$ and $F$ acts by shifting indices via the usual action $F\act \Q_2.$
	Here, $\ker(\alpha)$ corresponds to $\oplus_{\Q_2}\Z/2\Z$. Moreover, the copy of $\Z/2\Z$ indexed by the dyadic rational $0$ is nothing else than the centre of $G$.
	Hence, $\Z/2\Z\simeq Z(G)\subset\ker(\alpha).$
\end{remark}

\section{Categorical characterisation of simplicity}\label{sec:categorical-characterisation}

\subsection{The category of forest-skein categories}

We write $\Forest$ (resp. $\ForestOre$) for the collection of all FS (resp. Ore FS) categories. 
Recall from \cite{Brothier22} that a morphism $\gamma:\cF\to\Tilde\cF$ is a monoidal functor sending the object $1$ to $1$, i.e.~a map preserving the composition, the tensor product, and the number of roots and leaves. 
This endows both $\Forest$ and $\ForestOre$ with category structures. 
We use set theoretical terminology such as injective, surjective to qualify morphisms. 
If $\gamma:\cF\to\ti\cF$ is injective (resp.~surjective), then we may say that $\cF$ {\it embeds} inside $\ti\cF$ (resp.~$\ti\cF$ is a {\it quotient} of $\cF$).

\subsection{Faithfulness and proper quotients}
The next proposition shows that if we have a morphism between Ore FS categories then it induces a map that intertwines the canonical group actions. This is rather immediate by functoriality.
What is interesting and useful to observe is that the intertwinner is always {\it injective}. 

\begin{proposition}\label{prop:gamma}
	Consider two Ore FS categories $\cF,\ti\cF$, their canonical category actions $\beta:\cF\to\cQ$, $\ti\beta:\ti\cF\to\ti\cQ$, and let $\gamma:\cF\to\ti\cF$ be a morphism.
	Write $\al:G\act Q$ and $\ti\al:\ti G\act \ti Q$ for the associated canonical group actions.
	Denote by $$\gamma_G:G\to\ti G, \ [t,s]\mapsto [\gamma(t),\gamma(s)]$$ the group morphism $\Frac(\gamma,1).$
	
	The morphism $\gamma$ induces an \emph{injective} map:
	$$\gamma_Q:Q\to\ti Q,\ [t,j]\mapsto [\gamma(t),j]$$
	that is equivariant with respect to the canonical actions $\beta,\ti\beta$ and $\al,\ti\al$.
	Hence, if $f$ is a forest of $\cF$ with $n$ roots and $m$ leaves and if $g\in G$, then we have the commutative diagrams:
	$$\begin{tikzcd}
		Q^n \arrow[r, "\beta(f)" ] \arrow[d, "\gamma_Q^n",swap]& Q^m \arrow[d,"\gamma_Q^m"]\\
		\ti Q^n \arrow[r, "\ti\beta(\gamma(f))",swap] & \ti Q^n
	\end{tikzcd}
	\text{ and } 
	\begin{tikzcd}
		Q \arrow[r, "\alpha(g)" ] \arrow[d,"\gamma_Q",swap]& Q\arrow[d,"\gamma_Q"]\\
		\ti Q \arrow[r, "\ti\alpha(\gamma_G(g))",swap] & \ti Q
	\end{tikzcd}.
	$$
	In particular, we have that $$\ker(\alpha)=\gamma_G^{-1}(\ker(\ti\alpha)).$$
\end{proposition}

\begin{proof}
	Most of the proposition follows from abstract nonsense.
	The only non-trivial part resides in the fact that $\gamma_Q:Q\to\ti Q$ is injective (even when $\gamma:\cF\to\ti\cF$ is not injective).
	To prove it we identify $Q$ with the homogeneous space $G^T/G$ (see Section \ref{sec:canonical-action}).
	The functor $\Frac(-,1)$ provides $\gamma_G:G\to\ti G$ and considering the $T$-version we obtain another group morphism $\gamma_{G^T}:G^T\to {\ti G}^T$ which extends $\gamma_G$.
	This provides a map 
	$$\gamma_{G^T/G}:G^T/G\to {\ti G}^T/\ti G, \ gG\mapsto \gamma_{G^T}(g) \ti G.$$
	Up to the identifying $Q$ with $G^T/G$, we have $\gamma_{G^T/G}=\gamma_Q.$
	We  now prove the injectivity of $\gamma_{G^T/G}.$
	Consider $g_1,g_2\in G^T$ and assume that $\gamma_{G^T}(g_1)\ti G = \gamma_{G^T}(g_2)\ti G$.
	Since $\gamma_{G^T}$ is a group morphism we deduce that $\gamma_{G^T}(g_1^{-1}g_2)\in \ti G$ and thus $g_1^{-1}g_2\in G$ implying that $g_1G=g_2G$.
	Hence, $\gamma_{G/G^T}$ is injective.
\end{proof}

\begin{proposition}\label{prop:F-modulo-kerbeta-simple}
	Let $\cF$ be an Ore FS category, $G:=\Frac(\cF,1)$ its fraction group, and $\alpha:G\act Q$ its canonical group action.
	The group $\ti G:=G/\ker(\alpha)$ is an FS group isomorphic to $\Frac(\ti \cF,1)$ where $\ti\cF:=\cF/\ker(\beta).$
	Moreover, the canonical group action $\ti G\act \ti Q$ is faithful.
	
	More precisely, the quotient $\ti\cF:=\cF/\ker(\beta)$ has an obvious structure of FS category. 
	It is an Ore category and the canonical map $\gamma:\cF\onto\ti\cF$ defines a surjective group morphism $\gamma_G:G\onto \ti G$ where $\ti G:=\Frac(\ti\cF,1).$
	This group morphism is conjugate to the canonical quotient morphism $G\onto G/\ker(\al)$ so that
	$$\Frac(\cF/\ker(\beta),1)\simeq \Frac(\cF,1)/\ker(\al).$$
	Finally, the map $\gamma_Q:Q\to\ti Q$ of Proposition \ref{prop:gamma} is a bijection that intertwines the $G$ and $\ti G$ actions. This implies that $\ti G\act \ti Q$ is faithful.
\end{proposition}

\begin{proof}
	Denote by $\cC$ the set of all forests of $\cF$ mod out by the relation $\sim$ which is $f\sim f'$ if $\beta(f)=\beta(f')$.
	Write by $[f]$ the class of $f$ inside $\cC$.
	Note that if $f\sim f'$ then $f$ and $f'$ have same number of roots and same number of leaves.
	If $f\sim f'$ and $h\sim h'$, so that $f$ is composable with $h$, then $$\beta(f\circ h)=\beta(f)\circ \beta(h) = \beta(f')\circ \beta(h')=\beta(f'\circ h')$$ and thus $f\circ h\sim f'\circ h'$.
	Similarly, $f\ot k\sim f'\ot k'$ if $f\sim f'$ and $k\sim k'$.
	Moreover, $\beta(f)=\beta(f')$ is equivalent to have that $\beta(f_i,j)=\beta(f'_i,j)$ where $f_i$ is the $i$th tree of $f$ and $j$ is a leaf.
	This implies that the FS structure of $\cF$ passes through the quotient $\cC$ and thus the quotient set $\cC:=\cF/\ker(\beta)$ has a natural structure of FS category.
	Note that if $\cF$ admits the FS presentation $\la S|R\ra$, then $\cC$ admits the FS presentation $\la S| \ti R\ra$ where $\ti R$ is the set of pairs of trees $(t,s)$ satisfying $\beta(t)=\beta(s).$
	
	We now show that $\cC$ is an Ore category: it has common right-multiples and is left-cancellative.
	Since $\cF$ has common right-multiples, then so does any of its quotients, such as $\cC$.
	Let us show that $\cC$ is left-cancellative.
	Using tree decompositions of forests it is sufficient to show that if $t$ is a {\it tree} (rather than a {\it forest}) and $h,h'$ are forests, then $[t\circ h]=[t\circ h']$ implies $[h]=[h']$.
	Having $[t\circ h]=[t\circ h']$ is equivalent in having $g\in \ker(\al)$ where $g=[t\circ h,t\circ h'].$
	Let $n$ be the number of leaves of $t$ and write $h_i,h_i'$ for the $i$th tree of $h,h'$, respectively, with $1\leqslant i\leqslant n.$
	Lemma \ref{lem:kernel-cell} implies that $h_i$ and $h_i'$ have same number of leaves and moreover $[h_i,h_i']\in \ker(\al).$
	This implies that $h_i\sim h_i'$ for all $1\leqslant i\leqslant n$ and thus $h\sim h'.$
	Hence, $\cC$ is left-cancellative and thus is an Ore category.
	
	The category $\cC$ admits a calculus of fractions and we can define the fraction group $\ti G:=\Frac(\cC,1)$.
	Consider the quotient map $\gamma:\cF\to \cF/\ker(\beta)$. 
	Put $\ti\cF:=\cC$ and adopt the notation of the previous proposition.
	We have a commutative diagram given by $$\ti \beta\circ\gamma=\ov\gamma\circ \beta.$$
	Since $\ti\cF$ is precisely the quotient of $\cF$ by $\ker(\beta)$ we deduce that $\ti\beta$ is necessarily injective.
	Therefore, the canonical group action of the group $\ti G$ is injective, i.e. $\ker(\ti\al)=\{e\}$.
	
	By applying the functor $\Frac(-,1)$ to $\gamma:\cF\onto\ti\cF$ we obtain a group morphism $\gamma_G:G\to \ti G$ where $\ti G:=\Frac(\cF,1)$.
	Moreover, $\gamma_G$ is surjective because $\gamma$ is. 
	Hence, $\gamma_G$ factors through an isomorphism $G/\ker(\gamma_G)\to \ti G$.
	By the previous proposition we have that $\ker(\alpha)=\gamma_G^{-1}(\ker(\ti\alpha))$. Since $\ti\alpha:\ti G\act \ti Q$ has been proven to be faithful we deduce that $\ker(\alpha)=\ker(\gamma_G)$ and thus $\ti G$ is isomorphic to $G/\ker(\alpha)$.
\end{proof}

We deduce the main result of this section. 

\begin{theorem}\label{theo:Ore-simple}
	The canonical group action $\alpha:G\act Q$ is faithful if and only if $\cF$ has no proper quotient in the category of Ore FS categories.
\end{theorem}

\begin{proof}
	Let $\cF$ be an Ore FS category with fraction group $G$ and canonical actions $\beta:\cF\to\cQ$ and $\alpha:G\act Q$.
	Assume that $\cF$ has no proper quotient in the category of Ore FS categories. We have proven that $\cF/\ker(\beta)$ is an Ore FS category. Hence, the quotient map $\cF\onto\cF/\ker(\beta)$ must be an isomorphism implying that $\beta$ is injective and thus the canonical group action $\al:G\act Q$ is faithful.
	
	Conversely, assume that there exists an Ore FS category $\ti\cF$ and a surjective morphism $\gamma:\cF\onto\ti\cF$ so that $\ker(\gamma)$ is non-trivial. 
	As above we write $\gamma_G:G\to\ti G$ for the induced group morphism from $G:=\Frac(\cF,1)$ to $\ti G:=\Frac(\ti\cF,1)$.
	By Proposition \ref{prop:gamma} we have that $\ker(\alpha)=\gamma_G^{-1}(\ker(\ti\alpha))$ which in particular contains the non-trivial kernel of $\gamma_G$. Hence, $\alpha:G\act Q$ is not faithful.
\end{proof}

\begin{example}\label{ex:non-faithful-action} Here are some explicit Ore FS categories whose canonical group action is not faithful.
	\begin{enumerate}
		\item {\bf Shape-preserving FS categories.} Any {\it shape-preserving} Ore FS category (one with all skein relations having the same underlying monochromatic tree, for the formal definition see \cite[Section 2]{Brothier23c}) admits a non-faithful canonical group action. Indeed, fix a shape-preserving Ore FS category $\cF=\FS\la S|R\ra$ so each skein relation is of the form $(t,c)\sim(t,\Tilde{c})$ where $c,\Tilde{c}:\Ver(t)\to S$ are colour functions for the monochromatic tree $t$. Let $\cU$ be the Ore FS category of monochromatic binary forests. Consider the surjective morphism $\gamma:\FS\la S\ra\onto\cU$ given by $Y_s\mapsto Y$ for all $s\in S$. Since
		\begin{align*}
			\gamma(t,c)=t=\gamma(t,\Tilde{c})
		\end{align*}
		we have that $\gamma$ factors through a surjective morphism $\gamma:\cF\onto\cU$. Hence, $\cF$ has a quotient in $\ForestOre$, and thus the corresponding canonical group action cannot be faithful.
		A large family of shape preserving Ore FS categories are given by the FS presentations $\la a,b|t(a)=t(b)\ra$ where $t$ is a monochromatic tree. 
		\item {\bf Quasi-shape-preserving FS categories.} Shape preserving FS categories are not the only type of Ore FS category admitting Ore FS quotients. Consider the Ore FS category $\FS\la a,b|a_1a_1a_3a_3=b_1b_2b_3b_4\ra$. Clearly this is not shape preserving, however, note that $a_1a_1a_3a_3$ is a quasi-right-vine of length $2$ with cell $a_1a_1$ and $b_1b_2b_3b_4$ is a quasi-right-vine of length $2$ with cell $b_1b_2$ (see Section \ref{sec:trees-and-forests} for terminology). Hence we call this category \emph{quasi-shape-preserving}. 
		Observe that 
		$\Tilde{\cF}=\FS\la a,b|a_1a_1=b_1b_2\ra$ 
		is a non-trivial quotient of $\cF.$ Hence, the canonical group action of $\cF$ is not faithful.
	\end{enumerate}
\end{example}

\begin{remark}\label{rem:quasi-shape}
	The above example easily generalises: any Ore FS category $\FS\la a,b|u(a)=v(b)\ra$, where $u$ and $v$ are quasi-trees with the same skeleton, admits as an Ore FS quotient the category $\FS\la a,b| \ti u(a)=\ti v(b)\ra$, where $\ti u,\ti v$ are the cells of the quasi-trees $u,v$, respectively. We do not know if this is the {\it only} way an Ore FS category of the form $\FS\la a,b|x(a)=y(b)\ra$ admits an Ore FS quotient, though we highly suspect that it is.
\end{remark}

\subsection{A dictionary between kernels of groups and categories}

\medskip{\bf Kernels of $\cF$ give normal subgroups of $G$.}
Consider a morphism $\theta:\cF\to\ti\cF$ between two Ore FS categories.
Recall that the kernel of $\theta$ is the following set of pairs:
$$\ker(\theta):=\{(f,f')\in\cF\times\cF:\ \theta(f)=\theta(f')\}.$$
The range of $\theta$ is isomorphic to the quotient $\cF/\ker(\theta)$. 
This is again an Ore FS category. Hence, we may assume for simplicity that $\theta$ is surjective so that $\cF/\ker(\theta)\simeq \ti\cF.$

Here are key properties of $\ker(\theta)$.

\begin{enumerate}
	\item $\ker(\theta)$ is the graph of an equivalence relation;
	\item if $f,f'$ are forests then $(f,f')\in\ker(\theta)$ if and only if $(f_i,f_i')\in \ker(\theta)$ for all $i$ where $f_i$ is the $i$th tree of $f$. In particular, $f,f'$ have same number of roots and leaves.
	\item if $(f,f')\in\ker(\theta)$ and $(h,h')\in \ker(\theta)$ so that $f$ is composable with $h$, then $(f\circ h,f'\circ h')\in \ker(\theta)$;
	\item if $(t\circ p, t\circ q)\in\ker(\theta)$, then $(p,q)\in\ker(\theta).$
\end{enumerate}

Indeed, $\theta$ being a map gives (1),  $\theta$ being monoidal gives (2), $\theta$ being a functor gives (3), and finally since the target of $\theta$ is left-cancellative we obtain (4). 
Conversely, consider now a subset $R_\cF\subset \cF\times\cF$ satisfying (1,2,3,4).
Take now $R_\cT$ the set of all pairs of trees $(t,s)$ that are in $R_\cF$. They generate all $R_\cF$ by taking composition and tensor product. Hence, taking the free FS category $\cUF$ and its quotient $\ti\cF$ given by all the relation $(t,s)$ of $R_\cT$ defines an FS category which as a set is equal to the quotient $\cUF/R_\cF.$
Since $R_\cT$ contains the relations of $\cF$ and $\cF$ has common right-multiples we deduce that $\ti\cF$ has common right-multiples. 
Now, condition (4) gives that $\ti\cF$ is left-cancellative and thus $\ti\cF$ is an Ore FS category.
We deduce the following.

\begin{proposition}
	Consider an Ore FS category $\cF$ and let $R_\cF$ be a subset of $\cF\times\cF$. 
	There exists a surjective morphsim $\theta:\cF\onto\ti\cF$ between Ore FS categories such that $R_\cF=\ker(\theta)$ if and only if $R_\cF$ satisfies (1,2,3,4).
\end{proposition}

As above we consider $\theta:\cF\onto\ti\cF$ a surjective morphism between Ore FS categories.
The functor $\Frac(-,1)$ yields a surjective group morphism
$$\theta_G:G\onto \ti G$$ where as usual $G:=\Frac(\cF,1), \ti G:=\Frac(\ti\cF,1).$
Define now 
$$N_\theta:=\{[t,s] \in G:\ (t,s)\in \ker(\theta)\}.$$
This is a subgroup of $G$ that is normalised by the larger group $G^V.$
Moreover, $\ker(\theta_G)=N_\theta$ and thus $G/N_\theta\simeq \ti G.$

Hence, kernels of the FS category $\cF$ gives subgroups of $G$ normalised by $G^V$ and quotients of $\cF$ gives quotients of $G.$
We investigate the converse.

\medskip{\bf Partial converse: from normal subgroups of $G$ to kernels of $\cF$.}

Consider $N\subset G$ a subgroup that is normalised by $G^T$.
We now define a subset of $\cF\times \cF$.
Let $R_\cT$ be the set of pairs of trees $(t,s)$ satisfying that $[t,s]$ is in $N$.
Now, define $R_\cF$ as above which is the closure of $R_\cT$ inside $\cF$ for the tensor product and for the composition. 
Using Section \ref{sec:simple-faithful} and the discussion of above we immediately obtain the following.

\begin{proposition}
	The set $R_\cF$ satisfies (1,2,3) of the key properties of above.
	This implies in particular that $N$ is normalised by $G^V.$
	If moreover $R_\cF$ satisfies (4) we have that $G/N$ is an FS group. 
\end{proposition}

We note that in general there exist subgroups $N\subset G$ normalised by $G^T$ whose associated $R_\cF$ does not satisfy (4) as illustrated in the next example.

\begin{example}\label{ex:large-kernel}
	Consider the FS category $\cF$ with two colours $a,b$ and two relations: 
	$$a_1b_1 = b_1 a_1 \text{ and } a_1a_1=b_1b_1.$$
	The category $\cF$ is constructed from the homogeneously presented monoid $\Mon\la a,b| ab=ba, aa=bb\rangle.$
	It is an Ore category with associated FS group $G$ isomorphic to the wreath product 
	$\Z/2\Z \wr_{\Q_2} F$ with $\Q_2:=\Z[1/2]/\Z$
	by \cite{Brothier23c}.
	Let $N$ be the kernel of:
	$$\oplus_{\Q_2}\Z/ 2\Z \to \Z/2\Z, \ f\mapsto \sum_{x\in \Q_2} f(x) .$$
	It corresponds to all finitely supported maps from $\Q_2$ to $\Z/2\Z$ with a support of even order. It is indeed a subgroup of $G$ that is normalised by $G^T\simeq \Z/2\Z \wr_{\Q_2} T.$
	From $N$ we define $R_\cF$ as explained above. 
	One can show that $R_\cF$ does not satisfy (4).
	Indeed, take $t=Y_a,p=Y_a \ot Y_a , q=Y_b\ot Y_b$, and note that $[t\circ p, t\circ q]\in N$ and thus $(t\circ p, t\circ q)\in R_{\cF}$.
	However, $(p_1,q_1)=(Y_a,Y_b)\notin R_{\cF}$. 
\end{example}

\section{Abelianisation}\label{sec:abelianisation} In this section we describe the abelianisations of $T$ and $V$-type FS groups in full generality. The $F$-case is more subtle and we don't treat it here. 
Fix a presented Ore FS category $\cF=\FS\la S|R\ra$ with FS groups $G$, $G^T$, and $G^V$. 
\subsection{Generators of $G^V$}\label{sec:generators}
We follow the notations for generators of $G$ from Section 4 of \cite{Brothier22}. 
For each colour $x\in S$ and $1\leqslant j\leqslant n$ recall that $x_{j,n}$ is the $n$-rooted forest with one $x$-caret at $j$. 
\begin{center}{\bf From now on fix a colour $a\in S$.}\end{center}
Write $\rho_j$ for the $a$-coloured right-vine of length $j$ (so has $j+1$ leaves). We write
\begin{align*}
	x_j:=[\rho_{j}\circ x_{j,j+1},\rho_{j+1}]\quad\textnormal{and}\quad \hat{x}_j:=[\rho_{j-1}\circ x_{j,j},\rho_{j}].
\end{align*}
It follows from Theorem 4.5 of \cite{Brothier22} that $G$ is generated by the $a_j,b_j,\hat b_j$ with $b\in S\setminus\{a\}$ and $j\geqslant 1$.
We also introduce the following notation for the $T$ and $V$-type groups. Recall that $\Sigma_n$ is the $n$th symmetric group and $C_n$ is the $n$th cyclic group. For each $1<j<k$ let $\tau_j\in C_j$ be the unique cyclic permutation sending $1$ to $j$ and $\sigma_k\in\Sigma_k$ the transposition that swaps $1$ and $2$. We write
\begin{align*}
	\widehat{\tau}_{j}:=[\rho_{j-1}\circ\tau_j,\rho_{j-1}]\quad\textnormal{and}\quad\widehat{\sigma}_{k}:=[\rho_{k-1}\circ\sigma_{k},\rho_{k-1}].
\end{align*}
Since $\sigma_k$ fixes the last leaf of $\rho_k$ it follows that $\widehat{\sigma}_k=\widehat{\sigma}_{k+1}$ for all $k\geqslant3$. We set this common element of $G^V$ to be $\widehat{\sigma}$. 

It is standard that $\Sigma_1=C_1=\{e\}$, $\Sigma_2=C_2$, $\tau_n$ generates $C_n$ for all $n\geqslant2$, and $\{\tau_m,\sigma_m\}$ generates $\Sigma_m$ for all $m\geqslant3$.

\subsection{Colour counting and length} 

Recall that $\cF=\FS\la S|R\ra$ is an Ore FS category. Any free forest $f\in \FS\la S\ra$ defines a function $\chi_0(f)\in\N^S$ by setting $\chi_0(f)(b)$ to be the number of $b$-vertices of $f.$ It follows that $\chi_0:\FS\la S\ra\onto\N^S$ is a surjective monoidal functor valued in the free commutative monoid $\N^S$ that we call the
{\it colour counting functor}. Consider the word-length morphism
$\Sigma:\N^S\to\N$
and note that $\Sigma(\chi(f))$ is the number of carets of $f.$ The colour counting functor factors through
\begin{align*}
	\chi:\FS\la S|R\ra\onto\Mon\la S|ab=ba,\chi_0(u)=\chi_0(v):a,b\in S,(u,v)\in R\ra,
\end{align*}
noting that the target monoid is $\N^S/(\chi_0\times\chi_0)(R)=:M.$
By setting $\chi(\pi)=0$ for every permutation $\pi$ we obtain the extension $\chi:\cF^V\onto M.$
By applying the universal enveloping groupoid functor $U:\Cat\to\Groupoid$ (and using that $U$ preserves presentations) we obtain the chain of morphisms:
$$\Frac(\cF^V)\overset{\chi}{\onto} \Z^S/(\chi_0\times\chi_0)(R)\overset{\Sigma}{\onto} \Z.$$
If $g=[t\pi,s]\in G^V:=\Frac(\cF^V,1)$, then $\Sigma\circ \chi(g)=\Sigma\circ \chi(t)-\Sigma\circ\chi(s)=0$ since $t$ and $s$ have same number of carets.
It is not hard to see that $\chi(G^V)=\ker(\Sigma)$, i.e.~we have an exact sequence $0\to G^V\to \Z^S/(\chi_0\times\chi_0)(R)\to \Z\to 0.$
This means that counting the number of vertices for {\it all} colours is redundant for group elements of $G^V$. Hence, we may count all but one colour. We take $a$ to be this specific colour (the same colour we used for defining generators for $G$, $G^T$, and $G^V$ in the preceding section).
Define $M_a:=M/\la a=0\ra$ the monoid obtained from $M$ by adding the relation $a=0$. 
Moreover, we write $q:M\onto M_a$ and $q:U(M)\onto U(M_a)$ for the associated quotient maps. 
From the discussion of above it is rather immediate that $q$ restricts into an isomorphism
\begin{equation}\label{eq:q-restriction}
	q\restriction_{\ker(\Sigma)}: \ker(\Sigma)\to U(M_a).
\end{equation}

\subsection{Abelianisation theorem}
We will show that $q\circ\chi:G^Y\onto U(M_a)$ is conjugated to the abelianisation map for $Y=T,V.$ 

\medskip{\bf Some relations in $G^T$ and $G^V$.} The following are lists of relations that hold in $G$, $G^T$, and $G^V$. These are all standard and can be verified easily using the diagrammatic calculus, but they are essentially due to Cannon--Floyd--Parry \cite[Sections 3,5,6]{Cannon-Floyd-Parry96}. We will use these to compute abelianisations. We do not claim these are sufficient to present the aformentioned groups. 
For $x,y\in S$ and $1\leqslant j<k$ we have:
\begin{enumerate}[label=(F\arabic*)]
\item \label{F1} $x_ky_j=y_jx_{k+1}$.
\item \label{F2} $\hat{x}_ky_j=y_j\hat{x}_{k+1}$.
\end{enumerate}
These follow from moving carets past each other by an isotopy.
Recall that the generators are expressed using a fixed colour $a\in S$ (see above).
In particular, this implies that $\widehat a_k=e$ for all $k\geqslant 1.$
For $G^T$ we have the following. Let $x\in S$ and $k\geqslant2$. 
Then:
\begin{enumerate}[label=(T\arabic*)]
\item \label{T1} $\widehat{\tau}_{k+1}x_k=x_{k-1}\widehat{\tau}_{k+2}$.
\item \label{T2} $\widehat{\tau}_k \widehat{x}_k=x_{k-1} \widehat{\tau}_{k+1}$.
\item \label{T3} $\widehat{\tau}_ka_1=\widehat{\tau}_{k+1}^2$.
\item \label{T4} $\widehat{\tau}_{k+1}=a_{k}\widehat{\tau}_{k+2}$.
\end{enumerate}
Note that \ref{T4} is a special case of \ref{T2} with $x=a$. For $G^V$ we have:
\begin{enumerate}[label=(V\arabic*)]
\item \label{V1} $\widehat{\sigma}a_1=a_2[\widehat{\sigma}^{-1},\widehat{\tau}_4^{-1}]$.
\end{enumerate}

\begin{theorem}\label{theo:abelianisation}
Let $\cF=\FS\la S|R\ra$ be an Ore FS category with FS groups $G$, $G^T$, and $G^V$. For all $a\in S$ and $Y\in\{T,V\}$ the map
\begin{align*}
	\chi_a^Y:G^Y\to\Z^{S}/\langle (\chi_0\times\chi_0)(R), a=0\rangle, \ [t\pi,s]\mapsto\chi(t)-\chi(s)
\end{align*}
induces an isomorphism on $G^Y_{\ab}$.
\end{theorem}

\begin{proof}
Let $\cF$ be as above. Fix $a\in S$ and $Y\in\{T,V\}$ and consider the generating set described in Section \ref{sec:generators} given by the colour $a$. Let $\chi_{\ab}:G^Y_{\ab}\to\ker(\Sigma)$ be the map induced by $\chi:G^Y\to\ker(\Sigma)$.
Using Equation \ref{eq:q-restriction} it is sufficient to prove that $\chi_{\ab}$ is an isomorphism. We have already observed that $\chi$ is surjective and thus so is $\chi_{\ab}$. 
It remains to prove that $\chi_{\ab}$ is injective. 

To lighten notation we write $g\sim_X h$ to express that $g [G^X,G^X]=h [G^X,G^X]$ for $X\in\{F,T,V\}$. 
Note that $g\sim_Fh$ implies $g\sim_Th$ for $g,h\in G$ and $g'\sim_T h'$ implies $g'\sim_Vh'$ for $g',h'\in G^T$. 

To show $\chi_{\ab}$ is injective it is sufficient to witness the following:
\begin{enumerate}[label=(AB\arabic*)]
	\item \label{AB1} $a_i\sim_T 0$ for all $i\geqslant 1$;
	\item \label{AB2} $x_1\sim_T x_j\sim_T \hat x_{j}$ for all $x\in S, j\geqslant 1$;
	\item \label{AB3} $\widehat\tau_k\sim_T0$ for all $k\geqslant 2$; and
	\item \label{AB4} $\widehat \sigma\sim_V0$.
\end{enumerate}
Indeed, these exactly reflect the relations in $U(M_a)$. Relation (AB1) means we can forget the colour $a$, relation (AB2) means that counting colours is well-defined, and relations (AB3) and (AB4) mean that permutations are forgotten.

From \ref{F1} with $x=a$ we have
\begin{align}\label{eqn:abelianisation-1}
	a_{j}\sim_Fa_{j+1}    
\end{align}
for all $j\geqslant2$. Since $a_2\sim_Fa_3$ relation \ref{T1} with $k=3$ and $x=a$ implies $\widehat{\tau}_4\sim_T\widehat{\tau}_5$. With this, relation \ref{T4} with $k=3$ gives $a_3\sim_T0$, and so (\ref{eqn:abelianisation-1}) implies
\begin{align}\label{eqn:abelianisation-2}
	a_k\sim_T0    
\end{align}
for all $k\geqslant2$. This, with \ref{T4}, implies 
\begin{align}\label{eqn:abelianisation-3}
	\widehat{\tau}_{k+1}\sim_T\widehat{\tau}_{k+2}  
\end{align}
for all $k\geqslant2$. Now \ref{T1} with $k=2$ and $x=a$ gives $a_1\sim_Ta_2$ and so by (\ref{eqn:abelianisation-2}) we have 
\begin{align}\label{eqn:abelianisation-4}
	a_1\sim_T0.    
\end{align}
This, along with \ref{T3} with $k=3$, and (\ref{eqn:abelianisation-3}) with $j=2$, implies $\widehat{\tau}_3\sim_T0$. This, with (\ref{eqn:abelianisation-3}) gives 
\begin{align}
	\widehat{\tau}_k\sim_T0
\end{align}
for all $k\geqslant3$. 
Now, \ref{T3} with $k=2$, (\ref{eqn:abelianisation-3}), and (\ref{eqn:abelianisation-4}), gives $\widehat{\tau}_2\sim_T0$. 
Relation \ref{V1} along with (\ref{eqn:abelianisation-2}) and (\ref{eqn:abelianisation-4}) gives $\widehat{\sigma}\sim_V0$.
We have proven \ref{AB1}, \ref{AB3}, and \ref{AB4}.
Fix a colour $x\in S$. 
Relations \ref{F1} and \ref{F2} imply that $x_k\sim_F x_{2}$ and $\widehat{x}_k\sim_F \widehat{x}_2$ for all $k\geqslant 2.$ Moreover, \ref{AB3} and \ref{T2} imply that $x_1\sim_T \widehat{x}_2$ and $x_2\sim_T\widehat{x}_3$. It remains to show that $\widehat{x}_1\sim_T x_1.$
For this, we will need to use the existence of common right-multiples for $\cF.$
Indeed, $\widehat x_1$ is by definition the fraction of trees $[Y_x,Y_a]$.
There exist forests $p,q\in\cF$ satisfying that $Y_x\circ p=Y_a\circ q.$
We deduce that $$\widehat x_1 = [Y_x\circ p,Y_a\circ p]=[Y_a\circ q,Y_a\circ p].$$ 
It follows that $[Y_a\circ q,Y_a\circ p]$ decomposes as a product of the generators 
$y_j,\widehat z_k$ with $y,z\in S$, $z\not=a$, and $1\leqslant j<k$. Hence, we have written $\widehat x_1$ in terms of generators different from $\widehat y_1$ for $y\in S.$
Since $b_1\sim b_k\sim \widehat b_k$ for all $b\in S$ and $k\geqslant 2$ (as we just proved) we deduce that 
$$\widehat x_1= [Y_a\circ q,Y_a\circ p]\sim_T \prod_{b\neq a} b_1^{n_b} $$
where $n_b$ denotes the number of $b$-vertices in $q$ minus that for $p$.
Now, applying the colour counting map to both group elements above and evaluating at $b\not=a$ gives
\begin{align*}
	n_b=\chi(\hat{x}_1)(b)=
	\begin{cases}
		1&\textnormal{if $b=x$,}\\
		0&\textnormal{if $b\not=x$.}
	\end{cases}
\end{align*}
Hence $\widehat x_1\sim_T x_1$, which completes the proof.
\end{proof}

\begin{example}\label{ex:abelianisations}
We determine abelianisations for $T$ and $V$-type FS groups in two important cases.
\begin{enumerate}
	\item {\bf Monochromatic family.} Let $S$ be a colour set and $\tau=(t_s:s\in S)$ be a family of trees all with $n$ carets. Let $\cF_\tau=\FS\la S|R\ra$ be the corresponding Ore FS category as explained in Theorem \ref{theo:previous-FS-groups}. Write $G_\tau^T$ and $G_\tau^V$ for the corresponding FS groups and fix a colour $a\in S$. It follows that
	\begin{align*}
		(G_\tau^T)_{\ab}\simeq (G_\tau^V)_{\ab}\simeq \Z^S/\la ns=0, a=0: s\in S\ra\simeq(\Z/n\Z)^{S\setminus\{a\}}.
	\end{align*}
	In particular, if $S$ is finite, then the abelianisation is finite.
	\item {\bf Colour preserving.} A presented Ore FS category $\cF=\FS\la S|R\ra$, with FS groups $G^T$ and $G^V$, is called \emph{colour preserving} if for every skein relation $(r,s)\in R$ the number of $b$-vertices in $r$ is the same as that for $s$ for all $b\in S$. It follows that
	\begin{align*}
		(G^T)_{\ab}\simeq(G^V)_{\ab}\simeq\Z^{S\setminus\{a\}}.
	\end{align*}
\end{enumerate}
\end{example}

\section{Examples: Faithfulness implies simplicity}\label{sec:examples-via-dynamics}

In this section we prove Corollary \ref{maincor-Cleary}. The crux is to show elements of {\it certain} FS groups can be represented by tree pairs $(t,s)$ where $t$ and $s$ are so-called ``good'' trees. This provides a {\bf semi-normal form} for group elements similar to that of Thompson's group and the Cleary irrational slope Thompson group, see \cite{Cannon-Floyd-Parry96} and \cite[Section 7]{Burillo-Nucinkis-Reeves21}. From this we will be able to quickly deduce that the canonical group action for such FS groups is faithful. 

\begin{center}
{\bf This produces a large class of finitely presented simple groups.}
\end{center}

We now give the setup for this result. We will use heavily the tree notation introduced in Section \ref{sec:notation}, for instance, we write $a_1b_2c_1$ instead of $a_{1,1}b_{2,2}c_{1,3}$ or $Y_a(Y_c\otimes Y_b)$ and $a_1^n$ instead of $a_{1,1}a_{1,2}\cdots a_{1,n}$.

\subsection{Left-vine decomposition}\label{sec:left-vine-decomp}

Recall left-vines are defined inductively by starting from a caret and gluing carets onto the first leaf (refer to (\ref{eqn:vines}) for graphical representations of some left-vines). Now fix a colour set $S$. Any coloured tree $t$ can be written as a product of left-vines by applying moves $a_kb_j\to b_ja_{k+1}$ where $a,b\in S$ and $1\leqslant j<k$ throughout $t$. For example:
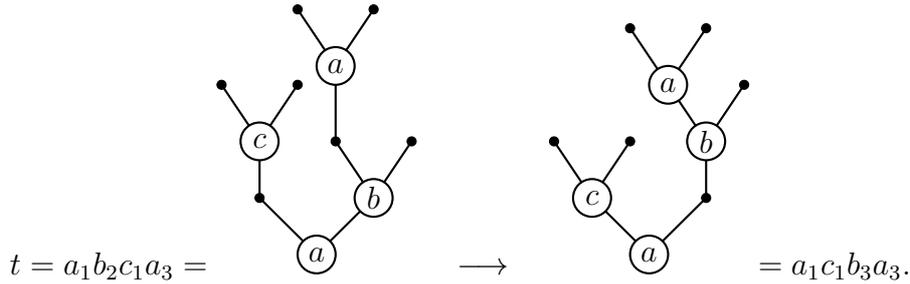
\begin{align*}
t=a_1b_2c_1a_3=\VineDecompositionA
\quad\longrightarrow\quad
\VineDecompositionB=a_1c_1b_3a_3.
\end{align*}
We write the composition of vines above as $t=\lambda^{(1)}_1\lambda^{(2)}_3$ where $\lambda^{(1)}_1=a_1c_1$ and $\lambda^{(2)}_3=b_3c_3$ are left-vines. In general we have
\begin{align}\label{eqn:vine-decomp}
t=\lambda^{(1)}_{k_1}\lambda^{(2)}_{k_2}\cdots\lambda^{(p)}_{k_p}
\end{align}
where $\lambda_{k_n}^{(n)}$ stands for a left-vine glued to the $k_n$th leaf of the tree made of the first $n-1$ left-vines. It can always be arranged that $1=k_{1}<k_{2}<\cdots<k_{p}$. For a free tree (no skein relations involved) this decomposition is unique, hence we call (\ref{eqn:vine-decomp}) the {\it left-vine decomposition} of $t$ and $p$ is called the {\it left-vine number} of $t$. When skein relations are involved, the notion of left-vine number is not well-defined. For example, if
\begin{align}\label{ex:skein-relation}
\FaithfulSkeinA
\quad\sim\quad
\FaithfulSkeinB
\end{align}
then the left and right tree represent the same tree, though the left tree has left-vine number $2$, while the right has left-vine number $1$.

\subsection{Assumptions on the skein presentation}\label{sec:faithful-assumptions} Consider a monochromatic tree $x$ of the form $x=Y(I\otimes z)$ where $z$ is an arbitrary monochromatic tree.
We also fix $y$ to be the left-vine of length $n\geqslant1$, that is
$$y=Y(Y\ot I)\cdots (Y\ot I^{\ot n-1}).$$
As usual we denote by $x(a)$ (resp. $y(b)$) the tree $x$ (resp. $y$) whose every vertex is coloured by $a$ (resp. $b$). In particular $y(b)=b_1^n$.
Write $\sim_z$ for the congruence on $(\FS\la a,b\ra,\circ,\ot)$ generated by $x(a)\sim_z y(b)$ and fix
\begin{align*}
\cF_z:=\FS\la a,b|x(a)=y(b)\ra
=\FS\la a,b\ra/\sim_z.
\end{align*}
A good example of such a skein relation to keep in mind is that in (\ref{ex:skein-relation}). Each generalised Cleary category also gives such an example (see \cite[Section 3.6.3]{Brothier22}). The category $\cF_z$ is Ore by item (3) of Theorem \ref{theo:previous-FS-groups}. We write $G_z$ for the fraction group of $\cF_z$ and $\al_z:G_z\act\fC_z$ for the canonical group action. 

\subsection{Good forests} Recall that forests in the free FS category $\FS\la a,b\ra$ are said to be {\it free}. A free left-vine $\lambda$ is called {\it good} if it is coloured from its root by a word of the form $a^pb^q$ where $0\leqslant p$ and $0\leqslant q<n$. 
A free tree $t$ is called {\it good} if all left-vines in the left-vine decomposition of $t$ are good (see Section \ref{sec:left-vine-decomp}). 
Recall that the natural number $n\geqslant 2$ is the number of vertices of the left-vine $y$.
Hence, bounding the power of $b$ by $n-1$ means that we demand that a good tree does not contain the pattern $y(b)$ but rather contains $x(a)$. To be precise, a free tree $t$ is good if it can be written as 
\begin{align}\label{eqn:good-tree}
t=a_{i_1}^{p_1}b_{i_1}^{q_1}a_{i_2}^{p_2}b_{i_2}^{q_2}\cdots a_{i_k}^{p_k}b_{i_k}^{q_k}
\end{align}
where 
\begin{enumerate}
\item $0\leqslant p_j$;
\item $0\leqslant q_j<n$;
\item $1\leqslant i_j\leqslant1+\sum_{l<j}(p_l+q_l)$; and
\item $i_1<i_2<\dots<i_k$
\end{enumerate}
for all $1\leqslant j\leqslant k$. Here the empty sum is taken to be $0$. 
Conditions (1) and (2) reflect the type of words we allow to colour the vines in $t$. Condition (3) guarantees the product (\ref{eqn:good-tree}) is a tree and condition (4) says (\ref{eqn:good-tree}) is a left-vine decomposition. 
Note that (3) implies $i_1=1.$
When we speak of a given good tree $t$ we will only write down an expression as in (\ref{eqn:good-tree}). It will be implicitly assumed that the numbers defining $t$ satisfy conditions (1) through (4) above. A free forest $f$ is called {\it good} if all of its trees are good.

\subsection{Semi-normal forms and faithfulness}

In the following theorem we will heavily use subtree and quasi-tree terminology introduced in Section \ref{sec:trees-and-forests}.

\begin{theorem}\label{theo:canonical-action-faithful} 
Let $z$ be a non-trivial monochromatic tree, $x:=Y(I\ot z)$, $y$ a left-vine, and $\cF_z:=\FS\la a,b| x(a)=y(b)\ra.$
The FS category $\cF_z$ is Ore and thus produces an FS group $G_z$ and a canonical group action $\alpha_z:G_z\act \fC_z.$
Write $\sim_z$ for the congruence relation on $(\FS\la a,b\ra,\circ,\ot)$ generated by the skein relation $x(a)\sim y(b).$
\begin{enumerate}
	\item For any free tree $t$ there exists a free $a$-forest $f$ for which $t\circ f$ is $\sim_z$-equivalent to a good tree.
	\item For any $g\in G_z$ there exists a good tree $t$ and a good $a$-tree $s$ with $g=[t,s]$.
	\item The canonical group action is faithful.
	\item We have $\Germ(G_z\act\fC_z,o)\simeq \Germ(G_z\act \fC_z,\omega)\simeq \Z.$
\end{enumerate}
\end{theorem}
\begin{proof}
Consider $z$ and $\cF_z$ as above.
The category $\cF_z$ is Ore by Theorem \ref{theo:previous-FS-groups}.
We obtain the FS group $G_z$ and the canonical group action $\alpha_z:G_z\act \fC_z.$
We now remove all subscripts $z$ for lighter notation.

(1) We begin with a claim.

\medskip{\bf Claim 1: Let $\lambda$ be a free left-vine and $k\geqslant 1$. There is an $a$-forest $f$ and a good tree $g$ with $p\geqslant k$ $a$-vertices in its first vine such that $\lambda\circ f\sim g$.}

Let $\lambda$ be a free left-vine with $N_a\geqslant0$ many $a$-vertices and $N_b\geqslant0$ many $b$-vertices and fix $k\geqslant1$. Let $f$ be the free forest composable with $\lambda$ whose $j$th tree is given by
\begin{align*}
	f_j:=
	\begin{cases}
		(x(a),1)^{k-1}\bullet x(a)&\textnormal{if $j=1$,}\\
		z(a)&\textnormal{if $j$ is right child of an $a$-vertex, and}\\
		I&\textnormal{otherwise.}
	\end{cases}
\end{align*}
Recall that $(x(a),1)^{k-1}\bullet x(a)$ is a quasi-left-vine with $k$ interior vertices and cell $x(a)$. Hence, the forest $f$ grows the right-child of every $a$-vertex $\nu$ in $\lambda$ by $z(a)$. This produces an $x(a)$ subtree rooted at $\nu$ in $\lambda\circ f$. Moreover, $f$ grows the first leaf of $\lambda$ by a quasi-left-vine whose skeleton has $k$ vertices and whose cell is $x(a)$. This is what allows us to control the number of $a$-vertices in the first vine of the resulting good tree. It follows that $\lambda\circ f$ is a composition of quasi-left-vines with cell $x(a)$ and left-$b$-vines. We now apply the skein relation $x(a)\sim y(b)$ to each of the $N_a+k$ many $x(a)$ subtrees inside $\lambda\circ f$. 
Since $y(b)$ is a left-vine $b$-vine, the resulting tree is a left-$b$-vine. Precisely, we have
\begin{align*}
	\lambda\circ f\sim b_1^{n(N_a+k)+N_b}.
\end{align*} 
Now divide $n(N_a+k)+N_b$ by $n$ so that $n(N_a+k)+N_b=np+q$ where $0\leqslant q<n$. In particular we note that $p\geqslant N_a+k\geqslant k$. By rewriting the above tree in terms of $p$ and $q$
\begin{align*}
	\lambda\circ f&\sim (b_1^n)^{p}b_1^{q}
	\sim a_1^{p}b_1^{q}\circ h
\end{align*} 
where $h$ is a free $a$-forest whose first $q+1$ trees are trivial (i.e. its non-trivial trees are attached to leaves of $a_1^p$ inside $a_1^pb_1^q$) and whose other trees are all $z(a)$. 

By definition the forest $h$ is good, and since the first tree of $h$ is trivial, all indices in $h$ are strictly greater than $1$. Hence $a^{p}_1b_1^{q}\circ h$ is a good tree whose first vine $a_1^{p}b_1^{q}$ has $p\geqslant k$ many $a$-vertices. This completes the claim.

\medskip{\bf Claim 2: For any good tree $t$ there exists an $a$-forest $f$ and a good forest $h$ for which $t\circ f\sim z(a)\circ h$.} 

Recall that $z$ is a monochromatic tree we have fixed and that $z(a)$ denotes $z$ with all vertices coloured by $a$. Fix a good tree $t$. By definition we may write $t$ as
\begin{align*}
	t=a_{i_1}^{p_1}b_{i_1}^{q_1}a_{i_2}^{p_2}b_{i_2}^{q_2}\cdots a_{i_k}^{p_k}b_{i_k}^{q_k}.
\end{align*}
Similar to as in the proof of Claim 1 we grow the first leaves of each left-vine in $t$ by an $a$-tree (a quasi-left-vine with cell $x(a)$ to be precise) so as to guarantee as many $a$-vertices in each vine as we like. Hence, up to growing $t$ and replacing by a $\sim$-equivalent tree, we may assume $p_j\geqslant|\Ver(x)|$ for all $1\leqslant j\leqslant k$, where we recall that $\Ver(x)$ is the set of interior vertices of the tree $x$. If $z(a)$ is already a rooted subtree of $t$, then we are done. Indeed, $f$ can be taken trivial and the unique forest $h$ obtained by removing $z(a)$ from $t$ is good because $t$ is a good tree.
 
If $z(a)$ is not already a rooted subtree of $t$, then there exists a vertex in $z(a)$ that is either: a $b$-vertex in $t$, or not a vertex of $t$ at all. Observe that the former case is impossible. Indeed, such a vertex $\nu$ (as a word in $\{l,r\}$, see Figure \ref{fig:infinite-binary-tree}) must have length $|\nu|$ strictly greater than some $p_j$, $1\leqslant j\leqslant k$, as this is where $b$-vertices first appear in $t$. But then
\begin{align*}
	|\Ver(z)|\leqslant|\Ver(x)|\leqslant p_j<|\nu|,
\end{align*}
which is absurd as $\nu$ is a vertex of $z$. As for the latter case, suppose $\nu$ is a vertex in $z(a)$ that is not a vertex of $t$. It follows that the path in $z(a)$ from the root $\varepsilon$ to $\nu$ must pass through a unique leaf $\ell_\nu$ of $t$. Since $|\ell_\nu|<|\nu|\leqslant|\Ver(z)|<|\Ver(x)|$ and, by assumption, all left-vines in $t$ start by at least $|\Ver(x)|$ many $a$-vertices, we have that $\ell_\nu$ is the child of an $a$-vertex in $t$. Growing the leaf $\ell_\nu$ of $t$ by a large enough $a$-coloured tree yields a good tree containing $\nu$ as an interior vertex. Doing this for all vertices of $z(a)$ that are not vertices of $t$ gives a forest $f$ such that $t\circ f$ contains $z(a)$ as a rooted subtree up to $\sim$-equivalence.

We can now complete the proof of item (1). We induct on left-vine number of $t$. If $t$ is a tree with left-vine number $1$, then $t$ is a left-vine, and by Claim 1 the result is true in this case. Assume the result holds true for trees with vine number no greater than $M\geqslant1$ and let $t$ be a tree with vine number $M+1$. We factor $t$ into $t_0\circ f_0$ where $t_0$ is the largest rooted left-vine subtree of $t$ and $f_0$ is a forest all of whose trees have vine number no greater than $M$. By definition of $t_0$ the first tree of $f_0$ is trivial. By the inductive hypothesis (applied to all non-trivial trees of $f_0$) we may pick an $a$-forest $h$ and a good forest $k$ for which $f_0\circ h\sim k$. In particular, the first tree of $k$ is trivial. Set $s:=t_0\circ k$ and let $\cA$ be the set of leaves in the left-vine $t_0$ which are right children of $a$-vertices in $t_0$. Define
\begin{align*}
	\cL:=\cA\cap\Leaf(s)\quad\textnormal{and}\quad \cM:=\cA\cap\Ver(s).
\end{align*}
Hence, elements of $\cL$ (resp. $\cM$) are vertices at which the corresponding tree in $k$ is trivial (resp. non-trivial). We now outline a growing procedure for the leaves of $s$.
\begin{enumerate}
\item (Leaves of $s$ which are descendants of $\cM$). Since $k$ is a good forest, for every $\nu\in\cM$ the tree $k_\nu$ of $k$ whose root corresponds to $\nu$ is a good tree. By Claim 2 we may grow $k_\nu$ by some $a$-forest to get, up to $\sim$-equivalence, a tree of the form $z(a)\circ l_\nu$, where $l_\nu$ is some good forest.
\item (Leaves of $s$ in $\cL$). Grow all leaves of $s$ in $\cL$ by the tree $z(a)$.
\item (All other leaves). What remains of $\Leaf(s)$ are descendants of right children of a $b$-vertices in $t_0$ and the first leaf of $s$. For such leaves we perform no growing.
\end{enumerate}
Let $\xi$ be the $a$-forest witnessing the above growings. 
It follows that $s\circ \xi\sim q\circ \eta$ where $\eta$ is a good forest and $q$ is a composition of quasi-left-vines whose cell is $x(a)$ with some left-$b$-vines. Using the skein relation $x(a)\sim y(b)$ on every $x(a)$ subtree in $q$, we deduce that $s\circ \xi\sim\lambda\circ\gamma$ where $\lambda$ is a left-$b$-vine $\lambda$ and $\gamma$ is a good forest with trivial first tree.
We now use the skein relation $y(b)\sim x(a)$ to transform the first occurrences of $y(b)$ inside $\lambda$ into $x(a)$ and leaving the remaining last (no more than $n-1$) many $b$-carets untouched.
We obtain a tree of the form $\lambda'\circ\gamma'$ where $\lambda'$ is now a good left-vine and $\gamma'$ is a good forest with trivial first tree. It follows that $\lambda'\circ\gamma'$ is a good tree and we have grown $t$ only by $a$-forests (namely $h$ and $\xi$) into a tree equivalent to $\lambda'\circ\gamma'$. This finishes the proof of the induction and the proof of (1).

(2) Recall from item (2) of Theorem \ref{theo:previous-FS-groups} that the sequence $(c_k(a))_{k\geqslant1}$ of complete $a$-trees is cofinal in $\cF$. Hence any tree in $\cF$ admits an $a$-coloured upper bound. With this in mind we fix an element $g\in G$ which we write as the fraction $g=[t_0,s_0]$ where $t_0,s_0\in\cT$. By the existence of the aforementioned cofinal sequence, we may pick a forest $f_0\in\cF$ for which $s_0f_0=c_k(a)$ for a certain $k\geqslant 1$. 
In particular, $s_0f_0$ is $a$-coloured.
By item (1) we may pick an $a$-forest $h_0$ for which $t_0f_0h_0$ is equal in $\cF$ to a good tree. Since $h_0$ is an $a$-forest and $s_0f_0$ is an $a$-tree $s_0f_0h_0$ is an $a$-tree. 
Since any $a$-tree is good by definition we have that $s_0f_0h_0$ is good. Setting $t=t_0f_0h_0$ and $s=s_0f_0h_0$ gives the result.

(3) Consider $g\in \ker(\alpha)$ and pick a good tree $t$ and a good $a$-tree $s$ for which $g=[t,s]$. Let $m$ be the common number of leaves of $t$ and $s$. We write
\begin{align*}
	t=a_{i_1}^{p_1}b_{i_1}^{q_1}a_{i_2}^{p_2}b_{i_2}^{q_2}\cdots a_{i_k}^{p_k}b_{i_k}^{q_k}\quad\textnormal{and}\quad s=
	a_{j_1}^{r_1}a_{j_2}^{r_2}\cdots a_{j_l}^{r_l}.
\end{align*}
Recall that in these decompositions we always have that $i_1=j_1=1.$
Let us assume that $g$ acts trivially on $\fC$. It follows that $\beta(t,\ell)=\beta(s,\ell)$ for all $1\leqslant\ell\leqslant m$ where $\beta:\cT_p\act\fC$ is the canonical monoid action of $\cF$. We notate the endomorphisms $\beta(Y_a,1)$ and $\beta(Y_b,1)$ by $A$ and $B$ respectively. Since $x(a)=Y_a(I\ot z(a))$ and $y(b)$ is a left-vine, we have that $(Y_a,1)$ is pruned equivalent to $(Y_b,1)^n$. Hence $A=B^n$. 
Now for $\ell=1$ we have
\begin{align*}
	B^{np_1+q_1}=A^{p_1}B^{q_1}
	=\beta(t,1)
	=\beta(s,1)
	=A^{r_1}
	=B^{nr_1}.
\end{align*}
Since $\beta(\cT_p)$ is a left-cancellative monoid and $B$ is non-trivial we have that $np_1+q_1=nr_1$. This implies $q_1$ is a multiple of $n$ but since $0\leqslant q_1<n$ we must have $q_1=0$. But then $np_1=nr_1$ and so $p_1=r_1$. Hence, the first vines in $t$ and $s$ must be equal. Therefore
\begin{align*}
	t=u \circ (a_{i_2}^{p_2}b_{i_2}^{q_2}\cdots a_{i_k}^{p_k}b_{i_k}^{q_k})\quad\textnormal{and}\quad s=u \circ (a_{j_2}^{r_2}\cdots a_{j_l}^{r_l})   
\end{align*}
where $u:=a_1^{r_1}.$
Now Lemma \ref{lem:kernel-cell} (applied to the subgroup $\ker(\alpha)\subset G$) implies that 
$$u^{-1}\circ t\circ s^{-1}\circ u=(a_{i_2}^{p_2}b_{i_2}^{q_2}\cdots a_{i_k}^{p_k}b_{i_k}^{q_k})\circ (a_{j_2}^{r_2}\cdots a_{j_l}^{r_l})^{-1}$$
is an element of the direct product $\ker(\al)^{r_1+1}$ (living inside $\Frac(\cF,r_1+1)$) where recall $r_1+1$ is the number of leaves of $u$.
This means that
$$(a_{i_2}^{p_2}b_{i_2}^{q_2}\cdots a_{i_k}^{p_k}b_{i_k}^{q_k})\circ (a_{j_2}^{r_2}\cdots a_{j_l}^{r_l})^{-1}$$ 
decomposes into a tensor product of cells in $\Frac(\cF)$
$$h_1\ot \cdots \ot h_{r_1+1}$$
where each $h_i$ is a group element that is necessarily in $\ker(\al).$
Moreover, we note that each $h_i$ is of the form $[t_i,s_i]$ where $t_i,s_i$ decompose as a product of left-vines of $t,s$, respectively.
In particular, each $t_i$ (resp.~$s_i$) is composed of strictly less left-vines than $t$ (resp.~$s$). 
Hence, we may continue to decompose each $h_i$ as we did for $g$ and eventually obtain trivial group elements.
We then deduce that $t\sim s$ and thus $g$ is the trivial group element.

(4) By item (3) the canonical group action of $\cF$ is faithful. 
Write $m$ for the number of vertices on the right side of the tree $z$ and recall that $y$ is a left-vine with $n$ vertices.
By Section \ref{sec:germ-presentation} we obtain that $\Germ(G\act\fC,o)\simeq\Gr\la a,b|a=b^n\ra\simeq\Z$ and $\Germ(G\act\fC,\omega)\simeq\Gr\la a,b|a^{m}=b\ra\simeq\Z$. 
\end{proof}

\begin{proof}[Proof of Corollary \ref{maincor-Cleary}]This follows from Theorems \ref{theo:canonical-action-faithful}, \ref{theo:simple-T-V}, and \ref{theo:simple-F}.
\end{proof}

\section{Examples: Simplicity implies faithfulness}\label{sec:examples-from-simple-groups}

Let $\cF$ be an Ore FS category, with FS groups $G$, $G^T$, $G^V$, and normal subgroup $K\lhd G$. 
Combining Theorems \ref{theo:simple-implies-faithful}, \ref{theo:simple-F}, \ref{theo:simple-T-V}, we deduce that if there is a simple intermediate subgroup $[K,K]\subset\Gamma\subset G^V$, then 
$[K,K]$, $[G^T,G^T]$, and $[G^V,G^V]$ are all simple.
This surprisingly provides new simple groups without making any computations (such as proving faithfulness of an action).
In this section we will construct a family of Ore FS categories $(\cH_n)_{n\geqslant2}$ whose $F$-type FS groups $H_n$ are known to have simple derived subgroups by other means, but whose $T$ and $V$-type groups are new. Along the way we will arrive at new embeddings of the $n$-ary Higman--Thompson groups.

\subsection{Higman--Thompson groups as forest-skein groups} For any $n\geqslant2$ the $n$-ary {\it Higman--Thompson group} $F_n$ is the group of piecewise linear homeomorphisms of $[0,1]$ with breakpoints in $\Z[1/n]$ and whose slopes are powers of $n$ (see \cite{Higman74,Brown87}). There are also $T$ and $V$-type $n$-ary groups by allowing for both cyclic and arbitrary permutations of $n$-adic intervals in $[0,1]$. These groups are denoted $T_n$ and $V_n$. Similar to Thompson's group $F$, the $n$-ary groups arise as the fraction groups of a monoid and a category. Let $\cF_n$ (resp. $\cF_{n,\infty})$ be the category (resp. monoid) of $n$-ary forests (resp. infinite forests). By $n$-ary forest we mean a forest whose non-leaf vertices have $n$ outgoing edges. It is standard that these structures are left-cancellative and have common right-multiples. Moreover, the monoid of infinite forests has presentation
\begin{align*}
\cF_{n,\infty}\simeq\Mon\la z_j:j\geqslant1|z_kz_j=z_jz_{k+(n-1)}:1\leqslant j<k\ra
\end{align*}
which can be seen by isotoping two $n$-carets past each other. By Brown's analysis \cite{Brown87} we have an isomorphism between the fraction group at $1$ of $\cF_n$ and that of $\cF_{n,\infty}$:
\begin{align*}
F_n\simeq\Frac(\cF_n,1)
\simeq\Frac(\cF_{n,\infty})
=:F_{n,\infty}.
\end{align*}
We henceforth identify all the groups above. 
It was shown by the first author in Section 3.6.2 of \cite{Brothier22} that the ternary Higman--Thompson group $F_3$ arises as the FS group of the {\it binary} Ore FS category
\begin{align}\label{eqn_ternary_category}
\cH_3:=\FS\la a,b|a_1b_2=b_1a_1\ra.
\end{align}
In this section we show that the $n$-ary Higman--Thompson group $F_n$ can be realised as the fraction group at $1$ of a binary Ore FS category for all $n\geqslant2$. 

\medskip{\bf Higman--Thompson forest-skein categories.} 
For every $n\geqslant2$ let
\begin{align*}
S_{n}&:=\{r\in\N:1\leqslant r<n\},\\
R_{n}&:=\{(Y_{r}(I\otimes Y_{s}), Y_{s}(Y_{r}\otimes I)):1\leqslant r<s<n\},
\end{align*}
and let $\cH_{n}:=\FS\la S_{n}|R_{n}\ra$. Hence, $S_n$ is the totally ordered set with $n-1$ elements and given any two $r<s$ inside $S_n$ we have one skein relation.
For convenience, we write $b$ for the last colour $n-1$ of $S_n.$
Observe that $\cH_2$ is the free FS category on one colour, so has FS group $F$, and $\cH_3$ is the category appearing in (\ref{eqn_ternary_category}), so has FS group $F_3$. In general we have the following result (see Section \ref{sec:CGP} and Section 3.4 of \cite{Brothier22} for the CGP).

\begin{proposition}\label{prop:Higman-Thompson-is-FS} Let $n\geqslant2$. The map
\begin{align*}
	\varphi_{n}:\cH_{n,\infty}\to\cF_{n,\infty}, \ r_j\mapsto z_{(j-1)(n-1)+r}, \ j\geqslant1, \ 1\leqslant r<n, 
\end{align*}
is a monoid isomorphism. Hence:
\begin{enumerate}
	\item $\cH_{n,\infty}$ is an Ore monoid and $\cH_n$ is an Ore category;
	\item The map $\varphi_n$ induces a group isomorphism 
	\begin{align*}
		\varphi_n:\Frac(\cH_{n,\infty})\to\Frac(\cF_{n,\infty}), \ x\circ y^{-1}\mapsto\varphi(x)\circ\varphi(y)^{-1}\textit{; \ and}
	\end{align*}
	\item The group $\Frac(\cH_{n},1)$ has the (right-)CGP at its last colour $b$. Hence, the map
	\begin{align*}
		\gamma:\cH_{n,\infty}&\to\Frac(\cH_{n},1),\\
		f\ot I^{\ot\infty}&\mapsto\rho_r\circ(f\otimes I)\circ\rho_{l}^{-1},
	\end{align*}
	where $\rho_k$ is the $b$-coloured right-vine with $k+1$ leaves and $f$ has $r$ roots and $l$ leaves,
	induces an isomorphism of fraction groups.
\end{enumerate}
\end{proposition}
\begin{proof} Given the FS presentation for $\cH_n$ Proposition 1.7 of \cite{Brothier22} implies the corresponding FS monoid $\cH_{n,\infty}$ admits a presentation with one generator $r_j$ for each colour $r\in S_n$ and $j\geqslant1$ (corresponding to the elementary forest $I^{\ot j-1}\ot Y_r\ot I^{\ot \infty}$), a relation
\begin{align}\label{eqn:Higman-Thompson}
	r_ks_j=s_jr_{k+1}
\end{align}
for every $r,s\in S_n$ and $1\leqslant j<k$, which correspond to moving carets past each other,
and a relation
\begin{align}\label{eqn:Higman-Skein}
	r_js_{j+1}=s_jr_j
\end{align}
for each $1\leqslant r<s<n$ and $j\geqslant1$, which correspond to skein relations. For each $r\in S_n$ and $x\geqslant1$ set
\begin{align*}
	\varphi_n(r_x):=z_{(x-1)(n-1)+r}
\end{align*}
and for any $j\geqslant1$ divide this by $n-1$ so as to write $j=(n-1)(x-1)+r$ for some $x\geqslant1$ and $1\leqslant r<n$ and set
\begin{align*}
	\psi_n(z_j)=r_x.    
\end{align*}
It is easy to see $\varphi_n$ and $\psi_n$ are mutually inverse. We now show they give a bijection between relator sets. Let $1\leqslant j<k$ and divide both these numbers by $n-1$, so that
\begin{align*}
	j=(n-1)(x-1)+r\quad\textnormal{and}\quad k=(n-1)(y-1)+s    
\end{align*}
where $1\leqslant x\leqslant y$ and $1\leqslant r,s<n$. There are two cases to consider. First suppose $x<y$. Then the $\cF_{n,\infty}$ relation corresponds to the skein relations (\ref{eqn:Higman-Thompson})
\begin{align*}
	(\psi_n(z_k)\psi_n(z_j),\psi_n(z_j)\psi_n(z_{k+(n-1)}))=(s_yr_x,r_xs_{y+1}).
\end{align*}
On the other hand if $x=y$, we have
\begin{align*}
	(\psi_n(z_k)\psi_n(z_j),\psi_n(z_j)\psi_n(z_{k+(n-1)}))=(s_xr_x,r_xs_{x+1}). 
\end{align*}
which correspond the relations (\ref{eqn:Higman-Thompson}). So $\varphi_n$ and $\psi_n$ are bijections between monoid presentations, so $\varphi_n$ extends to a monoid isomorphism.

(1) Since $\cF_{n,\infty}$ is an Ore monoid (this can be witnessed by combinatorics of $n$-ary forests) it follows that the isomorphic monoid $\cH_{n,\infty}$ is an Ore monoid also. It follows by Observation \ref{obs:Ore-category} that $\cH_n$ is an Ore category. We henceforth set
\begin{align*}
	H_n:=\Frac(\cH_n,1)\quad\textnormal{and}\quad H_{n,\infty}:=\Frac(\cH_{n,\infty}).
\end{align*}

(2) This follows from functoriality.

(3) Recall we notate the last colour $n-1\in S_n$ by the letter $b$.
In particular, for all other colours $r\in S_n$ we have the skein relation $Y_r(I\otimes Y_{b})\sim Y_{b}(Y_r\otimes I)$.
Recall the (right-)CGP at $b$ asserts that any $g\in H_n$ can be written as $g=[t,s]$ where the right sides of $t$ and $s$ are $b$-coloured (see Section \ref{sec:FS-groups}). 
This is what we now show. Fix $g=[t,s]\in H_n$. Up to growing the last leaf of $t$ and $s$ we may assume the last leaf of both $t$ and $s$ is the right child of a $b$-caret. At the first vertex (from leaf to root) in $t$ coloured by a non-$b$ caret, say $r$, we perform the skein relation $Y_r(I\otimes Y_{b})\sim Y_{b}(Y_r\otimes I)$. Now the right side of $t$ has one less non-$b$ caret. Repeat this until the right side of $t$ is made entirely of $b$-carets. Doing this also to the tree $s$ shows that $H_n$ has the CGP at the colour $b$. 
That the CGP permits the isomorphism described above is discussed in detail in Section 3.4 of \cite{Brothier22}. 
\end{proof}

\begin{remark}
\begin{enumerate}
	\item Since $\cH_n$ is an Ore FS category it admits $T$ and $V$-type Ore categories. Surprisingly, the fraction groups of these categories are {\it not} isomorphic to $T_n$ or $V_n$. Indeed Brown showed that the abelianisation of $T_n$ is trivial while that of $V_n$ is always finite (see Section 4D of \cite{Brown87}). However, in Example \ref{ex:abelianisations} we witnessed that
	\begin{align*}
		(H_n^T)_{\ab}\simeq\Z^{n-2} \quad \textnormal{and} \quad (H_n^V)_{\ab}\simeq\Z^{n-2}
	\end{align*}
	which are both infinite as long as $n\geqslant3$. The $n=2$ case corresponds to Thompson's classical groups $F,T,V$.
	\item Observe that the groupoid $\Frac(\cH^X_n)$ is connected for each $X=F,T,V$ (elementary forests give morphisms $k+1\to k$ for all $k\in\N$). This implies $\Frac(\cH^X_n,r)\simeq\Frac(\cH^X_n,1)$ for all $r\in\N$. This gives another difference between this class of FS groups and the Higman--Thompson groups: for the $T$ and $V$-type Higman--Thompson groups there is a dependence on $r$ (see \cite{Brown87,Pardo11}), but in the class of $T$ and $V$ ``Higman--Thompson'' FS groups, we only produce one isomorphism class per $n\geqslant2$.
\end{enumerate}
\end{remark}

\subsection{Simplicity and embeddings}

We now deduce some consequences of the FS realisation of the $n$-ary Higman--Thompson groups. First, we produce new simple groups. Indeed, recall the following result of Brown.

\begin{theorem}[{\cite[Theorem 4.13]{Brown87}}]\label{theo:Brown}
For any $n\geqslant2$ the group $[F_n,F_n]$ is simple.
\end{theorem}

Hence, by Theorem \ref{theo:simple-implies-faithful}, \ref{theo:Ore-simple}, and \ref{theo:simple-T-V} we have that
\begin{itemize}
\item $\cH_n$ admits no proper Ore FS quotients; and 
\item $[H_n^T,H_n^T]$ and $[H_n^V,H_n^V]$ are simple.
\end{itemize}

Hence, we have deduced Corollary \ref{maincor-Higman}. 

In a future article \cite{Brothier-Seelig24} we will use reconstruction theorems due to McCleary and Rubin \cite{McCleary78,Rubin89} to show these simple groups are different from any of the simple Higman--Thompson groups, as long as $n\geqslant3$. Since $(H_n^T)_{\ab}$ and $(H_n^V)_{\ab}$ are infinite when $n\geqslant3$, the topological finiteness properties for $[H^T_n,H^T_n]$ and $[H^V_n,H^V_n]$ are not obvious. It would be good to know if these groups are even finitely presented, though we do not address these questions here. 

Since the Ore FS categories $\cH_n$ admit no proper Ore FS quotients, we produce several embeddings between the corresponding FS groups. Let $n\leqslant m$. For any order preserving injection of $S_n$ into $S_m$ we obtain an injection of $R_n$ into $R_m$. This provides an embedding of the free FS categories $\FS\la S_n\ra\into\FS\la S_m\ra$ which factors through a morphism of Ore FS categories $\theta:\cH_n\to\cH_m$. Since $\cH_n$ has no proper Ore FS quotients we must have that $\theta$ is injective. We deduce group embeddings $H_n^X\into H_m^X$ for $X=F,T,V.$ By inspection, we observe that these embeddings between the $n$-ary Higman--Thompson groups $F_n$ are not the standard ones \cite{Brin-Guzman98,Burillo-Cleary-Stein01}.

\newcommand{\etalchar}[1]{$^{#1}$}

\end{document}